\documentclass{amsart}
\usepackage{amssymb}
\usepackage{tikz}
\usepackage{color}
\usepackage[mathscr]{eucal}
\usepackage{hyperref}
\usepackage{wasysym}
\usepackage{tikz}

\newtheorem{theorem}{Theorem}[section]
\newtheorem{lemma}[theorem]{Lemma}
\newtheorem{prop}[theorem]{Proposition}
\newtheorem{proposition}[theorem]{Proposition}
\newtheorem{corollary}[theorem]{Corollary}
\theoremstyle{definition}

\newtheorem{defn}[theorem]{Definition}
\newtheorem{definition}[theorem]{Definition}
\newtheorem{example}[theorem]{Example}

\newtheorem{remark}[theorem]{Remark}

\newtheorem{conven}[theorem]{Convention}

\newtheorem{convention}[theorem]{Assumption}
\newtheorem{proposal}[theorem]{Proposal}
\newtheorem{expl}[theorem]{Explanation}

\numberwithin{equation}{section}

\def\ggg{\mathfrak{g}}

\def\gl{\mathfrak{gl}}

\def\cb{\mathcal{B}}
\def\calb{\mathcal{B}}

\def\ggg{\mathfrak{g}}

\def\hhh{\mathfrak{h}}

\def\bbb{\mathfrak{b}}
\def\nnn{\mathfrak{n}}

\def\bbc{\mathbb{C}}

\def\bbz{\mathbb{Z}}

\def\bk{\mathbf{k}}

\def\bo{{\bar 1}}
\def\bz{{\bar 0}}
\def\ev{{\text{ev}}}

\def\sfd{\textsf{d}}

\def\whw{\widehat W}

\def\cM{\mathcal{M}}

\def\sG{\textsf{G}}
\def\sX{\textsf{X}}
\def\str{\textsf{str}}
\def\tr{\textsf{tr}}
\def\wsc{\widehat{\mathscr{C}}}
\def\ssp{\textsf{S}_{\widetilde\Pi}}

\def\spo{\textsf{S}_{\Pi_0}}
\def\scrw{\mathscr{W}}

\def\vep{\varepsilon}

\begin{document}
\title[Super Weyl groups]{Super Weyl groups, defining sequences and Coxeter graphs}
\author{Changjie Cheng}
\address{School of Mathematical Sciences,   East China Normal University, No. 500 Dongchuan Rd., Shanghai 200241,   P.R.China}
\email{cjcheng@math.ecnu.edu.cn}

\author{Yi-Yang Li}
\address{School of Mathematics, Physics and Statistics, Shanghai University of Engineering Science,
Shanghai 201620, China}\email{yiyangli1979@outlook.com}

\author{Bin Shu}
\address{School of Mathematical Sciences, Ministry of Education Key Laboratory of Mathematics and Engineering Applications \& Shanghai Key Laboratory of PMMP,  East China Normal University, No. 500 Dongchuan Rd., Shanghai 200241, China} \email{bshu@math.ecnu.edu.cn}

\thanks{This work is partially supported by the National Natural Science Foundation of China (12071136 and
12271345), and by Science and Technology Commission of Shanghai Municipality (No. 22DZ2229014).}

\begin{abstract}
The super Weyl group of a basic classical Lie superalgebra
was introduced and studied in \cite{PS}, which turns out to play an important role for the study of representations of  the basic classical Lie superalgebras and  algebraic supergroups (see \cite{PS, LS}).  These groups turn out to be some quotients of Coxeter groups. It is deserved to specially investigate super Weyl groups via revealing the related Coxeter systems.

 The purpose of this paper is twofold. { One is to describe the Coxeter systems for  super Weyl groups of basic classical Lie superalgebras. The other one is to introduce defining sequences which are a kind of new descriptions of fundamental root systems for classical Lie superalgebras of type $A,B,C$ and $D$.}
 Based on defining  sequences, we decide the Coxeter groups associated with those super Weyl groups via Coxeter graphs.
\end{abstract}

\maketitle

\setcounter{tocdepth}{1}\tableofcontents	
\setcounter{section}{-1}

\section{Introduction}\label{intro}

\subsection{}
As is well-known, for a semisimple Lie algebra $\textsf{g}$ over the complex number field, all Cartan subalgebras of $\textsf{g}$ are conjugate and all of their Borel subalgebras are conjugate too. Consequently, the machinery of Weyl groups becomes powerful to study Lie structure and representations for complex reductive groups and their Lie algebras and their modular counterparts (see \cite{Hum}, \cite{Jan},  \cite{Var}, {\sl etc}).

However,  the story is quite different in the super case.  As the super counterparts of complex semisimple Lie algebras,  basic classical Lie superalgebras are the most important classes of Lie superalgebras in the classification of complex finite-dimensional simple Lie superalgebras (see \cite{K}). By definition, those Lie superalgebras $\ggg$ admit the purely-even parts $\ggg_\bz$ which are reductive Lie algebras. Their purely-even parts $\ggg_\bz$ act on their purely-odd parts $\ggg_\bo$ such that $\ggg_\bo$ become completely reducible modules of $\ggg_\bz$. Those Lie superalgebras admit  non-degenerate invariant bilinear forms.  Different from the situation of ordinary Lie algebras,  Borel subalgebras of a basic classical Lie superalgebra are {not necessarily} conjugate.  The Weyl group of  a basic classical Lie superalgebra $\ggg=\ggg_\bz+\ggg_\bo$, which is by definition the Weyl group of the even part $\ggg_\bz$, is not sufficiently powerful to  encode intrinsic information of $\ggg$ and its representations.

In order to remedy this deficiency, {several different solutions have been proposed}.  What is the most important one  is to consider the so-called odd reflections  (see for \cite{Ser}, or \cite[\S1.3]{CW}).
Nevertheless, there was still lack of a {``unifying"} group playing  the role of the ordinary Weyl group in the situation of ordinary complex semisimple Lie algebras.

 In order to discover the relation  between roots and the change of the Borel subalgebras containing a fixed Cartan subalgebra, in \cite{PS} the authors  introduced   the notion of super Weyl group  $\widehat W$ for establishing certain Jantzen filtrations and Jantzen sum formulas in their study of  modular representations of basic classical Lie superalgebras.  The introduction of super Weyl group  $\widehat W$  is due to the following observation on a basic classical  Lie superalgebra $\ggg$.

\begin{itemize}

\item[(i)] Cartan subalgebras of $\ggg$ are actually Cartan subalgebras of $\ggg_\bz$;
\item[(ii)] All Cartan subalgebras of $\ggg$  are conjugate;

\item[(iii)] { Fix a Cartan subalgebra $\hhh$ of $\ggg$,  and denote by $\calb$ the set of Borel subalgebras of $\ggg$ containing $\hhh$.} Then $\calb$ is a finite set;
 \item[(iv)] All ordinary reflections (i.e. elements of the Weyl group  $\mathscr{W}(\ggg_\bz)$) and odd reflections can be {regarded as reflection transformations of $\calb$, i.e. invertible transformations} of order 2 (see \cite[\S3.2]{PS}).
\end{itemize}
Consequently,  a super Weyl group $\widehat W$ of $\ggg$ can be naturally introduced, which is by definition a subgroup of the  transformation group of  $\calb$ generated by ordinary reflections and odd reflections. Due to (iii), it is of course a finite group.
Naturally, all Borel subalgebras containing $\hhh$ are conjugate under $\widehat W$-action.
The group $\widehat W$ is independent of the choice of Cartan subalgebras,  up to isomorphisms, because of (i) and (ii).
{Such a point of view}  is actually valid to define a Weyl group for $\ggg_\bz$ which coincides with the ordinary abstract Weyl group $\scrw:=\mathscr{W}(\ggg_\bz)$ of $\ggg_\bz$.  This yields that $\scrw$ is a subgroup of $\widehat W$. So such a definition of super Weyl groups is compatible with that for ordinary complex semisimple Lie algebras.

{We also note that Sergeev and Veselov introduced in \cite{SV}  the notion of super Weyl
groupoid, associated to any generalized root system in Serganovas sense. The notion of
super Weyl group considered in the present paper has no direct relation with their notion.}

\subsection{} Apart from providing an approach to developing some modular theory involving Jantzen filtrations (see \cite{PS, LS}), the introduction of super Weyl groups enables us to obtain some new observations which {were well known in the ordinary case, but possibly not in the super case.}  For example, Theorem 3.10 in \cite{PS} shows that  there exists a distinguished element $\widehat w_0=\widehat r_\ell \cdots \widehat r_1\in \widehat W$ with $\widehat r_i=\widehat r_{\theta_i}$ for a series of simple positive roots $\theta_i, i=1,\cdots,\ell$, such that all positive roots exactly  correspond to the standard reduced expressions of $\widehat w_0$, modulo counting the non-isotropic odd roots which are half of positive even roots. This result is parallel to the property of the longest element of the ordinary Weyl group.

It is hopefully expected that the super Weyl groups become a  practicable tool in further study of representations of basic classical Lie superalgebras. It is necessary to understand more on such groups.

Note that as transformations of $\mathcal{B}$, both ordinary reflections or odd reflections have order $2$. This is to say, all of them are ``real" reflections. By definition, a super Weyl group $\whw$ is a finite group generated by involutions. Naturally, $\whw$ must be a quotient of some proper Coxeter group ${\mathcal{C}}$ associated with the Coxeter system $(\mathcal{C}, S)$
such that $S$ exactly corresponds to the canonical generator set
$\widehat r_{\theta}$  of $\whw$ with $\theta \in\widetilde\Pi$, and the entries $m_{xy}$ of the  matrix of the Coxeter system $({\mathcal{C}}, S)$ are identical to the corresponding ones in
$\whw$ (see \S\ref{sec: cox sys}). Such a Coxeter group $\mathcal{C}$ is unique: it is called the Coxeter group associated with $\whw$.

So it's a natural way to study super Weyl groups via investigation of the corresponding  Coxeter groups.

\subsection{Purposes and schemes}\label{sec: preword special cases}
The purpose of this paper is twofold. One is to introduce defining sequences of  fundamental root systems of classical Lie superalgebras of $A,B,C,D$. This is a new kind of parameters of fundamental root systems.  The other one is to present Coxeter graphs of the Coxeter groups for super Weyl groups of basic classical Lie superalgebras. The fulfilment of the second one  depends on the  first one.

Since the ordinary Weyl group $\scrw$ of $\ggg_\bz$  is a subgroup of $\whw$, the Coxeter system of $\scrw$ is a  Coxeter subsystem of $\whw$. As the Coxeter graph of $\scrw$ is well known,  we only need to complete the graph of $\whw$ by adding the weighted edges arising from the odd reflections appearing in the Coxeter system of $\whw$.  For this, we have to  compute all the orders $m_{xy}$ for $x,y$ in the Coxeter system  with $x$ being an odd reflection, and $y$ any other reflection (see \S\ref{sec: cox sys}).

The task of  computation mentioned above is obviously nontrivial. The scheme of the computation is divided into two steps. The first step  is to  parameterize all fundamental systems corresponding to the Borel subalgebras containing a fixed Cartan subalgebra, which will be sequences of some integers associated to roots.
The second one is to investigate the powers of $xy$ by computing the action of $xy$ on all sequences corresponding to the fundamental systems.  For example,  we can
easily fulfil the first step in type $A$, by naturally parametrizing all fundamental systems. However, it is quite challenging to fulfil the first step in  types
$B$, $C$ and $D$. We have to elaborately  make some design for parametrizing the fundamental systems {with some special parameters which are sequences of $I(m|n)$ (see \S\ref{sec: fun sys} for the notation), called the defining sequences}.  This design is the so-called Defining-Sequence Theorem ({see Theorem \ref{thm: general for Def-Seq}})
which shows that the defining sequences and the fundamental systems are in  a one-to-one correspondence and such a correspondence is $\widehat W$-equivariant ({see Theorem \ref{thm: general for Def-Seq}(3) and} Remark \ref{rem: super W-equivariant}).  Defining-Sequence Theorem  provides  a new classification of fundamental systems which seems to be more effective and
essential {than other ones before} (cf. the classification by $\varepsilon\delta$-sequences in \cite[\S1.3]{CW} for symplectic-orthogonal Lie superalgebras,
{{which is up to conjugation under the Weyl group action}}).

 \subsection{Main results}\label{sec: main results forward}  We {{successfully achieve our goals}} by completing the scheme mentioned above for types $A,B,C$ and $D$, and obtaining all Coxeter graphs.

 \subsubsection{} We first establish the following Defining-Sequence Theorem

 {
\begin{theorem}\label{thm: general for Def-Seq}
 Let $\ggg$ is a classical Lie superalgebra of type $A(m|n)$, $B(m|n)$, $C(m)$ or $D(m|n)$.  Then the following statements hold.
 \begin{itemize}
 \item[(1)] There is a bijection $\textsf{d}$ from the set $\digamma$ of fundamental systems to the set  $\daleth$ of defining sequences.

 \item[(2)] Both $\digamma$ and $\daleth$ are endowed with transitive $\whw$-action, respectively.

 \item[(3)] The correspondence $\sfd$ is $\whw$-equivariant, which means that $\sfd$ is compatible with $\whw$-action.

 \end{itemize}

 \end{theorem}
The above result for the case $A(m|n)$ is natural and trivial (see \S\ref{s4.1}-\S\ref{sec: 2.6}). For the other cases, it is quite complicated and has to be established separately (see Theorems \ref{prop: presenting seq type D}, \ref{prop: presenting seq type C}, and \ref{prop: presenting seq type B}).

 Thus the set of fundamental systems is  in a one-to-one correspondence with the set of defining sequences, compatible with $\whw$-transformation.
\subsubsection{} With the above setup,  we then obtain all Coxeter graphs.  For type $A$,  the Coxeter graph are presented in Theorem \ref{Cox glmn}.}
 For type $B, C, D$, the Coxeter groups are presented in Theorems  \ref{Thm5.3},  \ref{thm: type C}, and \ref{thm: 6.5} respectively.

\subsubsection{} Unfortunately, so far we are not able to parametrize all fundamental root systems for exceptional basic classical Lie superalgebras  $D(2,1;\alpha)$, $G(3)$, and $F(4)$. Nevertheless, we can make conjectures on their Coxeter groups according to some known fundamental systems in the appendix section (see  Proposals    \ref{thm: excep d},  \ref{thm: excep f}  and \ref{thm: excep g}).\footnote{Very recently, Y.-H. Shen and Z.-Y. Tan proved these conjectures (see: Coxeter Graphs for Super Weyl Groups of Exceptional Classical Lie Superalgebras, arXiv:2603.28312[math.RT]).}

\subsection{} It is worth mentioning that the quotient relations of a super Weyl group $\whw$ as a quotient of the corresponding Coxeter group
 are generally complicated. So far we have no other examples with the quotient relations available than the following special cases where we can give the precise structure of the super Weyl groups.
 \begin{itemize}
 \item[(0.5.1)] Case $\ggg=\gl(1|2)$): the super Weyl group is isomorphic to $I_2(6)$ (see Example \ref{gl12}).
{     \item[(0.5.2)] Case $\ggg=\mathfrak{spo}(2m|1)$: the super Weyl group coincides with the ordinary Weyl group. This is because there is not any isotropic roots in this case, and all Borel subalgebras containing a fixed Cartan subalgebra are conjugate under the ordinary Weyl group.}
  \item[(0.5.3)] Case $\ggg=\mathfrak{spo}(2|2)$:  the super Weyl group is isomorphic to the dihedral group of order $12$, the same as $I_2(6)$ (see Example \ref{ex: spo22}).
  \end{itemize}

{
\subsection*{\textsc{Acknowdgement}} The authors would like to thank very much the Referee for the careful reading of the manuscript and providing lots of helpful comments.
}

\section{Preliminaries}

We recall some basic materials involving super Weyl groups for basic classical Lie superalgebras and Coxeter groups. For the simplicity of arguments, our focus is on a field of characteristic $0$.  However, the arguments in this paper {are valid also in the
case of a field of any characteristic.}

From now on, we suppose that the ground field $\bk$ is  algebraically closed of characteristic $0$.

\subsection{Basic classical Lie superalgebras} For a given basic Lie superalgebra $\ggg$ with a Cartan subalgebra $\mathfrak{h}$, denote by $\Phi$ the root system of ${\ggg}$ relative to $\mathfrak{h}$ whose simple root system $\Pi=\{\alpha_1,\cdots,\alpha_l\}$ is standard, corresponding to the standard positive diagrams in the sense of \cite[\S1.3]{CW} (\textit{cf}. \cite[Proposition 1.5]{K2}).
Let $\Phi^+$ be the corresponding positive system in $\Phi$, and put $\Phi^-:=-\Phi^+$. Let ${\ggg}=\mathfrak{n}^-\oplus\mathfrak{h}\oplus\mathfrak{n}^+$ be the corresponding triangular decomposition of ${\ggg}$. There is a canonical Borel subalgebra $\bbb:=\hhh+\nnn^+$. Furthermore, $\ggg=\ggg_{\bz}+\ggg_{\bo}$, and $\ggg_{\bz}=\hhh+\sum_{\alpha\in \Phi_0}\ggg_\alpha$ and $\ggg_\bo=\sum_{\beta\in \Phi_1}\ggg_\beta$. The set of even roots $\Phi_0$ is occasionally  denoted by $\Phi_\ev$, and the set of odd roots $\Phi_1$  by $\Phi_{\text{odd}}$ for clarification.  There exists a non-degenerate even invariant super-symmetric bilinear form $(\cdot,\cdot)$ on $\ggg$ (see for example, \cite[Theorem 1.18]{CW}), which restricts to a non-degenerate form $(\cdot,\cdot)$ on $\hhh$ and on $\hhh^*$.  Especially, $(\cdot,\cdot)$ can be defined in $\Phi$.  A root $\gamma\in \Phi$ is called isotropic if $(\alpha,\alpha)=0$ (note that  an isotropic root is necessarily an odd root). Following \cite{Mu1} we set

$$\overline \Phi_0:=\{\alpha\in \Phi_0\mid {\alpha\over 2}\notin \Phi_1\},\;\; \overline\Phi_1:=\{\beta\in\Phi_1\mid 2\beta\notin\Phi_0\}. $$

By Lemma 8.3.2 of \cite{Mu1}, $\overline\Phi_1$ is just the set of isotropic roots. Then the set of nonisotropic roots is
\begin{align} \label{nonisotropic}
\Phi_{\text{nonisotropic}}=\overline\Phi_0\cup(\Phi_1\backslash \overline\Phi_1).
\end{align}

The following lemma is fundamental: indeed, it is the basis of the very definition of odd
reflections.

\begin{lemma} (\cite[Lemma 1.30]{CW}) \label{fun lemm} Let $\ggg$ be a basic classical  Lie superalgebra and maintain the above notations. For a given isotropic simple root $\gamma$  in  $\Pi$, set { $\Phi^+_\gamma:=\{-\gamma\}\cup (\Phi^+\backslash\{\gamma\})$}. Then $\Phi^+_\gamma$ is a new positive root system whose corresponding fundamental root system $\Pi_\gamma$ is given by
$$ \Pi_\gamma=\{\alpha\in\Pi\mid (\alpha,\gamma)=0,\alpha\ne\gamma\}\cup\{\alpha+\gamma\mid\alpha\in\Pi, (\alpha,\gamma)\ne 0\}\cup\{-\gamma\}.$$
\end{lemma}

The above operation changing $\Pi$ into $\Pi_\gamma$ is called an odd reflection with respect to an isotropic root $\gamma$. Furthermore, the following result shows that any isotropic odd root $\theta$ in $\Phi^+$ gives rise to an odd reflection  associated with a suitable fundamental system containing  $\theta$.

\begin{lemma}\label{lem: isotropic root}
 (\cite[Lemma 1.29]{CW}
Let $\ggg$ be a basic Lie superalgebra. Let $\Pi$ be the fundamental system in a positive system $\Phi^+$, and let $\theta$ be an isotropic odd root in $\Phi^+$. Then there exists  {{$w\in \mathcal{W}$ such that $w(\theta)\in \Pi$, where $\mathcal{W}$ denotes the ordinary Weyl group of $\ggg_0$ as before (see \S\ref{sec: preword special cases})}}.
\end{lemma}

\subsection{Super Weyl groups} Denote by $\mathcal B$ the set of all Borel subalgebras of $\ggg$ containing $\hhh$. For a given Borel subalgebra $B\in \mathcal B$ and the corresponding positive root system $\Phi^+(B)$,  we have  $B=\hhh+\sum\limits_{\alpha\in\Phi^+(B)}\ggg_\alpha$.

For  an isotropic simple root $\gamma$ with respect to the simple root system $\Pi(B)$ of $B$,  there exists an odd reflection $r_\gamma$ (with respect
to $\gamma$) which transforms $B$ into a new Borel subalgebra $B^\gamma=\hhh+\sum\limits_{\alpha\in\Phi^+(B^\gamma)}\ggg_\alpha$ where $\Phi^+(B^\gamma)$ is the positive root system of $\Pi(B^\gamma)$ and $\Pi(B^\gamma)=\Pi(B)_\gamma=
\{\alpha\in\Pi(B)|(\alpha,\gamma)=
0,\alpha\neq\gamma\}\cup\{\alpha+\gamma|\alpha\in\Pi(B),(\alpha,\gamma)\neq0\}\cup\{-\gamma\}$ (cf.\cite[\S3.1]{PS} or \cite{CW}).

Similarly, the odd reflection $r_{-\gamma}$ transforms $B^{\gamma}$ into the Borel subalgebra $B$.  We will identify $r_\gamma$ with $r_{-\gamma}$ for such a pair $\{B, B^{\gamma}\}$. By letting  $r_{\gamma}$ fix  {any  other Borel in $\mathcal{B}$}  whenever their fundamental systems do not contain $\gamma$, nor do contain $-\gamma$, we  make  $r_\gamma$ become a transformation of $\mathcal{B}$.

As to any given ordinary reflection $s_{\theta_0}$ ($\theta_0$ is either an even root or a non-isotropic odd root. We can also identify $s_\gamma$ with $s_{2\gamma}$ if $\gamma$ is a non-isotropic odd root with $2\gamma\in \Phi_0$),  we can define an  ordinary reflection on $\Phi$ for $\theta_0\in \Phi_0$ (correspondingly $(\theta_0,\theta_0)\ne0$)), we still have the following map as usual:
\begin{align*}
s_{\theta_0}: &\Phi\rightarrow \Phi\cr
&\alpha\mapsto s_{\theta_0}(\alpha)=
\alpha-{{2(\lambda,\theta_0)}\over{(\theta_0,\theta_0)}}\theta_0.
\end{align*}
 Thus $s_{\theta_0}(B)= \hhh+\sum_{\theta\in \Phi(B)^+}\ggg_{s_{\theta_0}(\theta)}$ for any given Borel subalgebra $B=\hhh+\sum_{\theta\in \Phi(B)^+}\ggg_\theta\in \mathcal{B}$.

\begin{defn}\label{superWeyl} (\cite{PS}) The subgroup of the transformation group of $\mathcal{B}$ generated by all ordinary and all odd reflections is called {the super Weyl group of $\ggg$. We denoted it} by $\widehat{W}(\ggg)$, or $\widehat{W}$ if the context is clear.
\end{defn}

{From now on, we will make use of the unified notation $\widehat r_\theta$ denoting both an odd reflection and an ordinary (even) reflection in  $\widehat{W}$. Obviously, all ordinary reflections and  odd reflections $\widehat r_\theta$  satisfy  $\widehat r^2_\theta=\hbox{id}$.  All of them are called super-reflections.}

Due to \cite[Proposition 1.32]{CW}, $\widehat W$ shows an important property that any two fundamental systems of $\ggg$ are conjugate {under $\widehat W$; that is to say, any two Borel} subalgebras are conjugate under the action of $\widehat W$, in some natural sense.          

{
\begin{remark}\label{rem: 1.4} Heckenberger and Yamane introduced in \cite{HY} a groupoid related to basic classical Lie superalgebras motivated
by Serganova's work, and in \cite{H} the notion of the Weyl groupoid for Nichols algebras.  Their basic idea is to consider generalized Coxeter groupoid, in contrast with ours.   It is interesting to look for the relationship between their generalization of Coxeter groups and our super Weyl groups.
\end{remark}
}

\subsection{}\label{sec: fun sys} Now we recall some basic material and notations. One can also refer to \cite{CW} or \S\ref{sec: 2.1}-\S\ref{sec: bas notations}.

Set $I(m|n)=\{\bar 1,\ldots, \bar m; 1,\ldots,n\}$ with a total order $\bar 1<\cdots<\bar m<0<1<\cdots<n$.
 Keep in mind some notations $\varepsilon_i$ for $i\in I(m|n)$ (see \S\ref{sec: bas notations}).
 For convenience of expression, we will often write
 \begin{align}\label{eq: delta exp bar}
 \delta_{i}=\varepsilon_{\bar i} \text{ for  }1\leq i\leq m.
 \end{align}

 Then we can talk about fundamental (root) systems of Lie superalgebras of type $A$, $B$, $C$ and $D$.
 \subsubsection{} If $\ggg$ is of type $A(m|n)$, $B(m)=\mathfrak{spo}(2m|1)$ or $C(m)=\mathfrak{spo}(2m|2)$, then its fundamental root system is listed as below
\begin{itemize}
\item[(i)] For $A(m|n)$, $\Pi=\{\delta_1-\delta_2,\cdots,\delta_n-\varepsilon_1,\varepsilon_1-\varepsilon_2,\cdots,\varepsilon_{m-1}-\varepsilon_m\}$.
\item[(ii)] { For $B(m)$, $\Pi=\{ \delta_1-\delta_2,\ldots,\delta_{m-1}-\delta_m,\delta_m\}$.}

\item[(iii)] For $C(m)$, $\Pi=\{\varepsilon_1-\delta_1, \delta_1-\delta_2,\delta_2-\delta_3,\ldots,\delta_{m-1}-\delta_m, 2\delta_m\}$.

\end{itemize}

The above three cases  $A(m|n)$, $B(m)$ and $C(m)$ are exceptional, with respect to our arguments below.

\subsubsection{} Let $\ggg$ be a basic classical simple Lie superalgebra of type different from type $A(m|n)$, $B(m)$, and $C(m)$. The standard fundamental root system $\Pi$ of  $\ggg$ is as listed below (\textit{cf}. \cite[\S2.5]{K} or \cite[\S1.3]{CW}):
\begin{align}\label{F1.3}
&\mathfrak{spo}(2m|2n+1), n\geq1:
  & \Pi=\{\delta_1-\delta_2,\cdots,\delta_m-\varepsilon_1,\varepsilon_1-\varepsilon_2,\cdots,\varepsilon_{n-1}-\varepsilon_n,\varepsilon_n\};\cr
&\mathfrak{spo}(2m|2n), n\geq 2:
  &\Pi=\{\delta_1-\delta_2,\cdots,\delta_m-\varepsilon_1,\varepsilon_1-\varepsilon_2,\cdots,\varepsilon_{n-1}-\varepsilon_n,\varepsilon_{n-1}+\varepsilon_n\};\cr
&F(4):& \Pi=\{ {1\over 2}(\varepsilon_1+ \varepsilon_2+ \varepsilon_3+\delta), -\varepsilon_3, \varepsilon_3-\varepsilon_1, \varepsilon_1-\varepsilon_2\};\cr
&G(3):& \Pi=\{\delta-\varepsilon_1,   \varepsilon_2-\varepsilon_3,  -\varepsilon_2\};    \cr
&D(2,1,\alpha):& \Pi=\{\delta+ \varepsilon_1+ \varepsilon_2, -2\varepsilon_1, -2\varepsilon_2\}.
\end{align}
Then by the extended standard fundamental system  $\widetilde \Pi$ for $\ggg$  { we mean }what follows(cf. \cite[\S3.2]{PS}):
\begin{align}\label{eq: ext fund sys}
&\mathfrak{spo}(2m|2n+1), m\geq1:& \widetilde\Pi= \{\alpha_0:=-2\delta_1\}\cup \Pi ;\cr
&\mathfrak{spo}(2m|2n), m\geq 2: & \widetilde\Pi=\{\alpha_0:=-2\delta_1\}\cup \Pi ;\cr
&F(4):& \widetilde\Pi=\{\alpha_0:=-\delta\}\cup\Pi;\cr
&G(3):& \widetilde\Pi=\{\alpha_0:=-2\delta\}\cup\Pi;    \cr
&D(2,1,\alpha):& \widetilde\Pi=\{\alpha_0:=-2\delta\}\cup \Pi.
\end{align}
For the above non-exceptional cases, the extended  standard Dynkin diagram corresponding to $\widetilde\Pi$ is   a diagram extending the standard Dynkin diagram associated with $\Pi$, by adding a vertex $\alpha_0$ connected to the first vertex of the standard Dynkin diagram. Adding such a vertex is exactly for assurance that the extended system $\widetilde\Pi$  contains  the standard fundamental root system $\Pi_0$ of $\ggg_\bz$. Consequently, the Weyl group of $\ggg$ naturally {turns into a  Coxeter subgroup inside $\widehat W$.}
 {All $\widehat r_{\theta}$ with $\theta\in \widetilde\Pi$ will be called simple {\sl{super-reflections}}.}

\subsubsection{}\label{sec:  type B(m)} As to the cases $A(m|n)$ and $C(m)$, we set $\widetilde\Pi=\Pi$. The case $B(m)$ will be not considered in this  paper because it does not admit any isotropic odd root, and its super Weyl group coincides with the ordinary Weyl group as mentioned in (0.4.2) of \S\ref{sec: main results forward}.

\subsection{}

 We have the following observation which is critical to  our arguments on Coxeter groups.

\begin{lemma}\label{lem: critical lem} The following statements hold.
\begin{itemize}
\item[(1)] Any two fundamental systems $\Pi_1$ and $\Pi_2$ are conjugate by the super Weyl group $\whw$, i.e. there exists $w\in \whw$ such that $\Pi_2=w(\Pi_1)$.
\item[(2)] The super Weyl group $\whw$ is generated by the simple super-reflections, that is, $\whw=\langle \widehat r_{\theta}\mid \theta\in\widetilde \Pi \rangle$.
\item[(3)] Any fundamental system comes from the standard {one by means of a composition of} simple reflections $\widehat r_{\theta}$ arising from some $\theta\in \widetilde \Pi$.
\end{itemize}
\end{lemma}
\begin{proof}
(1) follows  \cite[Lemma 3.2]{PS}. (2) follows   (\cite[Lemma 3.4]{PS}).  (3) follows from the former parts.
\end{proof}

\subsection{Coxeter groups and Coxeter graphs associated with $\widehat W$}\label{sec: cox sys}
\subsubsection{}\label{sec: cox sys order}
Let us recall some material on Coxeter groups and Coxeter graphs (see \cite[Chapter 1]{GP}).
For a given finite index set $\sX$, consider a symmetric matrix $\cM=(m_{xy})_{x,y\in\sX}$  whose entries are positive integers or $\infty$, such that $m_{xx}=1$ and $m_{x,y}>1$ for all different $x,y\in \sX$. Define a group $\sG(\cM)$ via the following presentation
$$\langle x\in\sX\mid  (xy)^{m_{xy}}=1 \text{ for }  y\in\sX, m_{xy}<\infty\rangle.$$
The pair $(\cM,\sX)$ is called a Coxeter system.
\subsubsection{}  With any symmetric matrix $\mathcal{M}$ mentioned above, we associate a graph,
called the Coxeter graph. It has vertices labelled by the elements of $S$, and two vertices labelled by $x \ne y$ are joined by an edge if $m_{xy} \geq 3$ (with $m_{xy}$ as in the above definition). Moreover, if $m_{xy}\geq 4$, we label the edge by $m_{xy}$.

\subsubsection{Coxeter systems for $\ggg$ and their subsystems} The following result is {a starting point for} our arguments in this paper.
\begin{prop}\label{prop: gen system}
The super Weyl group $\widehat W$ for $\ggg$ is a quotient of the Coxeter group associated with the generator system $\textsf{S}_{\widetilde\Pi}:=\{\widehat r_{\theta}\mid \theta\in\widetilde \Pi \}$.
\end{prop}

\begin{proof}
It follows from
 Lemma \ref{lem: critical lem}.
\end{proof}

So what we will do  is to describe the Coxeter group associated with $\widetilde\Pi$, which will be denoted by $\widehat{\mathscr{C}}(\ggg)$, or simply by $\wsc$. The Coxeter system is just $(\wsc, \ssp)$.
This means that we need compute the numbers $m_{xy}$ for all different $x,y\in \widetilde\Pi$. 

Denote by $\textsf{S}_{\Pi_0}$ the set of simple reflections corresponding to the standard fundamental system $\Pi_0$ of $\ggg_\bz$. By the construction and  analysis in \S\ref{sec: fun sys},  $(\scrw, \spo)$   { is  a Coxeter subsystem} of $(\wsc,\ssp)$.

\section{Defining sequences and Coxeter graphs for type $A(m|n)$}

\subsection{}\label{sec: 2.1}
The general linear Lie superalgebra $\frak {gl}(m|n)$ consists of block matrices of size $m|n$ as

\begin{equation}\label{matix}g=\begin{pmatrix}a&b\\c&d\end{pmatrix}
\end{equation}

Let  $\hhh$ be the Cartan subalgebra of $\ggg:=\frak {gl}(m|n)$ consisting of all diagonal matrices.
We parametrize the rows and columns of the matrices by the set $I(m|n)=\{\bar 1,\ldots, \bar m; 1,\ldots,n\}$ with a total order $\bar 1<\cdots<\bar m<0<1<\cdots<n$.

\subsection{}\label{sec: bas notations} Recall that the superstrace $\str$ is a function on $\ggg$, which is defined via   $$\str(g)=\tr(a)-\tr(d)$$
for $g$ in (\ref{matix}). The special linear Lie superalgebra $\frak{sl}(m|n)$ is a subalgebra of $\frak{gl}(m|n)$ consisting of block matrices $g$ with $\str(g)=0$. Recall that $\frak{gl}(m|n)$ and $\frak{sl}(m|n)$ have the same root system, sharing the same root theory and the same (super) Weyl group theory. So we will only consider $\frak{gl}(m|n)$ for type $A$.

Let $\{E_{ii}|i\in I(m|n)\}$ be the basis of $\hhh$.
{The superstrace restricted to $\hhh$ defines a  non-degenerate symmetric  bilinear form $(-,-)$ via the formula $(g',g''):=\textsf{str}(g',g'')$.} Then for $i,j\in I(m|n)$,
\begin{align*}
	(E_{ii},E_{jj})=
	 \begin{cases}
	 1 & \text{if}  ~\bar 1\leq i=j\leq \bar m,\cr
	 -1 &\text{if}   ~1\leq i=j\leq n,\cr
     0  &\text{if}   ~i\neq j.
	 \end{cases}
		 \end{align*}

 Denote by $\{\varepsilon_{\bar i}, \varepsilon_{j}|1\leq i\leq m, 1\leq j\leq n\}$  the basis of $\hhh^*$ dual to $\{E_{\bar i \bar i}, E_{jj}|1\leq i\leq m, 1\leq j\leq n\}$. Then $\varepsilon_{\bar i}=(E_{\bar i \bar i},-)$ and $\varepsilon_{j}=-(E_{j j},-)$.

The bilinear form $(-,-)$ on $\hhh$ induces  non-degenerate symmetric  bilinear form on $\hhh^*  $ which is also denoted by $(-,-)$. Then for $i,j\in I(m|n)$,  \begin{equation*}
	(\varepsilon_{i},\varepsilon_{j})
= \begin{cases}
	 1 & \text{if}  ~\bar 1\leq i=j\leq \bar m,\cr
	 -1 &\text{if}   ~1\leq i=j\leq n,\cr
     0  &\text{if}   ~i\neq j.
	 \end{cases}
	 	 \end{equation*}

\subsection{Examples: the super Weyl group  $\widehat{W}$  and the Coxeter groups $\wsc$  for $\mathfrak{gl}(1|2)$ and $\mathfrak{gl}(1|3)$}
\subsubsection{}The super Weyl group $\whw$ of $\gl(1|2)$ was {explicitly described} in \cite[Example 3.5]{PS}. Here we make a new exposition with some different  parameters.
\begin{example}\label{gl12}

Let $\ggg=\gl(1|2)$ and $\hhh$ be the standard Cartan subalgebra which  consists
of diagonal matrices.  Then all Borels subalgebras containing $\hhh$ correspond to  the following fundamental systems
$\{ \varepsilon_{\bar{1}}-\varepsilon_1,\varepsilon_1-\varepsilon_2\}$, $\{\varepsilon_1- \varepsilon_{\bar{1}}, \varepsilon_{\bar1}-\varepsilon_2\}$, $\{\varepsilon_1-\varepsilon_2,\varepsilon_2- \varepsilon_{\bar{1}}\}$,
$\{\varepsilon_{\bar1}- \varepsilon_{2}, \varepsilon_{2}-\varepsilon_1\}$, 			$\{\varepsilon_2- \varepsilon_{\bar{1}}, \varepsilon_{\bar{1}}-\varepsilon_1\}$,
	$\{\varepsilon_2-\varepsilon_1,\varepsilon_1- \varepsilon_{\bar{1}}\}$ which we parametrize by the numbers $\{1,2,\ldots,6\}$, respectively.

Simply denote such fundamental systems by arrangement sequences of
$\{\bo, 1,2\}$:
\begin{align}\label{eq: codes for gl 12}
\{\bar{1}12, {1}\bar12, 12\bar{1}, \bar121, 2\bar{1}1, 21\bar{1}\}
\end{align}
where the sequences  correspond to the ordered arrangements of subscripts in the expressions of simple roots.
Then the  $6$ fundamental systems are exactly arising from  the total permutations of  the sequence $\bar{1}12$, i.e. all elements of the symmetric group $S_3$.
For simplicity we write odd reflections
	 $ \widehat r_{\varepsilon_{\bar{1}}-\varepsilon_1}$, $ \widehat r_{\varepsilon_{\bar{1}}-\varepsilon_2}$ as $\widehat r_{\bo 1}, \widehat r_{\bo2}$ respectively. Also the ordinary reflection $\widehat r_{\varepsilon_1-\varepsilon_2}$ is written as $\widehat r_{12}$.
Thus, we can draw the following diagram which shows the fundamental systems and the change of them  by reflections:
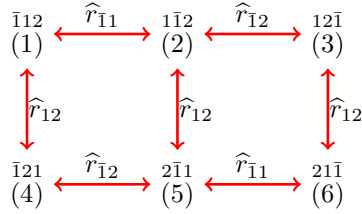
\begin{center}
	\begin{tikzpicture}
	[line width = 1pt,
	gray/.style = {circle, draw, fill = gray, minimum size = 0.3cm},
	empty/.style = {circle, draw, fill = white, minimum size = 0.3cm}]

	\node (1) at (-2,0) {$\overset{\bar{1}12}{(1)}$};
	
	\node (2) at (0,0) {$\overset{1\bar{1}2}{(2)}$};
	
	\node (3) at (2,0) {$\overset{12\bar{1}}{(3)}$};
	
	\node (4) at (-2,-2) {$\overset{\bar{1}21}{(4)}$};
	
	\node (5) at (0,-2) {$\overset{2\bar{1}1}{(5)}$};
	
	\node (6) at (2,-2) {$\overset{21\bar{1}}{(6)}$};

	\draw[<->][red] (1) --(2);
	
	\draw[<->][red] (3) --(2);

	\draw[<->][red] (4) --(5);
	
	\draw[<->][red] (6) --(5);
	
	\draw[<->][red] (1) --(4);
	
	\draw[<->][red] (2) --(5);
	
	\draw[<->][red] (3) --(6);
	
	\node at(-1.75,-1){ $\widehat r_{12}$};
	
	\node at(0.25,-1){ $\widehat r_{12}$};
	
	\node at(2.25,-1){ $\widehat r_{12}$};
	
	\node at(-1,0.25){ $\widehat r_{\bar{1}1}$};
	
	\node at(1,0.25){ $\widehat r_{\bar{1}2}$};
	
	\node at(-1,-1.75){ $\widehat r_{\bar{1}2}$ };
	
	\node at(1,-1.75){$\widehat r_{\bar{1}1}$};	
	\end{tikzpicture}
\end{center}
Then we have $\widehat r_{\bar{1}1}=(12)(56),\widehat r_{\bar{1}2}=(23)(45),\widehat r_{12}=(14)(25)(36)$.
	
	By calculation, we have $\widehat r_{\bar{1}2} = \widehat r_{12}\cdot \widehat r_{\bar{1}1}\cdot \widehat r_{12}$ and	
	 $(\widehat r_{12}\cdot \widehat r_{\bar{1}1})^6=1$.	So the super Weyl group
	$\widehat{W}(\mathfrak{gl}(1|2))$ can be generated by $\widehat r_{12}$  and $\widehat r_{\bar{1}1}$. 	
	So 	$\widehat{W}(\mathfrak{gl}(1|2))\cong I_{2}(6)$, which is the dihedral group of order 12.
	The super Weyl group $\widehat{W}(\mathfrak{gl}(1|2))$  exactly coincides with $\wsc(\mathfrak{gl}(1|2))$, {which is the
Coxeter group generated} by $\widehat r_{12}$  and $\widehat r_{\bo 1}$ with Coxeter graph:	

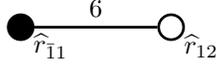
\begin{center}
\begin{tikzpicture}
		
		\node[circle,draw,minimum size=8pt,inner sep=0pt,fill=white] (r1) at (-2,1) {};
		\draw (r1.135) -- (r1.315);
		\draw(r1.225) -- (r1.45);
		
		\node[circle,draw,minimum size=8pt,inner sep=0pt,fill=white] (r2) at (0,1) {};

		\draw (r1) -- node[above=2pt] {\footnotesize 6} (r2);
		
		\node[below=3pt] at (r1) {$\widehat r_{\bar{1}1}$};
		\node[below=3pt] at (r2) {$\widehat r_{12}$};
		
\end{tikzpicture}
\end{center}
%
This Coxeter group is exactly of type $G_2$.
\end{example}

%

%
\begin{remark}\label{rem: white and black} As showed in the above example,  we will draw all Coxeter graphs for super Weyl groups associated with extended standard fundamental systems with the vertices presented as white dots when corresponding to ordinary  reflections, and presented as gray dots when corresponding to odd reflections.  {Here and further  we follow Kac's conventions,  in which the gray dot $\bigotimes$ denotes the isotropic odd simple root.  As the odd reflection is defined by an isotropic odd simple root,  we naturally indicate the odd reflections  by gray dots.}
\end{remark}
\subsubsection{} {It is worth mentioning that} the above example is one of extremely few ones where the super Weyl groups just coincides the associated  Coxeter groups. Let us  show this with a little nontrivial example.

\begin{example} Consider $\ggg=\mathfrak{gl}(1|3)$. Keep in mind that $\widetilde\Pi=\Pi=\{\varepsilon_{\bo}-\varepsilon_1,\varepsilon_1-\varepsilon_2,\varepsilon_2-\varepsilon_3\}$. In the same way as in Example \ref{gl12}, we have $\whw=\langle \widehat r_{\bo 1}, \widehat r_{12},\widehat r_{13} \rangle$.
Similarly to Example \ref{gl12}, we can draw the following diagram.
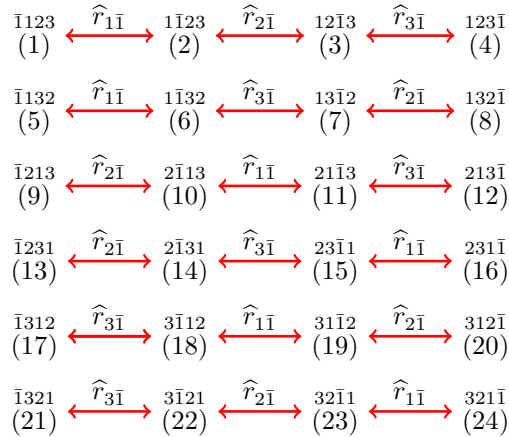
\begin{center}
	\begin{tikzpicture}
	[line width = 1pt,
	gray/.style = {circle, draw, fill = gray, minimum size = 0.3cm},
	empty/.style = {circle, draw, fill = white, minimum size = 0.3cm}]

	\node (1) at (-4,0) {$\overset{\bar{1}123}{(1)}$};
	
	\node (2) at (-2,0) {$\overset{1\bar{1}23}{(2)}$};
	
	\node (3) at (0,0) {$\overset{12\bar{1}3}{(3)}$};
	
	\node (4) at (2,0) {$\overset{123\bar{1}}{(4)}$};
	
	\node (5) at (-4,-1) {$\overset{\bar{1}132}{(5)}$};
	
	\node (6) at (-2,-1) {$\overset{1\bar{1}32}{(6)}$};
	
	\node (7) at (0,-1) {$\overset{13\bar{1}2}{(7)}$};
	
	\node (8) at (2,0-1) {$\overset{132\bar{1}}{(8)}$};
	
	\node (9) at (-4,-2) {$\overset{\bar{1}213}{(9)}$};
	
	\node (10) at (-2,-2) {$\overset{2\bar{1}13}{(10)}$};
	
	\node (11) at (0,-2) {$\overset{21\bar{1}3}{(11)}$};
	
	\node (12) at (2,-2) {$\overset{213\bar{1}}{(12)}$};
	
	\node (13) at (-4,-3) {$\overset{\bar{1}231}{(13)}$};
	
	\node (14) at (-2,-3) {$\overset{2\bar{1}31}{(14)}$};
	
	\node (15) at (0,-3) {$\overset{23\bar{1}1}{(15)}$};
	
	\node (16) at (2,-3) {$\overset{231\bar{1}}{(16)}$};
	
	\node (17) at (-4,-4) {$\overset{\bar{1}312}{(17)}$};
	
	\node (18) at (-2,-4) {$\overset{3\bar{1}12}{(18)}$};
	
	\node (19) at (0,-4) {$\overset{31\bar{1}2}{(19)}$};
	
	\node (20) at (2,-4) {$\overset{312\bar{1}}{(20)}$};
	
	\node (21) at (-4,-5) {$\overset{\bar{1}321}{(21)}$};
	
	\node (22) at (-2,-5) {$\overset{3\bar{1}21}{(22)}$};
	
	\node (23) at (0,-5) {$\overset{32\bar{1}1}{(23)}$};
	
	\node (24) at (2,-5) {$\overset{321\bar{1}}{(24)}$};

	\draw[<->][red] (1) --(2);
	
	\draw[<->][red] (2) --(3);
	
	\draw[<->][red] (3) --(4);

	\draw[<->][red] (5) --(6);
	
	\draw[<->][red] (6) --(7);
	\draw[<->][red] (7) --(8);

	\draw[<->][red] (9) --(10);
	\draw[<->][red] (10) --(11);
	
	\draw[<->][red] (11) --(12);
	
	;
	\draw[<->][red] (13) --(14);
	
	\draw[<->][red] (14) --(15);
	
	\draw[<->][red] (15) --(16);

	\draw[<->][red] (17) --(18);
	
	\draw[<->][red] (18) --(19);
	\draw[<->][red] (19) --(20);

	\draw[<->][red] (21) --(22);
	
	\draw[<->][red] (22) --(23);
	\draw[<->][red] (23) --(24);
	
	\draw[<->][red] (17) --(18);

	\node at(-3,0.25){$\widehat r_{1\bar{1}}$};
	
	\node at(-1,0.25){$\widehat r_{2\bar{1}}$};
	
	\node at(1,0.25){$\widehat r_{3\bar{1}}$};
	
	\node at(-3,-0.75){$\widehat r_{1\bar{1}}$};
	
	\node at(-1,-0.75){$\widehat r_{3\bar{1}}$};
	
	\node at(1,-0.75){$\widehat r_{2\bar{1}}$};
	
	\node at(-3,-1.75){$\widehat r_{2\bar{1}}$};
	
	\node at(-1,-1.75){$\widehat r_{1\bar{1}}$};
	
	\node at(1,-1.75){$\widehat r_{3\bar{1}}$};
	
	\node at(-3,-2.75){$\widehat r_{2\bar{1}}$};
	
	\node at(-1,-2.75){$\widehat r_{3\bar{1}}$};
	
	\node at(1,-2.75){$\widehat r_{1\bar{1}}$};

	\node at(-3,-3.75){$\widehat r_{3\bar{1}}$};
	
	\node at(-1,-3.75){$\widehat r_{1\bar{1}}$};
	
	\node at(1,-3.75){$\widehat r_{2\bar{1}}$};
	
	\node at(-3,-4.75){$\widehat r_{3\bar{1}}$};
	
	\node at(-1,-4.75){$\widehat r_{2\bar{1}}$};
	
	\node at(1,-4.75){$\widehat r_{1\bar{1}}$};

	\end{tikzpicture}
\end{center}
Denote  three generators  by
$$a=\widehat r_{\bar{1}1}=(1,2)(5,6)(10,11)(15,16)(18,19)(23,24),$$
$$b=\widehat r_{23}=(1,5)(2,6)(3,7)(4,8)(9,17)(10,18)(11,19)(12,20)(13,21)$$$$(14,22)(15,23)(16,24),$$
$$c=\widehat r_{12}=(1,9)(2,10)(3,11)(4,12)(5,13)(6,14)(7,15)(8,16)(17,21)$$$$(18,22)(19,23)(20,24),$$
respectively.
Making use of the computer software SageMath, we have
 $$\widehat{W}(\mathfrak{gl}(1|3))=\langle a,b,c\mid a^2=1, b^2=1, c^2=1, (a  b)^2=1, (c  b)^3=1, (c  a)^{12}=1,(a  c  b  a  c)^6=1,((c  a  b)^2  (c  a)^3)^4$$
$$=1,a  c  a  b  (c  a  b  c  a)^2  b  ((c  a)^3  b  c  a)^2  ((b  c  a)^2  c  a)^2  c  a  b  (c  a)^3  c  b=1,
(a  b  (c  a)^2  b  (c  a  b  (c  a)^2)^2  (c  a  b)^2  (c  a)^2  b  c)^2=1 \rangle.$$

\subsubsection{} So the super Weyl group $\widehat{W}(\mathfrak{gl}(1|3))$ is already quite complicated although {it has small size.}  Generally, it is difficult to give the presentation of a super Weyl group via generators and defining relations. It is reasonable to describe the related Coxeter groups $\wsc$.  {We turn back} to the super Weyl group for $\mathfrak{gl}(1,3)$. The corresponding Coxeter group can be easily described via generator and relations as  following: $$\wsc(\mathfrak{gl}(1|3))=\langle a,b,c\mid a^2=1, b^2=1, c^2=1, (a  b)^2=1, (c  b)^3=1, (c  a)^{12}=1 \rangle.$$
The Coxeter graph  is

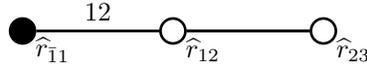
\begin{center}
	\begin{tikzpicture}
		
		\node[circle,draw,minimum size=8pt,inner sep=0pt,fill=white] (r1) at (-2,1) {};
		\draw (r1.135) -- (r1.315);
		\draw(r1.225) -- (r1.45);
		
		\node[circle,draw,minimum size=8pt,inner sep=0pt,fill=white] (r2) at (0,1) {};
		
		\node[circle,draw,minimum size=8pt,inner sep=0pt,fill=white] (r3) at (2,1) {};

		\draw (r1) -- node[above=2pt] {\footnotesize 12} (r2);
		\draw (r2) -- (r3);
		
		\node[below=3pt] at (r1) {$\widehat r_{\bar{1}1}$};
		\node[below=3pt] at (r2) {$\widehat r_{12}$};
		\node[below=3pt] at (r3) {$\widehat r_{23}$};
				
	\end{tikzpicture}
\end{center}
%
\end{example}

\subsection{Defining sequences for $A(m|n)$} \label{s4.1}
Now we turn to the general $\gl(m|n)$ or $\frak{sl}(m|n)$ with $m>0,n>0$.  Its root system
$\Phi=\Phi_{\bar{0}}\bigcup \Phi_{\bar{1}}$ is as follows (see for example \cite{CW}):
\begin{align*}
&\Phi_{\bar{0}}=\{\varepsilon_{i}-\varepsilon_{j}\mid i\neq j,\bar{1}\leq i,j\leq \overline{m}\text{~or} ~~1\leq i,j\leq n\}\cr
&\Phi_{\bar{1}}=\{\pm(\varepsilon_{i}-\varepsilon_{j})\mid \bar{1}\leq i\leq \overline{m},1\leq j\leq n\}.
\end{align*}
  Consider  the standard Borel subalgebra consisting of upper triangular matrices with the Cartan subalgebra $\hhh$ of diagonal matrices. Then the  corresponding standard positive system is
  $$\{\varepsilon_{i}-\varepsilon_{j}\mid i,j\in I(m\rvert n),i<j\}$$ with the standard fundamental system
$$\Pi=\{\varepsilon_{\bar{i}}-\varepsilon_{\overline{i+1}},
\varepsilon_{\overline{m}}-\varepsilon_{1},\varepsilon_{j}-\varepsilon_{j+1}
\mid
 1\leq i \leq m-1, 1\leq j \leq n-1\}.$$
In the same way as in Example \ref{gl12}, we assign a sequence $\bar{1}\bar{2}\bar{3}\cdots\overline{m}123\cdots n$ to present the standard fundamental system.


{{
Note that   any fundamental system  comes from the standard fundamental system by means of some permutation of $I(m|n)$. We can generally introduce the following definition.

\begin{definition}(Definition of defining sequences for $A(m|n)$)  \label{sec: general knowing for fun sys}
Given any Borel  subalgebra $B$ in $\mathcal{B}$ for the Lie superalgebra $\ggg$ of type $A(m|n)$, it  corresponds to some unique fundamental system $\Pi$, say
\begin{align}\label{eq: general knowing for fun sys}
\{\varepsilon_{i_1}-\varepsilon_{i_2}, \ldots, \varepsilon_{i_{m+n-1}}-\varepsilon_{i_{m+n}}\}
 \end{align}
 for $\{i_1, i_2,\ldots,i_{m+n}\}=I(m|n)$.
  The sequence $(i_1i_2\ldots i_{m+n})$ is then called the defining sequence of $\Pi$.

\end{definition}
}}

The total number of  fundamental systems for $\mathfrak{gl(m|n)}$ is $(m+n)!$. So we have the same number of defining sequences.

{{
\begin{remark} The notion of defining sequence is a central tool for  our arguments of this paper. We will introduce it separately in different types. For other types different from $A(m|n)$,  it will be complicated. The definition of defining sequence for type $B,C$ and $D$ will be introduced  after  the establishment of its existence and uniqueness, respectively (see Lemma \ref{lem: defining sec for D} and then Definition \ref{def: def seq D} for type $D(m|n)$, Lemma \ref{lem: lemma 2.9 for type C} and then Definition \ref{def: def seq C} for type $C(m)$, Lemma \ref{lem: lemma 2.11 for type B} and then Definition \ref{def: def seq B} for type $B(m|n)$).
\end{remark}
}}

\subsection{}\label{s4.2}
By definition (see Lemma \ref{fun lemm}),  $\widehat r_{\varepsilon_{\overline{m}} - \varepsilon_{1}}(\Pi)$ becomes $$\Pi_{\varepsilon_{\overline{m}} - \varepsilon_{1}}=\{\varepsilon_{\bar{i}}-\varepsilon_{\overline{i+1}},
\varepsilon_{\overline{m-1}}-\varepsilon_{1},
\varepsilon_{1}-\varepsilon_{\overline m}, \varepsilon_{\overline{m}}-\varepsilon_{2},
\varepsilon_{j}-\varepsilon_{j+1}\mid 1\leq i \leq m-2, 2\leq j \leq n-1\}.$$
The corresponding defining sequence of $\Pi_{\varepsilon_{\overline{m}} - \varepsilon_{1}}$  is $$(\bar{1}\bar{2}\bar{3}\cdots\overline{m-1}1\overline{m}23\cdots n).$$
Thus, we can consider the action of a  reflection   $\widehat r_{\theta}$ on the defining {sequences: in this way,} we have
$$\widehat r_{\varepsilon_{\overline{m}} - \varepsilon_{1}}(\bar{1}\bar{2}\bar{3}\cdots\overline{m}123\cdots n)=\bar{1}\bar{2}\bar{3}\cdots\overline{m-1}1\overline{m}23\cdots n.$$
  Consequently, an odd reflection plays a role of permutation for the sequences corresponding to fundamental systems. Keep this in mind. We can make computations for determination of the Coxeter graphs in the sequel.

\subsection{Coxeter graph for $\gl(m|n)$}\label{sec: 2.6}

In the following, we only need to describe $\gl(m|n)$ for $m+n\geq 4$ because the other cases have been dealt with before (note that $\gl(3|1)\cong\gl(1|3)$).  Keep the notations and conventions in \S\ref{sec: fun sys} and Remark \ref{rem: white and black}.
In a way similar to the examples before, recall that $ \delta_{i}=\varepsilon_{\bar i} \text{ for  }1\leq i\leq m$,  we simply write  $\widehat r_{\overline{i,i+1}}$, $\widehat r_{\overline m, 1}$ and $\widehat r_{j,j+1}$ for $\widehat r_{\delta_i-\delta_{i+1}}$, $\widehat r_{\delta_m-\varepsilon_1}$ and $\widehat r_{\varepsilon_j-\varepsilon_{j+1}}$ respectively, $i=1,\ldots,m-1$ and $j=1,\ldots,n-1$.

\begin{theorem}\label{Cox glmn}Let $\ggg=\mathfrak{gl}(m|n)$ with positive numbers $m,n$ satisfying $m+n\geq 4$.
{{The Coxeter graph associated with its Coxeter system $(\wsc,\ssp)$ is}}

\begin{center}
	\begin{tikzpicture}
		
		\node[circle,draw,minimum size=8pt,inner sep=0pt,fill=white] (r1) at (-4,1) {};

		\node[circle,draw,minimum size=8pt,inner sep=0pt,fill=white] (r2) at (-3,1) {};
		
		\node[circle,draw,minimum size=8pt,inner sep=0pt,fill=white] (r3) at (-2,1) {};
		
		\node[circle,draw,minimum size=8pt,inner sep=0pt,fill=white] (r4) at (-1,1) {};
		
		\node[circle,draw,minimum size=8pt,inner sep=0pt,fill=white] (r5) at (0,1) {};
		\draw (r5.135) -- (r5.315);
		\draw(r5.225) -- (r5.45);
		
		\node[circle,draw,minimum size=8pt,inner sep=0pt,fill=white] (r6) at (1,1) {};
		
		\node[circle,draw,minimum size=8pt,inner sep=0pt,fill=white] (r7) at (2,1) {};
		
		\node[circle,draw,minimum size=8pt,inner sep=0pt,fill=white] (r8) at (3,1) {};

		\draw (r1) -- node[above=2pt] {\footnotesize 4} (r2);
		\draw (r2) -- (r3);
		\node at(-1.5,1){ $\cdots$};
		\draw (r4) -- node[above=2pt] {\footnotesize 12} (r5);
		\draw (r5) -- node[above=2pt] {\footnotesize 12} (r6);
		\draw (r6) -- (r7);
		\node at(2.5,1){ $\cdots$};;

		\node[below=3pt] at (r1) {$\widehat r_{\overline{ 1 2}}$};
		\node[below=3pt] at (r2) {$\widehat r_{\overline{23}}$};
		\node[below=3pt] at (r3) {$\widehat r_{\overline{34}}$};
		\node[below=3pt] at (r4) {$\widehat r_{\overline{m-1,m}}$};
		\node[below=3pt] at (r5) {$\widehat r_{{\overline{m}1}}$};
		\node[below=3pt] at (r6) {$\widehat r_{12}$};
		\node[below=3pt] at (r7) {$\widehat r_{23}$};
		\node[below=3pt] at (r8) {$\widehat r_{n-1,n}$};
		
	\end{tikzpicture}
\end{center}

\end{theorem}

\begin{proof}
Note that the Coxeter graph of  the subsystem $(\scrw, \spo)$ is already known (see for example, \cite[Theorem 1.3.3]{GP}). Recall a general fact  in group theory that $m_{xy}=m_{yx}$ in \S\ref{sec: cox sys order}, i.e. the orders of  $xy$ and $yx$ are equal. So it suffices for us to compute $m_{xy}$ involving $x=\widehat r_{\overline{m}1}$ in \S\ref{sec: fun sys}  associated with Proposition \ref{prop: gen system}. For this, we need to consider the action of the powers of  $xy$ on  $\mathcal{B}$, i.e. on Borel subalgebras in $\mathcal{B}$, equivalently on defining sequences defined in \S\ref{s4.1}, for $x=\widehat r_{\overline{m}1}$ and all other $y\in\Pi$. All possibilities are considered case by case below.

(1) Consider the case when  $y=\widehat r_{\overline{i,i+1}}$ $(1\leq i \leq m-2)$. This is an ordinary (even) reflection. By definition, it is a linear transformation on $\hhh^*$, via exchanging the positions of the numbers $\bar{i}$ and $\overline{i+1}$. This property does not change the situation whether or not the fundamental system $\Pi(B)$ for a given $B\in \cb$ contains $\pm (\varepsilon_{\overline{m}}-\varepsilon_1)$, neither influence the positions of $\bar i$ and $\overline{i+1}$. This yields that $\widehat r_{\overline{i, i+1}}\cdot \widehat r_{\overline{m}1}= \widehat r_{\overline{m}1}\cdot \widehat r_{\overline{i,i+1}}$.  Hence  $(\widehat r_{\overline{i,i+1}}\cdot \widehat r_{\overline{m}1})^2=1$.

(2) Consider the case when $y=\widehat r_{j,j+1}$ $(2\leq j \leq n-1)$. In this case, $y$ is an ordinary (even) reflection too. With the same reasoning as above,  we still have  $(\widehat r_{j,j+1}\cdot \widehat r_{\overline{m}1})^2=1$.

(3) Consider the case when $y=\widehat r_{12}$. In this case, there are only three possible  defining sequences which have to be dealt with separately. As to the other defining sequences, $\widehat r_{\overline{m}1}\cdot \widehat r_{12}$ fixes them.

(Subcase 3.I) The case is when the defining sequences satisfy that  three numbers $1$, $2$, and $\overline{m}$ are adjacent one by one (for example, the sequences of the form like $(\cdots\overline{m}12\cdots)$).
We demonstrate the action of $\widehat r_{\overline{m}1}\cdot \widehat r_{12}$ on such sequences  which represent the related fundamental systems:
 \begin{align*}
 &(\cdots\overline{m}12\cdots)\overset
{\widehat r_{12}}{\longrightarrow}(\cdots\overline{m}21\cdots)\overset
{\widehat r_{\overline{m}1}}{\longrightarrow}\cr
&(\cdots\overline{m}21\cdots)
\overset
{\widehat r_{12}}{\longrightarrow}(\cdots\overline{m}12\cdots)\overset
{\widehat r_{\overline{m}1}}{\longrightarrow}\cr
&(\cdots1\overline{m}2\cdots)
\overset
{\widehat r_{12}}{\longrightarrow} (\cdots
2\overline{m}1\cdots)\overset
{\widehat r_{\overline{m}1}}{\longrightarrow}\cr
&(\cdots21\overline{m}
\cdots)\overset
{\widehat r_{12}}{\longrightarrow}(\cdots12\overline{m}\cdots)\overset
{\widehat r_{\overline{m}1}}{\longrightarrow}\cr
&(\cdots12\overline{m}\cdots)\overset
{\widehat r_{12}}{\longrightarrow}(\cdots21\overline{m}\cdots)\overset
{\widehat r_{\overline{m}1}}{\longrightarrow}\cr
&(\cdots2\overline{m}1\cdots)\overset
{\widehat r_{12}}{\longrightarrow}(\cdots1\overline{m}2\cdots)\overset
{\widehat r_{\overline{m}1}}{\longrightarrow}\cr
&(\cdots\overline{m}12\cdots).
\end{align*}
So for this case, the order of the action $\widehat r_{\overline{m}1}\cdot \widehat r_{12}$ on such sequences is $6$.

There are  other  possibilities of similar sequences:  $(\cdots\overline{m}21\cdots)$, $(\cdots1\overline{m}2\cdots)$, $(\cdots12\overline{m}\cdots)$, $(\cdots2\overline{m}1\cdots)$, $(\cdots,21\overline{m}\cdots)$. By the same arguments we can show that  the order of the action $\widehat r_{\overline{m}1}\cdot \widehat r_{12}$ on such sequences is $6$.

(Subcase 3.II) The case is when the defining sequences satisfy that among three numbers $\overline{m}$, $1$ and $2$, there are only two numbers (including $\overline m$) adjacent in the sequences, but the third one is remote to them (for example,  $\overline{m}$  and $ 1$ are adjacent to each other, $2$ is not adjacent to them). In this case,  $m\geq 3$ or $n\geq 3$, consequently, $m+n\geq 4$. We similarly demonstrate  the action of $\widehat r_{\overline{m}1}\cdot \widehat r_{12}$ on those sequences
\begin{align*}
&(\cdots\overline{m}1\cdots 2\cdots)\overset
{\widehat r_{12}}{\longrightarrow}(\cdots\overline{m}2\cdots 1\cdots)\overset
{\widehat r_{\overline{m}1}}{\longrightarrow}\cr
&(\cdots\overline{m}2\cdots 1\cdots)\overset
{\widehat r_{12}}{\longrightarrow}(\cdots\overline{m}1\cdots 2\cdots)\overset
{\widehat r_{\overline{m}1}}{\longrightarrow}\cr
&(\cdots1\overline{m}\cdots 2\cdots)\overset
{\widehat r_{12}}{\longrightarrow}(\cdots,2\overline{m}\cdots 1\cdots)\overset
{\widehat r_{\overline{m}1}}{\longrightarrow}\cr
&(\cdots,2\overline{m}\cdots 1\cdots)\overset
{\widehat r_{12}}{\longrightarrow}(\cdots1\overline{m}\cdots 2\cdots)\overset
{\widehat r_{\overline{m}1}}{\longrightarrow}\cr
&(\cdots\overline{m}1\cdots 2\cdots).
\end{align*}
  So the order of the action of $\widehat r_{\overline{m}1}\cdot \widehat r_{12}$ on such sequences is  $4$.

 There is another possibility of similar sequences: $\overline m$ and $2$ are adjacent while $1$ is remote.  By the same arguments we can show that  the order of the action $\widehat r_{\overline{m}1}\cdot \widehat r_{12}$ on such sequences is $4$.

Recall that there are $(m+n)!$ possible fundamental systems for $\mathfrak{gl(m|n)}$, and there are $(m+n)!$  defining sequences of length $m+n$.  When $m+n\geq 4$,  both cases  (3.I) and (3.II) necessarily occur on some sequences.  So the least common multiplicity of the orders  $6$ and $4$ respectively occurring  for two cases must be exactly the order of $\widehat r_{12}\cdot \widehat r_{\overline{m}1}$ which is equal to $12$.

 (4) The case when $y=\widehat r_{\overline{m-1,m}}$: in  this case, one can  similarly prove that order of $\widehat r_{\overline{m-1,m}}\cdot \widehat r_{\overline{m}1}$ is equal to $12$  as we have done in (3).

 Summing up, the proof is completed.
 \end{proof}

\begin{remark} Note that $\mathfrak{gl}(m|n)\simeq \mathfrak{gl}(n|m)$.
Example \ref{gl12} provides all possible information on Coxeter groups for the non-trivial general linear Lie algebras $\mathfrak{gl}(m|n)$ with  $m+n< 4$.
\end{remark}

\section{Defining sequences for fundamental systems for $\mathfrak{spo}(2m|l)$}\label{sec: def seq for osp}

 We first recall the root systems of $\mathfrak{spo}(2m|l)$. Then we introduce defining sequences for fundamental (root) systems.

\subsection{Root system for $\mathfrak{spo}(2m|2n)$}
{By definition, $\mathfrak{spo}(2m|2n)\subset\mathfrak{gl}(2m|2n)$ as a Lie subalgebra. The root system of $\mathfrak{spo}(2m|2n)$ can be understood in the same way as in the case of $\gl(2m|2n)$, by taking the standard Cartan subalgebra $\hhh$ to be the one consisting of  diagonal matrices in $\mathfrak{spo}(2m|2n)$.
}

 Keep the notations in \S\ref{sec: fun sys}, in particular we  write $\delta_i=\varepsilon_{\bar i}$.
Then the root system $\Phi=\Phi_{\bar{0}}\bigcup \Phi_{\bar{1}}$ of $\mathfrak{spo}(2m|2n)$
 is given by
 \begin{align*}
\{\pm\delta_{i}\pm\delta_{j}, \pm 2\delta_{p},\pm \varepsilon_{k}\pm \varepsilon_{l}\}\cup\{\pm\delta_{p}\pm\varepsilon_{q}\}
\end{align*}
where $1\leq i<j\leq m, 1\leq k<l\leq n, 1\leq p\leq m, 1\leq q\leq n$ (see for example, \cite{K, SW} or \cite[\S1.3.4]{CW}).

\subsubsection{}\label{sec: first type CD}
The standard positive system $\Phi^{+}=\Phi_{\bar{0}}^{+}\bigcup \Phi_{\bar{1}}^{+}$ corresponding to the standard Borel subalgebra for   $\mathfrak{spo}(2m|2n)$ is $\{\delta_{i}\pm\delta_{j},  2\delta_{p}, \varepsilon_{k}\pm \varepsilon_{l}\}\cup\{\delta_{p}\pm\varepsilon_{q}\}$
where $1\leq i<j\leq m, 1\leq k<l\leq n, 1\leq p\leq m, 1\leq q\leq n$.
The fundamental system $\Pi$ of $\Phi^{+}$ contains one odd simple root $\delta_{m}-\varepsilon_{1}$, and it is given by
\begin{align}\label{eq: FR first type CD}
\Pi(\mathfrak{spo}(2m|2n))=\{\delta_{i}-\delta_{i+1},\delta_{m}-\varepsilon_{1}, \varepsilon_{k}-\varepsilon_{k+1}, \varepsilon_{n-1}+\varepsilon_{n}\mid 1\leq i\leq m-1, 1\leq k\leq n-1\}.
\end{align}
Recall that the extended standard fundamental system  $\widetilde \Pi$ for $\mathfrak{spo}(2m|2n), n\geq 2$ is $\widetilde\Pi=\Pi\cup \{-2\delta_1\}$(cf. (\ref{eq: ext fund sys})), and that the Coxeter system is  $(\wsc, \ssp)$.

\subsubsection{}\label{sec: second type CD}  We  list another type of positive systems in $\Phi$:  $\{\delta_i\pm \delta_j, 2\delta_p, \varepsilon_k\pm \varepsilon_l\} \cup \{ \varepsilon_q\pm \delta_p\}$, where $1\leq i<j\leq m$ and $1\leq k<l\leq n$,  $1\leq p\leq m$,  and $1\leq q\leq n$. The fundamental system is
\begin{align}\label{eq: FR second type CD}
\{\varepsilon_k-\varepsilon_{k+1}, \varepsilon_n-\delta_1, \delta_i-\delta_{i+1}, 2\delta_m\mid 1\leq i\leq m-1, 1\leq k\leq n-1\}.
\end{align}

\subsubsection{Positive and fundamental systems for $\mathfrak{spo}(2m|2n)$}\label{S5.2}
Recall that there are  (at least) two Dynkin diagrams arising from fundamental systems, for $\mathfrak{spo}(2m|2n)$ of different shapes. The analysis on  them can be divided into $2$ cases (see \cite[\S1.3.4]{CW} for details).

{\textbf{Case 1: the fundamental systems do not contain any long root (i.e., a root of the form $2\delta_p$)}}

With parity ignored, the subset $\widetilde\Phi:=\Phi\backslash \{\pm 2\delta_p\mid 1\leq p\leq m\}$ may be identified with the root system of the classical Lie algebra $\mathfrak{so}(m+n)$, while all such fundamental systems of $\Phi$ in this case are exactly  the fundamental systems for $\widetilde \Phi$. The number of  fundamental systems that do not contain
any long root (i.e., a root of the form $2\delta_p$) is $\frac{n}{m+n}2^{m+n-1}\cdot(m+n)! =n2^{m+n-1}\cdot(m+n-1)! $ (cf. \cite[\S1.3.4]{CW}).

{\textbf{Case 2: the fundamental systems contain some long root}}
In this case, we consider $\widehat\Phi:=\Phi\cup \{\pm 2\varepsilon_j\mid 1\leq j\leq n\}$. With the parity ignored, $\widehat\Phi$ may be identified with the root system of $\mathfrak{sp}(2m+2n)$, while all such fundamental systems of $\Phi$ in this case are exactly  the fundamental systems for $\widehat \Phi$.

The number of  fundamental systems that contain some long root  is $\frac{m}{m+n}2^{m+n}\cdot(m+n)! =m2^{m+n}\cdot(m+n-1)!$ (cf. \cite[\S1.3.4]{CW}).

\subsubsection{} The number of all positive systems of $\mathfrak{spo}(2m|2n)$ is $$\frac{n}{m+n}2^{m+n-1}\cdot(m+n)!+\frac{m}{m+n}2^{m+n}\cdot(m+n)!=(n+2m)2^{m+n-1}\cdot(m+n-1)!
$$
{{\text{ (\cite[\S1.3.4]{CW})}}}.

\subsection{Root system for $\mathfrak{spo}(2m|2n+1)$}\label{s5.1}

By definition, $\mathfrak{spo}(2m|2n+1)\subset\mathfrak{gl}(2m|2n+1)$ as a Lie subalgebra. The root system of $\mathfrak{spo}(2m|2n+1)$ can be understood in the same way as in the case of $\gl(2m|2n+1)$, by taking the standard Cartan subalgebra $\hhh$ to be the one consisting of diagonal matrices in $\mathfrak{spo}(2m|2n+1)$.
 Then  the root system can be described as:
$\Phi=\Phi_{\bar{0}}\bigcup \Phi_{\bar{1}}$  with $$\{\pm\delta_{i}\pm\delta_{j}, \pm 2\delta_{p},\pm \varepsilon_{k}\pm \varepsilon_{l},\pm \varepsilon_{q}\}\cup\{\pm\delta_{p}\pm\varepsilon_{q},\pm\delta_{p}\}$$
where $1\leq i<j\leq m, 1\leq k<l\leq n, 1\leq p\leq m, 1\leq q\leq n$.

Similar to $\mathfrak{spo}(2m|2n)$, one can list further information on root systems of $\mathfrak{spo}(2m|2n+1)$ (see \cite[\S1.3.4]{CW}). We omit here.

\subsection{General convention concerning super reflections of  $\mathfrak{spo}(2m|l)$}\label{sec: order conv after odd general} Let $\ggg=\mathfrak{spo}(2m|l)$ with $l=2n+1$ or $2n$, not less than $2$, and $\Pi_0$  an ordered simple fundamental (root) system of $\ggg$ containing an isotropic odd root $\theta$:
$$\Pi_0=\{\theta_1,\ldots,\theta_k=\theta,\ldots, \theta_{N-1}, \theta_N\}$$
with $N:=m+n$.
Denote by $\Pi_0'$ the fundamental system $\widehat r_{\theta}(\Pi_0)$, which comes from $\Pi_0$ by transformation via the odd reflection  $\widehat r_{\theta}$.
 We make a convention on the order of $\Pi'_0$ as below {unless stated otherwise}
 \begin{align}\label{eq:  change order after odd ref}
\{\theta'_1,\ldots,  \theta'_{k-1}
, -\theta_{k}, \theta'_{k+1},\ldots, \theta'_N\}
\end{align}
where we denote for $i\ne k$
\begin{align}\label{eq:  change order after odd ref-2}
\theta_i'=\begin{cases} \theta_i+\theta, &\text{ if } (\theta_i,\theta)\ne0;\cr
\theta_i, &\text{   otherwise}.
\end{cases}
\end{align}
As to an ordinary (even) reflection $\widehat r_\alpha$, it is a well-known fact that for any ordered fundamental root system $\Pi_0=\{\theta_1,\ldots,\theta_N\}$, it becomes under the action of $\widehat r_\alpha$,
$$\Pi'_0=\{\widehat r_\alpha(\theta_1),\ldots,\widehat r_\alpha(\theta_N)\}$$
which is in an order corresponding to the order of $\Pi_0$.

\subsubsection{Principle to determine the order of a fundamental system}
\label{sec: order conv after odd}
 Keep the notations as in the previous subsection. In particular, let $\Pi_0=\{\theta_1,\ldots,\theta_k=\theta,\ldots, \theta_{N-1}, \theta_N\}$ be an ordered simple fundamental (root) system of $\ggg$ containing  an isotropic odd root $\theta=\pm(\vep_1-\delta_m)$.
The modified convention on the order of $\Pi'_0=\widehat r_\theta(\Pi_0)$ is as follows
 \begin{itemize}
 \item[(Case 1)] if $k=N$ (i.e. $\theta_N=\theta$) and $\theta_{N-1}=\mp(\delta_m+\vep_1)$,
then we define the order of $\Pi_0'$ to be as presented below
 \begin{align}\label{eq: Pi 1 exp}
\{\theta'_1,\ldots, \theta'_{N-2},-\theta,\mp 2\delta_m\}
\end{align}
where we  exchanged the position of the $(N-1)$th and the $N$th new simple roots such that $\mp 2\delta_m$
only appears in the end, and where
we denote for $i\ne N-1, N$,
\begin{align*}
\theta_i'=\begin{cases}
\theta_i+\theta, &\text{ if } (\theta_i,\theta)\ne0;\cr
\theta_i, &\text{   otherwise }.
\end{cases}
\end{align*}
 \item[(Case 2)] Otherwise,  we define  the order to be as presented below
 \begin{align*}
\{\theta'_1,\ldots, \ldots, \theta'_{k-1}
, -\theta_{k}, \theta'_{k+1},\ldots, \theta'_N\}
\end{align*}
where we denote for $i\ne k$
\begin{align*}
\theta_i'=\begin{cases} \theta_i+\theta, &\text{ if } (\theta_i,\theta)\ne0;\cr
\theta_i, &\text{   otherwise}.
\end{cases}
\end{align*}
\end{itemize}
As to an ordinary (even) reflection $\widehat r_\alpha$, it is a well-known fact that for any ordered fundamental root system $\Pi_0=\{\theta_1,\ldots,\theta_N\}$, it becomes under the action of $\widehat r_\alpha$,
$$\Pi'_0=\{\widehat r_\alpha(\theta_1),\ldots,\widehat r_\alpha(\theta_N)\}$$
which is in an order corresponding to the order of $\Pi_0$.

\subsubsection{}
From now on, we will use the following convention  for $\widetilde{i}\in\{\bar 1,\ldots,\overline{m-1}\}\cup \{1,\ldots,n-1\}$.
\begin{conven}
Set $\vep_{\widetilde i}-\vep_{\widetilde{i+1}}:=\vep_{\bar i}-\vep_{\overline{i+1}}$ ($=\delta_i-\delta_{i+1}$) if   $\widetilde {i}\in\{ \bar1,\ldots,\overline{m-1}\}$,  and $\vep_i-\vep_{i+1}$ if $\widetilde i\in \{1,\ldots,n-1\}$.
\end{conven}

\subsection{Type $D(m|n)$}\label{sec: corres type D} Now we first consider $\ggg=\mathfrak{spo}(2m|2n)$ with $n\geq 2$. We have to modify the general convention  in \S\ref{sec: order conv after odd general} concerning the order of a fundamental system after the action of odd simple reflections.

Keep  in mind that the order of $\Pi$ for  $\ggg=\frak{spo}(2m|2n)$ ($n\geq 2$) is the one given below  $$\Pi=\{\delta_1-\delta_2,\ldots,\delta_m-\varepsilon_1,
\varepsilon_1-\varepsilon_2,\ldots,
\varepsilon_{n-1}-\vep_n, \vep_n+\vep_{n-1}\}.$$
 Any two adjacent roots of the ordered $\Pi$  form a pair of the following form: either $\vep_{\widetilde i}-\vep_{\widetilde{i+1}}, \vep_{\widetilde{i+1}}-\vep_{\widetilde{i+2}}$ {with $\widetilde i\in\{\bar 1,\ldots,\overline {m-2}\}\cup\{1,\ldots,n-2\}$, or $\delta_m-\vep_1, \vep_1-\vep_2$,  or $\varepsilon_{n-1}-\vep_n, \vep_n+\vep_{n-1}$, or $\delta_{m-1}-\delta_m, \delta_{m-1}-\delta_1$.}

The signs of positive  and negative, say, $\natural,\sharp\in \{\pm1\}$, have to be taken into consideration after transformations on $\Pi$ through {the super Weyl group action.}

\subsubsection{} \label{eq: Pi 1 exp} With the  above conventions, especially in \S\ref{sec: order conv after odd},  we define the order of any fundamental system $\Pi_1={w}(\Pi)$ which comes from the standard fundamental system $\Pi$ by a composition of super simple reflections $w=\widehat r_l\cdots \widehat r_1$,  via one-to-one correspondence between $\Pi_1$ and $\Pi$ through the action of $w$.

\subsubsection{} After the above preparation, we proceed with arguments on the precise description of the ordered fundamental systems.

We first make Assumption \ref{conv: appendix} {which is available for all $\mathfrak{spo}(2m|l)$}, then we show the assumption holds true.

\begin{convention}\label{conv: appendix} Let $\Pi_1$ be a fundamental system with order as {{in \S\ref{sec: order conv after odd}.}}
Then {any two adjacent roots} are of the forms like $\natural\vep_{i}-\sharp\vep_{j}, \sharp\vep_j+\heartsuit$ satisfying the following conditions
\begin{itemize}
\item[(1)] $\vep_i\ne\vep_j$ for  $i,j\in I(m|n)$ with $i\ne j$.

\item[(2)] Both $\natural\vep_{i}-\sharp\vep_{j}$ and $\sharp\vep_j+\heartsuit\in \Phi$ with $\heartsuit=\pm \varepsilon_k$ for some $k\in I(m|n)$.

\item[(3)] $\Pi_1$ is of the form
\begin{align}\label{eq: general fs form}
\{\natural_1\vep_{i_1}-\sharp_2\vep_{i_2}, \natural_2\vep_{i_2}-\sharp_3\vep_{i_3}, &\ldots,\ldots,\cr
 \ldots,&\natural_{m+n-1}\vep_{i_{m+n-1}}-\sharp_{m+n}\vep_{i_{m+n}}, \natural_{m+n}\vep_{i_{m+n}}+\sharp_{m+n+1}\vep_{i_{m+n+1}}\}
\end{align}
with $\sharp_k=\natural_k$ for $k=2,\ldots,m+n-1$,
where the sequence $(i_1i_2\cdots i_{m+n})$ is a permutation of the sequence $(\bar1\overline2\cdots\overline{m}12\cdots n)$.
\end{itemize}
\end{convention}
In the above assumption, $\natural_l\vep_{i_l}$ is called the {{front}} term of the $l$th root, and $\sharp_{l+1}\vep_{i_{l+1}}$ is called its rear term. Still set $N=m+n$ in the following.

\begin{lemma}\label{lem: str lem for def seq} Suppose $\Pi_0=\{\theta_1,\ldots, \theta_k, \theta_{k+1},\ldots, \theta_{N}\}$ is of the form  (\ref{eq: general fs form}),  and $\Pi_0'=\widehat r_{\theta}(\Pi_0)$ is the transformation of $\Pi_0$ under  the odd reflection $\widehat r_{\theta}$ for {$\theta\in\{\pm(\delta_m-\vep_1)\}$}. Then  $\Pi_0'$  is also  of the form (\ref{eq: general fs form}).
\end{lemma}

\begin{proof}
By definition, if any of the elements   $\theta=\natural(\delta_m-\vep_1)$ and $\natural(\vep_1-\delta_m)$ do not appear in $\Pi_0$, then $\widehat r_{\theta}$ keeps $\Pi_0$ fixed pointwise. In this case, there is nothing  to do.

Now suppose that $\theta_k$ is equal to $\natural(\delta_m-\vep_1)$ or $\natural(\vep_1-\delta_m)$ for $k\in \{1,\ldots,N\}$.

(1) If $\theta_k=\natural(\delta_m-\vep_1)$, or $\natural(\vep_1-\delta_m)$ with $k<m+n$, then $\theta_{k+1}=\natural\vep_1-\sharp_{k+2}\vep_{i_{k+2}}$ or $\natural\delta_m-\sharp_{k+2}\vep_{i_{k+2}}$ respectively. For example, if $\theta_k=\delta_m-\vep_1$, by the assumption on $\Pi_0$ it is easily deduced that $\Pi_0$ must be of the form:
$$(\cdots, \natural_{k-1}\vep_{i_{k-1}}-\delta_m,\delta_m-\vep_1, \vep_1-\sharp_{k+2}\vep_{i_{k+2}},\cdots),$$
and by the transformation of $\widehat r_{\delta_m-\vep_1}$, $\Pi_0$ becomes
$$\widehat r_{\delta_m-\vep_1}(\Pi_0)= (\cdots, \natural_{k-1}\vep_{i_{k-1}}-\vep_1, \vep_1-\delta_m, \delta_m-\sharp_{k+2}\vep_{i_{k+2}},\cdots) $$
which is still of the form (\ref{eq: general fs form}).  As to the situation $\theta_k=-(\delta_m-\vep_1)$ or $\natural(\vep_1-\delta_m)$ with $k<N (=m+n)$, by the same arguments it is still shown that $\widehat r_{\delta_m-\vep_1}(\Pi_0)$ is of the form (\ref{eq: general fs form}).

(2) The remaining case is  $\theta_k=\natural(\delta_m-\vep_1)$, or $\natural(\vep_1-\delta_m)$ with $k=N$,  with which  we need to make a special treatment.

(2.1)  First we deal with the case when $\theta_N=\pm(\delta_m-\vep_1)$. Then  the assumption on $\Pi_0$ yields that {the last but one root of $\Pi_0$ by its order is of} the form
 $$\natural_{N-1}\vep_{i_{N-1}}- (\pm\delta_m).$$
 We claim that  {this cannot occur: we prove that by contradiction.  Suppose it does.}  We first deal with the case when  $\theta_N=\delta_m-\vep_1$ where $\delta_m$ is the front term. Then $\vep_1$ must appear in some $l$th root as its {{front}} term with $l<N$ because $\Pi_0$ is of the form (\ref{eq: general fs form}).

(Case 1) Suppose $l=N-1$, i.e.  $\vep_{i_{N-1}}=\vep_1$. So the $(N-1)$th simple root is  of the form $\pm \vep_1-\delta_m$. This is absurd, because  this contradicts the property of fundamental  root systems.   Precisely, the occurrence of  $\vep_1-\delta_m$ leads to the linearly-dependence of {the last root and the last but one} in $\Pi_0$. The occurrence of $-\vep_1-\delta_m$ leads to the absurdity of $2\vep_1$ being a root, which is deduced by the  property (see Lemma \ref{fun lemm}) that $(-\vep_1-\delta_m, \delta_m-\vep_1)\ne0$ implies that the sum of both is still a root.

 (Case 2) If $l<N-1$, let us analyse the form of the $l$th root in $\Pi_0$. First it is impossible that its front term is  $\vep_1$ because this leads to the sum of all roots running through the $l$th root to the $(N-1)$th {root to be equal to} $\vep_1-\delta_m$. This sum must be a positive root with respect to $\Pi_0$, contradicting the $N$th root being $\delta_m-\vep_1$ in $\Pi_0$.  Suppose $-\vep_1$ is the {{front}} term of the $l$th root, then its rear term is never $\delta_m$ (otherwise, the $l$th term becomes $-\vep_1+\delta_m$  linearly dependent on the last root). Hence the $l$th root $-\vep_1-\sharp_{l+1}\vep_{i_{l+1}}$ admits the property that $(-\vep_1-\sharp_{l+1}\vep_{i_{l+1}}, \delta_m-\vep_1)\ne0$, consequently $-2\vep_1-\sharp_{l+1}\vep_{i_{l+1}}+\delta_m$ is a root by Lemma \ref{fun lemm}. This is absurd.

By the same arguments, we can exclude the situation when the last root is $-\delta_m+\vep_1$ with front term $-\delta_m$.

(2.2) Now we deal with the situation when $\theta_N=\pm(\vep_1-\delta_m)$, which means the front term is $\pm\vep_1$, correspondingly the rear term of $\theta_{N-1}$ is $\mp\vep_1$.

 We first deal with the case when $\theta_N=\vep_1-\delta_m$ with front term $\vep_1$. Then the rear term of $\theta_{N-1}$ is $-\vep_1$.
Let us analyse the possible root $\theta_l$ whose front term is $\natural\delta_m$. Such $l$ is unique, due to the assumption on $\Pi_0$. We proceed with different subcases.

(Case 1) If $l=N-1$. Then $\theta_{N-1}=-\delta_m-\vep_1$. Correspondingly, $\theta_{N-2}=\natural_{N-2}\vep_{i_{N-2}}+\delta_m$ with $\vep_{i_{N-2}}\ne \vep_1$, $\delta_m$. In this case, after the transformation $\widehat r_{\vep_1-\delta_m}=\widehat r_{\delta_m-\vep_1}$, $\Pi_0$ changes into $$\Pi_0'=(\cdots\ldots\ldots, -2\delta_m, \delta_m-\vep_1)$$
where its $(N-2)$th root becomes $\natural_{N-2}\vep_{i_{N-2}}+\varepsilon_1$,  the first $N-3$ roots do not change. By the convention in \S\ref{sec: order conv after odd},  the order of $\Pi_0'$ is defined to be the one as below
 $$\Pi_0'=(\cdots\ldots\ldots,\natural_{N-2}\vep_{i_{N-2}}+\varepsilon_1,  -\vep_1+\delta_m, -2\delta_m)$$
 which is still of the form (\ref{eq: general fs form}). This is what is desired.

(Case 2) If $l<N-1$, the front term of $\theta_{N-1}$ cannot be $\delta_m$ because if so, then $\delta_m-\vep_1$ becomes a positive root in the positive root system arising from $\Pi_0$ by the same arguments in (Case 2) of the previous  (2.1).  So  $\theta_{l}=-\delta_m-\sharp_{l+1}\vep_{i_{l+1}}$. Now $\vep_{i_{l+1}}$ is not identical to $\pm\vep_1$ because $\Pi_0$ is of the form (\ref{eq: general fs form}), and $l<N-1$. Thus $(\theta_l, \vep_1-\delta_m)\ne0$, which leads to the contradiction that $-2\delta_m-\sharp_{l+1}\vep_{i_{l+1}}+\vep_1$ is a root.

By the same arguments, we can prove that $\Pi_0'$ is still of the form (\ref{eq: general fs form}) when the last root is $-\vep_1+\delta_m$ with {{front}} term $-\vep_1$.
\end{proof}

According to the previous proof,  the following result is directly  concluded.
\begin{corollary}  Keep the notations and assumptions as the previous lemma. Particularly,
$$\Pi_0=\{\theta_1,\ldots, \theta_k, \theta_{k+1},\ldots, \theta_{N}\}$$
 is of the form  (\ref{eq: general fs form}). If $\theta_N=\pm (\varepsilon_1-\delta_m)$, then $\theta_{N-1}=\mp(\delta_m+\varepsilon_1)$.
\end{corollary}

Next we look at the situation about the ordinary (even) reflections.

\begin{lemma}\label{lem: str lem for def seq ordianary case} Suppose $\Pi_0=\{\theta_1,\ldots, \theta_k, \theta_{k+1},\ldots, \theta_{N}\}$ is of the form  (\ref{eq: general fs form}),  and $\Pi_0'=\widehat r_{\theta}(\Pi_0)$ the transformation of $\Pi_0$ under  the ordinary (even) reflection $\widehat r_{\alpha}$ for $\alpha=\vep_{\widetilde{i}}-\vep_{\widetilde{i+1}}$ with $\widetilde{i}\in\{\bar 1,\ldots,\overline{m-1}\}\cup \{1,\ldots,n-1\}$, $2\delta_1$ or $\vep_{n}+\vep_{n+1}$. Then  $\Pi_0'$  is also  of the form (\ref{eq: general fs form}).
\end{lemma}

\begin{proof}  Note that $\widehat r_{\alpha}$ can be identified with a transformation of $J:=\{\pm \bar i\mid i=1,\ldots,m\}\cup\{\pm j\mid j=1,\ldots, n\}$. Precisely, if $\alpha=\vep_{\widetilde{i}}-\vep_{\widetilde{i+1}}$,  then $\widehat r_{\alpha}$ can be identified with the reflection arising from $(\widetilde i \widetilde {(i+1)})$,  which means  exchanging $\widetilde i$ and $\widetilde{(i+1)}$, fixing other indexes.  If $\alpha=2\delta_1$, then $\widehat r_\alpha$ can be identified with the reflection arising from $(1, -1)$ which means exchanging $1$ and $-1$, fixing other indexes. If $\alpha=\vep_{n-1}+\vep_n$, then $\widehat r_{\alpha}$ can be identified with the reflection arising from $(n-1, -n)$ and $(n, -n-1)$,  which means  exchanging $n-1$ and $-n$, $n$ and $-(n-1)$, fixing other indexes.   Hence, the lemma follows.
\end{proof}

Summing up Lemmas \ref{lem: str lem for def seq} and \ref{lem: str lem for def seq ordianary case}, we have a direct corollary.

\begin{corollary}\label{lem: str lem for def seq summing}
Any fundamental system of $\Pi_1$ {which  is
ordered as explained in} \S\ref{sec: order conv after odd} is of the form (\ref{eq: general fs form}).
\end{corollary}


\subsubsection{Defining sequences}

\begin{lemma}\label{lem: defining sec for D} Keep the notations and assumptions as in the previous corollary. Set
$$\sfd(\Pi_1):=(\natural_1 i_1,\natural_2 i_2,\ldots,\natural_{m+n}i_{m+n}).$$
 {Then the  sequence $\sfd(\Pi_1)$ is uniquely determined by  $\Pi_1$.} Furthermore,  different fundamental systems $\Pi$ have different associated sequences $\sfd(\Pi)$.
\end{lemma}

\begin{proof} (1) We first look at the uniqueness of $\sfd(\Pi_1)$. By Corollary \ref{lem: str lem for def seq summing}, $\Pi_1$ has the form (\ref{eq: general fs form}). Note that it is impossible that  two roots of the forms $\alpha-\beta$, $\beta-\alpha$ appear  simultaneously in any fundamental  system. So the sequence $\sfd(\Pi_1)$ is unique.

(2) Suppose $\Pi_1$ and $\Pi_1'$ are two fundamental systems with the same defining sequence. Then their first $(N-1)$ roots  are the same (still denote $N=m+n$), and the {{front}} terms of their last roots are the same.
{We claim that  their last roots must be the same.  Consequently, $\Pi_1$ and $\Pi_1'$ are the same. }

Now we prove the claim. Suppose $\Pi_1=\{\gamma_1,\ldots,\gamma_{N-1}, \gamma_N\}$, and $\Pi_1'=\{\gamma_1,\ldots,\gamma_{N-1}, \gamma_N'\}$ with
\begin{align*}
\gamma_N'&=\natural_N\vep_{i_N}+\sharp_{N+1}\vep_{i_{N+1}},\cr
\gamma_N&=\natural_N\vep_{i_N}+\sharp_{N+1}'\vep_{i_{N+1}}'
\end{align*}
where $\natural_N$, $\sharp_{N+1}$ and $\sharp_{N+1}'\in \{\pm 1\}$, and $\vep_{i_{N+1}}, \vep_{i_{N+1}}'\in \{\delta_1,\ldots,\delta_m;\vep_1,\ldots,\vep_n\}$. This implies that  $\gamma_N$ belongs to the positive root system with respect to $\Pi_1$, and $\gamma_N'$ belongs to the positive root system with respect to $\Pi_1'$. 
 Thus, we can express $\gamma_N'=\sum_{i=1}^{N-1}a_i\gamma_i+a\gamma_N$ and
 $\gamma_N=\sum_{i=1}^{N-1}b_i\gamma_i+b\gamma_N'$ with all $a_i, b_i, a,b$ being non-negative numbers. Thus we have
$(1-ab)\gamma_N=\sum_{i=1}^{N-1}(b_i+ba_i)\gamma_i$ which implies that $ab=1$ due to the linearly-independence of $\gamma_i$, $i=1,\ldots,N$. Correspondingly, $a=1=b$ and $a_i=b_i=0$.  Hence $\gamma_N=\gamma_N'$.

The proof is completed.
\end{proof}

{{
Due to the above Lemma, we introduce the defining sequence of type $D(m|n)$.
\begin{definition}\label{def: def seq D}
(Defining sequence of type $D(m|n)$)
Given any Borel  subalgebra $B$ in $\mathcal{B}$ for the Lie superalgebra $\ggg$ of type $D(m|n)$, it  corresponds to some unique fundamental system $\Pi_1$ which  is
ordered as explained in \S\ref{sec: order conv after odd}, thereby is of the form  (\ref{eq: general fs form}). The sequence  $(\natural_1 i_1,\natural_2 i_2,\ldots,\natural_{m+n}i_{m+n})$  is then called the defining sequence of $\Pi_1$ for type $D(m|n)$.

\end{definition}
}}

{
\begin{remark}\label{remark: compatable of both action}
 (1) It is worth mentioning that in \cite{CW}, the conjugacy classification of fundamental systems are discussed through $\varepsilon\delta$-sequences, up to the action of the usual Weyl group $\scrw$. However, it seems not useful in  our arguments.

{ (2)
From the construction of the definition-sequences, it is easily seen that there is a correspondence between the fundamental systems and the defining sequences (the formal formulation of this correspondence will be given in Theorem \ref{prop: presenting seq type D} (before that theorem, we will first give an initial convention for the defining sequence of the standard simple system, and determine the change of a defining sequence corresponding to the action of $\whw$ on the standard simple system (see Proposition \ref{app: gen seq}). This is necessary for us to draw Coxeter graphs), and  its compatibleness with $\whw$.  Actually, an element of $\whw$ uniquely determines a transformation of the set of simple root systems. Correspondingly, the transformation of this element from $\whw$ on the set of  defining sequences  is given simultaneously according to the correspondence mentioned above.

}
\end{remark}



}
\subsection{Type $C(m)$} Let $\ggg=\mathfrak{spo}(2m|2)$. Set $N=m+1$.
Keep  in mind that the order of $\Pi$ is as follows  $$\Pi=\{\vep_1-\delta_1,\delta_1-\delta_2,\ldots,\delta_{m-1}-\delta_m, 2\delta_m\}.$$
 Any {two adjacent} roots of $\Pi$ in this order are a pair of the following form: either $\delta_{i}-\delta_{{i+1}}, \delta_{{i+1}}-\delta_{{i+2}}$ with $ i\in\{ 1,\ldots, m\}$, or $\vep_1-\delta_1, \delta_1-\delta_2$;  or $\delta_{m-1}-\delta_m, 2\delta_m$.

 Any other fundamental system $\Pi_1$ is obtained via certain element $w=\widehat r_{\theta_l}\widehat r_{\theta_{l-1}}\cdots \widehat r_{\theta_1}$ {of the super Weyl group}, i.e. $\Pi_1=w(\Pi)$, the order of
 $\Pi_1$ is decided by the composition of changes arising {from all the $\widehat r_{\theta_k}$'s} ($k=1,\ldots,l$). The change of the order arising from the single super reflection $\widehat r_{\theta_k}$ is  defined below:

 \begin{conven}\label{conven: change of order under odd ref for type C}
 Let
  $\Pi_0=\{\theta_1,\ldots,\theta_k=\theta,\ldots, \theta_{N-1}, \theta_N\}$ be an ordered simple fundamental (root) system of $\ggg$ containing  an isotropic odd root $\theta=\pm(\vep_1-\delta_1)$,  and $\Pi_1=\widehat r_\theta(\Pi_0)$.
 \begin{itemize}
 \item[(Case 1)] if $k=N$ (i.e. $\Theta_N=\theta$) and $\theta_{N-1}=\mp(\delta_1+\vep_1)$,
then the order of  $\Pi_1$ is defined to be
 \begin{align}\label{eq: Pi 1 exp}
\{\theta'_1,\ldots, \theta'_{N-2},-\theta,\mp 2\delta_1\}
\end{align}
where we  exchanged the position of the $(N-1)$th and the $N$th new simple roots such that $\mp 2\delta_m$
appears in the end, and where
we denote for $i\ne N-1, N$,
\begin{align*}
\theta_i'=\begin{cases}
\theta_i+\theta, &\text{ if } (\theta_i,\theta)\ne0;\cr
\theta_i, &\text{   otherwise }.
\end{cases}
\end{align*}
 \item[(Case 2)] Otherwise,  we define  the order to be presented below
 \begin{align*}
\{\theta'_1,\ldots, \ldots, \theta'_{k-1}
, -\theta_{k}, \theta'_{k+1},\ldots, \theta'_N\}
\end{align*}
\noindent where we denote for $i\ne k$
\begin{align*}
\theta_i'=\begin{cases} \theta_i+\theta, &\text{ if } (\theta_i,\theta)\ne0;\cr
\theta_i, &\text{   otherwise}.
\end{cases}
\end{align*}
\end{itemize}

As to an ordinary (even) reflection $\widehat r_\alpha$, it is a well-known fact that for any ordered fundamental root system $\Pi_0=\{\theta_1,\ldots,\theta_N\}$, it becomes under the action of $\widehat r_\alpha$,
$$\Pi'_0=\{\widehat r_\alpha(\theta_1),\ldots,\widehat r_\alpha(\theta_N)\}$$
which is in an order corresponding to the order of $\Pi_0$.
\end{conven}

Similarly, the signs of positive  and negative, say, $\natural,\sharp\in \{\pm1\}$, have to be taken into consideration after transformations of $\Pi$ through the super Weyl groups.  We still keep Assumption \ref{conv: appendix} but putting $n=1$, then we show the assumption holds true.

\begin{lemma}\label{lem: lemma 2.8 for type C}  Suppose $\Pi_0=\{\theta_1,\ldots, \theta_k, \theta_{k+1},\ldots, \theta_{N}\}$ is of the form  (\ref{eq: general fs form}),  and $\Pi_0'=\widehat r_{\theta}(\Pi_0)$ is the transformation of $\Pi_0$ under  a super simple reflection $\widehat r_{\theta}\in \textsf{S}_{\widetilde \Pi}$ (see Proposition \ref{prop: gen system}). Then  $\Pi_0'$  is also  of the form (\ref{eq: general fs form}).
\end{lemma}

\begin{proof} The argument is similar to the proofs of Lemmas \ref{lem: str lem for def seq} and \ref{lem: str lem for def seq ordianary case}.  We omit it.
\end{proof}

\begin{lemma}\label{lem: lemma 2.9 for type C}  Keep the notations and assumptions as in the previous lemma. Set
$$\sfd(\Pi_1):=(\natural_1 i_1,\natural_2 i_2,\ldots,\natural_{m+1}i_{m+1}).$$
Then $\sfd(\Pi_1)$ is unique (called the defining  sequence of $\Pi_1$). Furthermore, different fundamental systems have different defining sequences.
\end{lemma}

\begin{proof} The argument is similar to the proof of Lemma \ref{lem: defining sec for D}, but simplified.  We omit it.
\end{proof}

{{
Due to the above Lemma, we introduce the defining sequence of type $C(m)$.
\begin{definition}\label{def: def seq C}(Definition of defining sequences for $C(m)$)
Given any Borel  subalgebra $B$ in $\mathcal{B}$ for the Lie superalgebra $\ggg$ of type $C(m)$, it  corresponds to some unique fundamental system $\Pi_1$ which  is
ordered as explained in \S\ref{sec: order conv after odd}, thereby is of the form  (\ref{eq: general fs form}) with $n=1$. The sequence  $(\natural_1 i_1,\natural_2 i_2,\ldots,\natural_{m+1}i_{m+1})$  is then called the defining sequence of $\Pi_1$ for type $C(m)$.

\end{definition}
}}

\subsection{$B(m|n)$ with $m+n\geq 4, m,n>0$}

Let $\ggg=\mathfrak{spo}(2m|2n+1)$ with $m>0, n>0$. It admits the standard fundamental system
$$\Pi=\{\delta_1-\delta_2,\cdots,\delta_m-\varepsilon_1,\varepsilon_1-\varepsilon_2,
\cdots,\varepsilon_{n-1}-\varepsilon_n,\varepsilon_n\}.$$
 Any two adjacent roots of $\Pi$ in this order are a pair of the following form: either $\vep_{\widetilde i}-\vep_{\widetilde{i+1}}, \vep_{\widetilde{i+1}}-\vep_{\widetilde{i+2}}$ with $\widetilde i\in I(m|n)=\{\bar 1,\ldots,\overline {m-2}\}\cup\{1,\ldots,n-2\}$, or $\delta_m-\vep_1, \vep_1-\vep_2$;  or $\varepsilon_{n-1}-\vep_n, \vep_n$ or $\delta_{m-1}-\delta_m, \delta_{m-1}-\vep_1$.

 Any other fundamental system $\Pi_1$ is obtained via some element $w=\widehat r_{\theta_l}\widehat r_{\theta_{l-1}}\cdots \widehat r_{\theta_1}$ of the super Weyl groups, i.e. $\Pi_1=w(\Pi)$, the order of
 $\Pi_1$ is decided by the composition of changes arising from all $\widehat r_{\theta_k}$ ($k=1,\ldots,l$). The change of the order arising from the single super reflection $\widehat r_{\theta_k}$ is  defined in \S\ref{sec: order conv after odd general}. In particular, the principle concerning the change of orders arising from the odd reflection is given via  (\ref{eq:  change order after odd ref}) and (\ref{eq:  change order after odd ref-2}).

 We still keep Assumption \ref{conv: appendix}, then we show the assumption {always holds ture.}  Keep in mind that the signs of positive  and negative, say, $\natural,\sharp\in \{\pm1\}$, have to be taken into consideration after transformations {of $\Pi$ through the super Weyl group.}

\begin{lemma}\label{lem: lemma 2.10 for type B}
 Let $\ggg=\mathfrak{spo}(2m|2n+1)$ with $n\geq1$.
 \begin{itemize}
 \item[(1)] Consider $n>1$.  Suppose $\Pi_1=\{\theta_1,\ldots, \theta_k, \theta_{k+1},\ldots, \theta_{N}\}$ is of the  form  (\ref{eq: general fs form})${}_\clubsuit$  which is
     \begin{center}
     \begin{itemize}
     \item[$\clubsuit$] \qquad (\ref{eq: general fs form}) but with modification $\sharp_{m+n+1}=0$,
          \end{itemize}
          \end{center}
          and $\Pi'_1=\widehat r_{\theta}(\Pi_1)$ is  the transformation of $\Pi_0$ under    a super simple reflection $\widehat r_{\theta}\in \textsf{S}_{\widetilde \Pi}$ (see Proposition \ref{prop: gen system} for notations). Then  $\Pi'_1$  is also of the form  (\ref{eq: general fs form})${}_\clubsuit$.
\item[(2)] Consider $n=1$. Suppose $\Pi_1=\{\theta_1,\ldots, \theta_k, \theta_{k+1},\ldots, \theta_{N}\}$ is of the  form  (\ref{eq: general fs form})${}_\spadesuit$ which is
    \begin{itemize}
    \item[$\spadesuit$] \qquad
    (\ref{eq: general fs form}) but with modification $\sharp_{m+n}=0$ while the sequence $(i_1i_2\cdots i_{m+n-1} i_{m+n+1})$ is a permutation of the sequence $(\bar1\overline2\cdots\overline{m}12\cdots n)$,
    \end{itemize}
    and $\Pi'_1=\widehat r_{\theta}(\Pi_1)$ is  the transformation of $\Pi_0$ under   a super simple reflection $\widehat r_{\theta}\in \textsf{S}_{\widetilde \Pi}$. Then  $\Pi'_1$  is also of the form  (\ref{eq: general fs form})${}_\spadesuit$.
\end{itemize}
\end{lemma}

\begin{remark}
Before giving the proof, we make some necessary explanations on the forms (\ref{eq: general fs form})${}_\clubsuit$ and   (\ref{eq: general fs form})${}_\spadesuit$. Recall that $\ggg=\mathfrak{spo}(2m|3)$ admits the standard fundamental system
$$\Pi=\{\delta_1-\delta_2,\cdots,\delta_m-\varepsilon_1,
\varepsilon_1\}.$$
By the simple ordinary (even) reflection $\widehat r_{\vep_1}$, $\Pi$ becomes $\Pi_1=\{\delta_1-\delta_2,\cdots, \delta_m, -\vep_1\}$ which is of the form (\ref{eq: general fs form})${}_\spadesuit$. In comparison,   $\ggg=\mathfrak{spo}(2m|2n+1)$ with $n>1$ admits the standard fundamental system
$$\Pi=\{\delta_1-\delta_2,\cdots,\delta_m-\varepsilon_1,\vep_1-\vep_2,\ldots,
\varepsilon_{n-1}-\vep_n, \vep_n\}.$$
Under a simple super reflection $\widehat r_{\theta}$ for $\theta\in \widetilde\Pi$, $\Pi$ becomes $\Pi_1$ which is of the form  (\ref{eq: general fs form})${}_\clubsuit$.
\end{remark}

\begin{proof} If $\widehat r_\theta$ is an ordinary (even) reflection, then the proof can be carried out  by the same arguments as in the proof of Lemma \ref{lem: str lem for def seq ordianary case}.  Now suppose $\widehat r_\theta$ is the odd reflection.  By definition, {if both elements}   $\natural(\delta_m-\vep_1)$ and $\natural(\vep_1-\delta_m)$ do not appear in $\Pi_1$, then $\widehat r_{\theta}$ keeps $\Pi_1$ fixed pointwise. In this case, there is nothing  to do. In the following, we suppose  $\theta_k=\pm\theta$, i.e. $\theta_k$ is that odd isotropic simple root.

We first consider $\ggg=\mathfrak{osp}(2m|3)$, i.e. $n=1$.
Still set $N=m+n=m+1$. By the assumption, $k<N$.
We divide our arguments into different cases.

(Case 1) If $\theta_k=\natural(\delta_m-\vep_1)$, or $\natural(\vep_1-\delta_m)$ with $k<N-1$, then $\theta_{k+1}=\natural\vep_1-\sharp_{k+2}\vep_{i_{k+2}}$ or $\natural\delta_m-\sharp_{k+2}\vep_{i_{k+2}}$ respectively. For example, $\theta_k=\delta_m-\vep_1$, by the assumption of $\Pi_1$ it is easily deduced that $\Pi_1$ must be of the form:
$$(\cdots, \natural_{k-1}\vep_{i_{k-1}}-\delta_m,\delta_m-\vep_1, \vep_1-\sharp_{k+2}\vep_{i_{k+2}},\cdots),$$
and by the transformation of $\widehat r_{\delta_m-\vep_1}$, $\Pi_1$ becomes
$$\widehat r_{\delta_m-\vep_1}(\Pi_1)= (\cdots, \natural_{k-1}\vep_{i_{k-1}}-\vep_1, \vep_1-\delta_m, \delta_m-\sharp_{k+2}\vep_{i_{k+2}},\cdots) $$
which is still of the form (\ref{eq: general fs form})${}_\spadesuit$.  As to the situation $\theta_k=-(\delta_m-\vep_1)$ or $\natural(\vep_1-\delta_m)$ with $k<N (=m+n)$, by the same arguments it is still shown that $\widehat r_{\delta_m-\vep_1}(\Pi_1)$ is of the form (\ref{eq: general fs form})${}_\spadesuit$.

(Case 2) The remaining case is  $\theta_k=\natural(\delta_m-\vep_1)$, or $\natural(\vep_1-\delta_m)$ with $k=N-1$,  for which  we  make a special treatment.   For example,  $\theta_{N-1}=\pm(\delta_m-\vep_1)$. Then  the assumption on $\Pi_1$ yields that the last-but-two root of $\Pi_1$ by its order is of the form
 $$\natural_{N-1}\vep_{i_{N-1}}- (\pm\delta_m),$$
 and the last root of $\Pi_1$ by its order  is $\pm \vep_{\widetilde i}$ with $\widetilde i\in I(m|n)$. It's easily seen that $\widehat r_{\delta_m-\vep_1}(\Pi_1)$ is of the form (\ref{eq: general fs form})${}_\spadesuit$.

As to (1) with $\ggg=\mathfrak{spo}(2m|2n+1)$ ($n>1$), the same arguments as above yield  that  the odd reflection $\widehat r_{\delta_m-\vep_1}$ transforms a fundamental system of the form (\ref{eq: general fs form})${}_\clubsuit$ into another fundamental system of the  form (\ref{eq: general fs form})${}_\clubsuit$.
\end{proof}

\begin{lemma}\label{lem: lemma 2.11 for type B}  Keep the notations and assumptions as in the previous lemma. Set
\begin{align*}
\sfd(\Pi_1):=\begin{cases} (\natural_1 i_1,\natural_2 i_2,\ldots, \natural_{m+n-1}i_{m+n-1}, \natural_{m+n}i_{m+n}), &\text{ if }n>1;\cr
(\natural_1 i_1,\natural_2 i_2,\ldots, \natural_{m+n-1}i_{m+n-1}, \sharp_{m+n+1}i_{m+n+1}), &\text{ if }n=1.
\end{cases}
\end{align*}
Then $\sfd(\Pi_1)$ is unique (called the defining sequence of $\Pi_1$). Furthermore, different fundamental systems have different defining sequences.
\end{lemma}
\begin{proof} We first look at the uniqueness of $\sfd(\Pi_1)$. By Lemma \ref{lem: lemma 2.10 for type B}, $\Pi_1$ has the form (\ref{eq: general fs form})${}_\clubsuit$ or (\ref{eq: general fs form})${}_\spadesuit$, depending on whether $n>1$ or $n=1$.  This lemma also yields that sequence $\sfd(\Pi_1)$ is unique.

 Suppose $\Pi_1$ and $\Pi_1'$ are two fundamental systems with the same defining sequence.

 (1) Suppose $\ggg=\mathfrak{spo}(2m|2n+1)$ with $n>1$. In this case, by the form (\ref{eq: general fs form})${}_\clubsuit$ it is easily seen  that $\Pi_1$ and $\Pi_1'$ are uniquely determined by their defining sequences. Hence, $\Pi_1$ and $\Pi_1'$ are the same.

(2) Suppose $\ggg=\mathfrak{spo}(2m|3)$. In this case, from the form (\ref{eq: general fs form})${}_\spadesuit$ we can deduce that  $\Pi_1$ and $\Pi_1'$ admit the same first $(N-2)$th roots (say $\gamma_1,\ldots,\gamma_{N-2}$ with $N=m+1$),  the same front terms of the last but one roots,   and   the same last roots. So we can suppose
\begin{align*}
\Pi_1&=\{\gamma_1,\ldots,\gamma_{N-2}, \gamma_{N-1}, \gamma_n\}\cr
\Pi_1'&=\{\gamma_1,\ldots,\gamma_{N-2}, \gamma_{N-1}', \gamma_n\}\cr
\end{align*}
with $\gamma_{N-1}=\natural_{N-1}\vep_{i_{N-1}}-\sharp_{N}\vep_{i_N}$,
and  $\gamma_{N-1}'=\natural_{N-1}\vep_{i_{N-1}}-\sharp_{N}'\vep_{i_N}'$. We only need to show that $\sharp_{N}\vep_{i_N}=\sharp_{N}'\vep_{i_N}'$. The same arguments as in the proof of Lemma \ref{lem: defining sec for D} ensures the desired equality.

The proof is completed.
\end{proof}

{{
Due to the above Lemma, we introduce the defining sequence of type $B(m|n)$.
\begin{definition}\label{def: def seq B} (Definition of defining sequences for $B(m|n)$)
Given any Borel  subalgebra $B$ in $\mathcal{B}$ for the Lie superalgebra $\ggg$ of type $B(m|n)$, it  corresponds to some unique fundamental system $\Pi_1$ as presented in Lemma \ref{lem: lemma 2.10 for type B}. The sequence  
\begin{align*}
\sfd(\Pi_1):=\begin{cases} (\natural_1 i_1,\natural_2 i_2,\ldots, \natural_{m+n-1}i_{m+n-1}, \natural_{m+n}i_{m+n}), &\text{ if }n>1;\cr
(\natural_1 i_1,\natural_2 i_2,\ldots, \natural_{m+n-1}i_{m+n-1}, \sharp_{m+n+1}i_{m+n+1}), &\text{ if }n=1.
\end{cases}
\end{align*}
 is then called the defining sequence of $\Pi_1$ for type $B(m|n)$.

\end{definition}
}}

\section{Coxeter graphs for $D(m|n)$ and for $C(m)$}\label{se:5}


{
According to the definition of super Weyl groups,  any two fundamental systems for $\ggg$ can be transformed into each other via the action of  $\whw$.
In order to draw the Coxeter graph of the Coxeter system $({\mathcal{C}}, S)$  associated with $\whw$,  we need to precisely compute the orders of any product $ss'$ for $s,s'\in S$ via the defining sequences corresponding to the standard fundamental system. For this, as mentioned in Remark \ref{remark: compatable of both action}(3),  we need to precisely present the change of initial datum corresponding to the standard fundamental system under the action of the super Weyl group. We will do that in this and next sections. We start by dealing with $\ggg$ of type $D(m|n)$ in this section (from \S\ref{sec: present seq} to \S\ref{sec: type D for CoGr}), tnen deal with $\ggg$ of type $C(m)$  in \S\ref{sec: type C for CoGr}.

}

\subsection{Defining sequences (revisit)}\label{sec: present seq}
Based on the defining sequences introduced in \S\ref{sec: corres type D}, and their correspondence with fundamental systems, we demonstrate  that the correspondence is actually equivariant under the action of super simple reflections in the following.

Recall that the standard fundamental system (\ref{eq: FR first type CD}) which is listed in the following  order
$$\Pi=\{\delta_1-\delta_2,\ldots,\delta_m-\varepsilon_1,
\varepsilon_1-\varepsilon_2,\ldots,
\varepsilon_{n-1}-\vep_n, \vep_n+\vep_{n-1}\}$$
corresponds to the standard defining sequence
    $$(\bar 1,\bar 2, \ldots, \overline  m, 1,\ldots, n).$$
Let us first present the change rule under super simple reflections via examples.

\begin{example} Suppose $n\geq2$.

(1) Defining  sequences under even simple reflection $\widehat r_{\overline  1}:=\widehat r_{2\delta_1}$:
\begin{align*}
\Pi=\{\delta_1-\delta_2, \delta_2-\delta_3,\cdots \cdots\}\quad &{\overset{\widehat r_{\overline 1}}\longrightarrow}\quad\Pi_1= \{-\delta_1-\delta_2,\delta_2-\delta_3,\cdots \cdots\}\cr
(\overline 1,\overline 2, \cdots\cdots,n)\quad \quad                  &\longrightarrow \qquad \quad(-\overline 1,\overline 2,\cdots\cdots,n).
\end{align*}

(2) Defining sequences under even simple reflection $\widehat r_{\overline  {12}}:=\widehat r_{\delta_1-\delta_2}$:
\begin{align*}
\Pi_1= \{-\delta_1-\delta_2,\delta_2-\delta_3,\cdots \cdots\}
\quad &{\overset{\widehat r_{\overline{12}}}{\longrightarrow}}\quad\Pi_2= \{-\delta_2-\delta_1,\delta_1-\delta_3,\cdots \cdots\}\cr
(-\overline 1,\overline 2, \overline 3\cdots\cdots,n)\quad \quad                  &\longrightarrow \qquad \quad(-\overline 2,\overline 1,\overline 3,\cdots\cdots,n)
\end{align*}

(3) defining sequences under even simple reflection $\widehat r_{\overline  {23}}:=\widehat r_{\delta_2-\delta_3}$:
\begin{align*}
\Pi_2= \{-\delta_2-\delta_1,\delta_1-\delta_3,\delta_3-\delta_4\cdots \cdots\}
\quad &{\overset{\widehat r_{\overline{23}}}{\longrightarrow}}\quad\Pi_3= \{-\delta_3-\delta_1,\delta_1-\delta_2,\delta_2-\delta_4\cdots \cdots\}\cr
(-\overline 2,\overline 1, \overline 3,\cdots\cdots,n)\quad \quad                  &\longrightarrow \qquad \quad(-\overline 3,\overline 1,\overline 2,\cdots\cdots,n)
\end{align*}

 (4) defining sequences under odd simple reflection $\widehat r_{\overline m 1}:=\widehat r_{\delta_m-\varepsilon_1}$:
\begin{align*}
\Pi=\{\cdots, \delta_{m-1}-\delta_m, \delta_m-\varepsilon_1,\varepsilon_1-\varepsilon_2,\cdots\}\quad &{\overset{\widehat r_{\overline m 1}}\longrightarrow}\quad\Pi_4= \{\cdots, \delta_{m-1}-\varepsilon_1, \varepsilon_1-\delta_m, \delta_m-\varepsilon_2,\cdots\}\cr
(\cdots\cdots,\overline{m-1},\overline m, 1,2,\cdots\cdots)\quad \quad                  &\longrightarrow  \quad(\cdots\cdots,\overline{m-1},1,\overline m, 2,\cdots\cdots)
\end{align*}

(5) defining sequences under even simple reflection $\widehat r_{12}:=\widehat r_{\varepsilon_1-\varepsilon_2}$:
\begin{align*}
\Pi=\{\cdots, \delta_m-\varepsilon_1,\varepsilon_1-\varepsilon_2,\varepsilon_2-\varepsilon_3,\cdots\}\quad &{\overset{\widehat r_{12}}\longrightarrow}\quad\Pi_5= \{\cdots, \delta_m-\varepsilon_2, \varepsilon_2-\varepsilon_1,\varepsilon_1-\varepsilon_3,\cdots\}\cr
(\cdots\cdots,\overline m, 1,2,3,\cdots\cdots)\quad \quad                  &\longrightarrow  \quad(\cdots\cdots,\overline{m},2,1,3,\cdots\cdots)
\end{align*}

(6)  defining sequences under odd simple reflection ${\widehat r_{n-1,n}}:=\widehat r_{\varepsilon_{n-1}-\varepsilon_{n}}$:
\begin{align*}
\Pi=\{\cdots, \varepsilon_{n-2}-\varepsilon_{n-1},\varepsilon_{n-1}-\varepsilon_{n},
\varepsilon_{n}+\varepsilon_{n-1}\}\quad &{\overset{{\widehat r_{n-1,n}}}\longrightarrow}\quad\Pi_6= \{\cdots, \varepsilon_{n-2}-\varepsilon_{n}, \varepsilon_{n}- \varepsilon_{n-1},
\varepsilon_{n-1}+\varepsilon_{n}\}\cr
(\cdots\cdots, n-2,n-1, n)\quad \quad                  &\longrightarrow  \quad(\cdots\cdots,n-2,n, n-1)
\end{align*}

(7) defining sequences under odd simple reflection ${\widehat r_{n-1,n}}':=\widehat r_{\varepsilon_{n-1}+\varepsilon_{n}}$:
\begin{align*}
\Pi=\{\cdots, \varepsilon_{n-2}-\varepsilon_{n-1},\varepsilon_{n-1}-\varepsilon_{n},
\varepsilon_{n}+\varepsilon_{n-1}\}\quad &{\overset{{\widehat r_{n-1,n}}'}\longrightarrow}\quad\Pi_7= \{\cdots, \varepsilon_{n-2}+\varepsilon_{n}, -\varepsilon_{n}+ \varepsilon_{n-1},
-\varepsilon_{n-1}-\varepsilon_{n}\}\cr
(\cdots\cdots, n-2,n-1, n)\quad \quad                  &\longrightarrow  \quad(\cdots\cdots,n-2,-n, -(n-1))
\end{align*}
\end{example}

\subsubsection{Defining-sequence Theorem}
Note that  any fundamental system comes from the standard fundamental system  (\ref{eq: FR first type CD}) by  a composition of some super simple reflections in $\textsf{S}_{\widetilde \Pi}$ (see Lemma \ref{lem: critical lem}(3) and Proposition \ref{prop: gen system}). Convention {and} Assumption \ref{conv: appendix}, Lemmas \ref{lem: str lem for def seq}  and  \ref{lem: str lem for def seq ordianary case} can be summarized and reformulated  in the formation of  defining sequences, which inductively decide the defining  sequence for any given fundamental systems.

\begin{proposition}\label{app: gen seq} Suppose $n\geq 2$.
 Let the standard fundamental system (\ref{eq: FR first type CD}) correspond to the sequence
    $(\bar 1,\bar 2, \ldots \overline  m, 1,\ldots, n)$. Suppose a fundamental system $\Pi_1$ comes from another fundamental system $\Pi_0$ by a simple reflection $\widehat r_\theta$ for $\theta\in \widetilde\Pi$. The following statements hold.
    \begin{itemize}
    \item[(1)] If $\theta=\delta_m-\varepsilon_1$ and consequently the defining sequence of $\Pi_0$ can be written as

\begin{itemize}
        \item[(Case 1.1)] when $\pm 1$ is in the end of the defining sequence of $\Pi_0$, in the meanwhile the sequence is of the form
        $$(\cdots\cdots,\mp \overline{m},\pm 1),$$
        which means that $\Pi_0$ looks like $(\cdots, \mp(\delta_m+\varepsilon_1), \pm(\vep_1-\delta_m))$,
        then  $\Pi_1$ corresponds to the sequence $$(\cdots\cdots,\mp 1,\mp\overline{m}).$$
         This means that the defining sequence of $\Pi_1$ is obtained by exchanging the positions of $\mp\overline m$ and $\pm1$ of $\Pi_0$, and the leftward $\pm 1$ changes the sign.

      \item[(Case 1.2)] {Otherwise, that is to say} if the defining sequence of $\Pi_0$ is of the form
        $$(\cdots ,\pm \overline{m},\pm 1, \cdots) \text{ or }(\cdots, \pm 1,\pm \overline m, \cdots)$$
        (equivalently,  $\Pi_0$ looks like $(\cdots, \pm(\delta_m-\varepsilon_1), \pm(\varepsilon_1+\heartsuit), \cdots)$ or  $(\cdots, \pm(\varepsilon_1-\delta_m), \pm(\delta_m+\diamondsuit), \cdots)$ with $\heartsuit,\diamondsuit\in \{\pm\vep_i\mid i\in I(m|n)\}$ such that $\varepsilon_1+\heartsuit$ and $\delta_m+\diamondsuit$ are roots),
        then $\Pi_1$ corresponds to the sequence $$(\cdots ,\pm 1,\pm \overline{m},\cdots) \text{ or }(\cdots, \pm \overline m, \pm1,\cdots),$$
         respectively.  This means that we only exchange the positions of $\pm\overline m$ and $\pm1$, all others do not change.
        \end{itemize}

    \item[(2)] Suppose  $\theta=\varepsilon_{\widetilde{i}}-\varepsilon_{\widetilde{i+1}}$ for $\widetilde{i}\in \{\bar 1,\ldots, \overline {m-1}\}$ or $\widetilde{i}\in \{1,\ldots,n-1\}$
    and consequently the defining sequence of $\Pi_0$ can be written as
\begin{itemize}
\item[(Case 2.1)]
       $(\cdots ,\pm\widetilde{i},\cdots,\pm \widetilde{(i+1)}, \cdots) \text{ or }(\cdots, \pm\widetilde{(i+1)},\cdots,\pm\widetilde{i}, \cdots)$; or
  \item[(Case 2.2)] $(\cdots ,\pm\widetilde{i},\cdots,\mp \widetilde{(i+1)}, \cdots) \text{ or }(\cdots, \mp\widetilde{(i+1)},\cdots,\pm\widetilde{i}, \cdots)$.
\end{itemize}

        Then in (Case 2.1), $\Pi_1$ corresponds to the sequence $$(\cdots ,\pm \widetilde{(i+1)},\cdots,\pm \widetilde{i},\cdots) \text{ or }(\cdots, \pm\widetilde{i},\cdots,\pm\widetilde{(i+1)}, \cdots),$$
         respectively,  which means that we only exchange the positions of $\pm\widetilde{i}$ and $\pm\widetilde{(i+1)}$, all others do not change.
          In (Case 2.2),  $\Pi_1$ corresponds to the sequence $$(\cdots ,\pm \widetilde{(i+1)},\cdots,\mp \widetilde{i},\cdots) \text{ or }(\cdots, \mp\widetilde{i},\cdots,\pm\widetilde{(i+1)}, \cdots),$$
         respectively,  which means that we exchange the positions of $\pm\widetilde{i}$ and $\mp\widetilde{(i+1)}$, and simultaneously change both signs, all others do not change.

\item[(3)] If $\theta=\varepsilon_{n-1}+\varepsilon_n$,  and consequently the defining sequence of $\Pi_0$ can be written as $(\cdots,\pm (n-1),\pm n,\cdots)$ or $(\cdots,\pm n,\pm(n-1),\cdots)$, then $\Pi_1$ corresponds to the sequence $(\cdots,\mp n,\mp(n-1),\cdots)$ or $(\cdots,\mp (n-1),\mp n,\cdots)$, respectively.  This means that we change the positions of $\pm(n-1)$ and $\pm n$, simultaneously change both signs,   all others do not change.

    \item[(4)] If $\widehat r_\theta= \widehat r_{-2\delta_1}= \widehat r_{2\delta_1}$ and consequently the defining sequence of $\Pi_0$ can be written as $(\cdots,\pm \bar1,\cdots)$, then   $\Pi_1$ corresponds to the sequence $(\cdots,\mp \bar 1,\cdots)$.  This means that we only change the sign of $\pm\bar 1$,   all others do not change.
        \end{itemize}
\end{proposition}

\begin{remark}
  $\Pi_0$ could look like  $(\cdots, \mp(\delta_m+\varepsilon_1), \pm(\vep_1-\delta_m))$ as in (Case 1.1) of the above proposition.  But the adjacent pair $ \mp(\delta_m+\vep_1), \pm(\vep_1-\delta_m)$ could not  occur in middle positions of any ordered fundamental system $\Pi_0$. This is because by Lemmas \ref{lem: str lem for def seq}  and  \ref{lem: str lem for def seq ordianary case}, $\Pi_0$ must be of form
(\ref{eq: general fs form}). If so,  the root after $\pm(\vep_1-\delta_m)$ must have the front term $\pm \delta_m$, contradicting the convention for (\ref{eq: general fs form}). The convention does not  allow the same front term to appear twice in an ordered fundamental system.
\end{remark}

\subsubsection{One-to-one correspondence}
Now we are in a position to introduce a key result which is a reformulation of Lemma \ref{lem: defining sec for D}.
\begin{theorem}\label{prop: presenting seq type D} (Defining-Sequence Theorem) The set of fundamental systems are in a one-to-one correspondence with the set of defining sequences. {Such a one-to-one correspondence is $\widehat W$-equivariant in the sense that  it is compatible with  $\widehat W$-action.}
\end{theorem}

{
\begin{proof} The first part of the theorem follows from Assumption \ref{conv: appendix}, Corollary \ref{lem: str lem for def seq summing} and Proposition \ref{app: gen seq}.  The compatibility of $\sfd$  with the $\widehat W$-action is explained in Remark \ref{remark: compatable of both action}(2).
\end{proof}
}

\begin{remark}\label{rem: super W-equivariant}    The statement of Theorem \ref{prop: presenting seq type D}  is valid too for type $C(m)$ and $B(m;n)$ because of the corresponding defining-sequence construction (see  Theorems \ref{prop: presenting seq type C}, and \ref{prop: presenting seq type B}).
\end{remark}

\subsection{Weyl group for $\mathfrak{spo}(2m|2n)$ ($n\geq 2$) and its Coxeter graph}\label{sec: 3.3}
By definition, the Weyl group of $\mathfrak{spo}(2m|2n)$ with $n\geq2$, is exactly the Weyl group of $\mathfrak{g}_{\bar{0}}=\mathfrak{sp}(2m)\oplus\mathfrak{so}(2n)$ which is isomorphic to $(\mathbb{Z}_{2}^m\rtimes\mathfrak{S}_{m})\times(\mathbb{Z}_{2}^n\rtimes
\mathfrak{S}_{n})$.

Keeping the extended standard fundamental system of $\mathfrak{spo}(2m|2n)$ in mind (see around (\ref{eq: FR first type CD})), we set
\begin{align*}
&\widehat r_{\bar 1}:=\widehat r_{-2\delta_1}=\widehat r_{2\delta_1}, \cr
&\widehat r_{\overline{ i,i+1}}:=\widehat r_{\delta_i-\delta_{i+1}},\cr
&\widehat r_{{ i,i+1}}:=\widehat r_{\varepsilon_j-\varepsilon_{j+1}}
\end{align*}
and
\begin{align*}
&\widehat r_{\bar{i}j} := \widehat r_{\delta_i-\varepsilon_j}, \cr
&r'_{ij} :=\widehat r_{\varepsilon_i+\varepsilon_j}.
\end{align*}
The Coxeter graph of the Weyl group for $\mathfrak{spo}(2m|2n)$ consist {of two disconnected parts} as below (see for example, \cite[Theorem 1.3.3]{GP}):

\begin{center}
	\begin{tikzpicture}
		
		\node[circle,draw,minimum size=8pt,inner sep=0pt,fill=white] (r1) at (-4,1) {};

		\node[circle,draw,minimum size=8pt,inner sep=0pt,fill=white] (r2) at (-3,1) {};
		
		\node[circle,draw,minimum size=8pt,inner sep=0pt,fill=white] (r3) at (-2,1) {};
		
		\node[circle,draw,minimum size=8pt,inner sep=0pt,fill=white] (r4) at (-1,1) {};

		\draw (r1) -- node[above=2pt] {\footnotesize 4} (r2);
		\draw (r2) -- (r3);
		\node (sl)at(-1.5,1){ $\cdots$};
		
		\node[below=3pt] at (r1) {$\widehat r_{\bar 1}$};
		\node[below=3pt] at (r2) {$\widehat r_{\overline{1,2}}$};
		\node[below=3pt] at (r3) {$\widehat r_{\overline{ 2,3}}$};
		\node[below=3pt] at (r4) {$\widehat r_{\overline{m-1,m}}$};

	\end{tikzpicture}
\end{center}

and

\begin{center}
	\begin{tikzpicture}
		
		\node[circle,draw,minimum size=8pt,inner sep=0pt,fill=white] (r1) at (0,1) {};

		\node[circle,draw,minimum size=8pt,inner sep=0pt,fill=white] (r2) at (2,1) {};
		
		\node[circle,draw,minimum size=8pt,inner sep=0pt,fill=white] (r3) at (4,1) {};
		
		\node[circle,draw,minimum size=8pt,inner sep=0pt,fill=white] (r4) at (6,1) {};
		
		\node[circle,draw,minimum size=8pt,inner sep=0pt,fill=white] (r5) at (4,2) {};

		\draw (r1) -- (r2);
		\draw (r3) -- (r4);
		\draw (r3) -- (r5);
		\node at(2.5,1){ $\cdots$};
		
		\node[below,sloped] at (r1) {$\widehat r_{1 2}$};
		\node[below,sloped] at (r2) {$\widehat r_{23}$};
		\node[below,sloped] at (r3) {$\widehat r_{n-2, n-1}$};
		\node[below,sloped] at (r4) {$\widehat r_{n-1,n}$};
		\node[below,sloped] at (r5) {$\widehat r'_{n-1,n}$};

		
	\end{tikzpicture}
\end{center}

\subsection{Super Weyl group for $\mathfrak{spo}(2m|2n)$ with $n\geq 2$ and its Coxeter graph}\label{sec: type D for CoGr}

Keep the conventions as before (for example Remark \ref{rem: white and black}). Based on Proposition \ref{app: gen seq} and Defining-sequence Theorem (Theorem \ref{prop: presenting seq type D}),  the Coxeter graph for $\mathfrak{spo}(2m|2n)$ can be described below.

\begin{theorem}\label{Thm5.3}
	The Coxeter graph of $\widehat{\mathcal{C}}(\mathfrak{spo}(2m|2n))$ with {$n\geq2$}  is
	
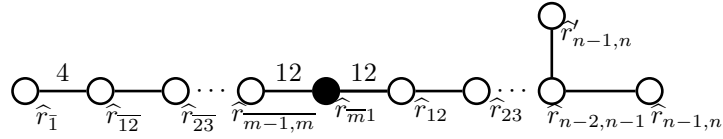
\begin{center}
	\begin{tikzpicture}
		
		\node[circle,draw,minimum size=8pt,inner sep=0pt,fill=white] (r1) at (-4,1) {};
		
		\node[circle,draw,minimum size=8pt,inner sep=0pt,fill=white] (r2) at (-3,1) {};
		
		\node[circle,draw,minimum size=8pt,inner sep=0pt,fill=white] (r3) at (-2,1) {};
		
		\node[circle,draw,minimum size=8pt,inner sep=0pt,fill=white] (r4) at (-1,1) {};
		
		\node[circle,draw,minimum size=8pt,inner sep=0pt,fill=white] (r5) at (0,1) {};
		\draw (r5.135) -- (r5.315);
		\draw(r5.225) -- (r5.45);

		\node[circle,draw,minimum size=8pt,inner sep=0pt,fill=white] (r6) at (1,1) {};
		
		\node[circle,draw,minimum size=8pt,inner sep=0pt,fill=white] (r7) at (2,1) {};
		
		\node[circle,draw,minimum size=8pt,inner sep=0pt,fill=white] (r8) at (3,1) {};
		
		\node[circle,draw,minimum size=8pt,inner sep=0pt,fill=white] (r9) at (5,1) {};
		
		\node[circle,draw,minimum size=8pt,inner sep=0pt,fill=white] (r10) at (3,2) {};

		\draw (r1) -- node[above=2pt] {\footnotesize 4} (r2);
		\draw (r2) -- (r3);
		\node at(-1.5,1){ $\cdots$};
		\draw (r4) -- node[above=2pt] {\footnotesize 12}(r5);
		\draw (r5) -- node[above=2pt] {\footnotesize 12}(r6);
		\draw (r6) -- (r7);
		\node at(2.5,1){ $\cdots$};
		\draw (r8) -- (r9);
		\draw (r8) -- (r10);
		
		\node[below=3pt] at (r1) {$\widehat r_{\bar 1}$};
		\node[below=3pt] at (r2) {$\widehat r_{\overline {12}}$};
		\node[below=3pt] at (r3) {$\widehat r_{\overline{23}}$};
		\node[below=3pt] at (r4) {$\widehat r_{\overline{m-1,m}}$};
		\node[below=3pt] at (r5) {$\widehat r_{\overline{m} 1}$};
		\node[below=3pt] at (r6) {$\widehat r_{1 2}$};
		\node[below=3pt] at (r7) {$\widehat r_{23}$};
		\node[below=3pt] at (r8) { $\widehat r_{n-2, n-1}$};
		\node[below=3pt] at (r9) {$\widehat r_{n-1,n}$};
		\node[below=3pt] at (r10) {$\widehat r'_{n-1,n}$};

	\end{tikzpicture}
\end{center}

\end{theorem}

\begin{proof}
Note that the Coxeter graph of  the subsystem $(\scrw, \spo)$ is already known (see \cite[Theorem 1.3.3]{GP}). It suffices for us to compute $m_{xy}$ involving $\widehat r_{\overline{m}1}$ as  introduced in  \S\ref{sec: fun sys} for Proposition \ref{prop: gen system}. Keeping the notations and statements in \S\ref{sec: 3.3}, we proceed with arguments on the super reflections involving $\widehat r_{\overline{m}1}$.
By the same reasoning as in the part (1) of the proof for Theorem \ref{Cox glmn}, we can see that $\widehat r_{\overline{m}1}$ commutes with all $\widehat r_{\overline{i,i+1}}$ and $\widehat r_{j,j+1}$ except $i+1=m$ or $j=1$. So we only need to compute the orders of $\widehat r_{\overline{m-1,m}} \widehat r_{\overline m1}$, and of $\widehat r_{\overline m1}\widehat r_{12}$. In the following, we give computation by considering their action in all nontrivial cases.
	
(1)	 For the reflection $\widehat r_{12}$, we calculate the order of $\widehat r_{12}\cdot \widehat r_{\overline{m}1}$ by the same argument as in the proof of Theorem \ref{Cox glmn}. {Basing upon} Proposition \ref{app: gen seq} and Defining-sequence Theorem, we divide it into the following four cases (in the arguments we don't care the position of $2$ (or $-2$) in a given sequence):
	
	(Case 1.1) Consider the case when  $\overline{m}$ and  $1$ appear adjacently in the sequence.  We have
 \begin{align*}
&(\cdots,\overline{m},1,2,\cdots)\overset
	{\widehat r_{\overline{m}1}}{\longrightarrow}(\cdots, 1,\overline{m},2,\cdots)\overset
	{\widehat r_{12}}{\longrightarrow}\cr
&(\cdots,2,\overline{m},1,\cdots)\overset
	{\widehat r_{\overline{m}1}}{\longrightarrow}(\cdots,2, 1,\overline{m},\cdots)\overset
	{\widehat r_{12}}{\longrightarrow}\cr
&(\cdots,1, 2,\overline{m},\cdots)\overset
	{\widehat r_{\overline{m}1}}{\longrightarrow}(\cdots,1, 2,\overline{m},\cdots)\overset
	{\widehat r_{12}}{\longrightarrow}\cr
&(\cdots,2, 1,\overline{m},\cdots)\overset
	{\widehat r_{\overline{m}1}}{\longrightarrow}(\cdots,2,\overline{m},1,\cdots)\overset
	{\widehat r_{12}}{\longrightarrow}\cr
&(\cdots,1,\overline{m},2,\cdots)\overset
	{\widehat r_{\overline{m}1}}{\longrightarrow}(\cdots,\overline{m},1,2, \cdots)\overset
	{\widehat r_{12}}{\longrightarrow}\cr
&(\cdots,\overline{m},2,1,\cdots)\overset
	{\widehat r_{\overline{m}1}}{\longrightarrow}(\cdots,\overline{m},2,1,\cdots)\overset
	{\widehat r_{12}}{\longrightarrow}\cr
&(\cdots,\overline{m},1,2,\cdots),
\end{align*}
{ which shows that the order of the action of $\widehat r_{12}\cdot \widehat r_{\overline{m}1}$ on such sequences is $6$,  and
 \begin{align*}
&(\cdots,\overline{m},1,-2,\cdots)\overset
	{\widehat r_{\overline{m}1}}{\longrightarrow}(\cdots, 1,\overline{m},-2,\cdots)\overset
	{\widehat r_{12}}{\longrightarrow}\cr
&(\cdots, 2,\overline{m}, -1,\cdots)\overset
	{\widehat r_{\overline{m}1}}{\longrightarrow}(\cdots,2, \overline{m},-1\cdots)\overset
	{\widehat r_{12}}{\longrightarrow}\cr
&(\cdots,1,\overline{m},-2\cdots)\overset
	{\widehat r_{\overline{m}1}}{\longrightarrow}(\cdots, \overline{m},1,-2\cdots)\overset
	{\widehat r_{12}}{\longrightarrow}\cr
&(\cdots,\overline{m}, 2,-1\cdots)\overset
	{\widehat r_{\overline{m}1}}{\longrightarrow}(\cdots,\overline{m},2,-1,\cdots)\overset
	{\widehat r_{12}}{\longrightarrow}\cr
&(\cdots,\overline{m},1,-2,\cdots)
\end{align*}
 which shows that  the order of the action of $\widehat r_{12}\cdot \widehat r_{\overline{m}1}$ on such sequences  is $4$.}

	(Case 1.2) Consider the case when  $-\overline{m}$  and $-1$ appear  adjacently in the middle of the sequence. We have
\begin{align*}
&	 (\cdots,2, -\overline{m},-1,\cdots)\overset
	{\widehat r_{\overline{m}1}}{\longrightarrow} (\cdots,2, -1,-\overline{m},\cdots)\overset
	{\widehat r_{12}}{\longrightarrow}\cr
&(\cdots, 1, -2,-\overline{m},\cdots)\overset
	{\widehat r_{\overline{m}1}}{\longrightarrow}(\cdots,1, -2,-\overline{m},\cdots)\overset
	{\widehat r_{12}}{\longrightarrow}\cr
&(\cdots,2,-1, -\overline{m},\cdots)\overset
	{\widehat r_{\overline{m}1}}{\longrightarrow}(\cdots,2, -\overline{m}, -1,\cdots)\overset
	{\widehat r_{12}}{\longrightarrow}\cr
&(\cdots, 1,-\overline{m},-2,\cdots)\overset
	{\widehat r_{\overline{m}1}}{\longrightarrow}(\cdots, 1,-\overline m,-2,\cdots) \overset
	{\widehat r_{12}}{\longrightarrow}\cr
&(\cdots,2,-\overline{m},-1,\cdots),
 \end{align*}
  which shows that  the order of the action of $\widehat r_{12}\cdot \widehat r_{\overline{m}1}$ on such sequences is $4$, and

\begin{align*}
&	 (\cdots,-2, -\overline{m},-1,\cdots)\overset
	{\widehat r_{\overline{m}1}}{\longrightarrow} (\cdots,-2, -1,-\overline{m},\cdots)\overset
	{\widehat r_{12}}{\longrightarrow}\cr
&(\cdots, -1, -2,-\overline{m},\cdots)\overset
	{\widehat r_{\overline{m}1}}{\longrightarrow}(\cdots,-1, -2,-\overline{m},\cdots)\overset
	{\widehat r_{12}}{\longrightarrow}\cr
&(\cdots,-2,-1, -\overline{m},\cdots)\overset
	{\widehat r_{\overline{m}1}}{\longrightarrow}(\cdots, -2, -\overline{m}, -1,\cdots)\overset
	{\widehat r_{12}}{\longrightarrow}\cr
&(\cdots, -1,-\overline{m},-2,\cdots)\overset
	{\widehat r_{\overline{m}1}}{\longrightarrow}(\cdots, -\overline m,-1, -2,\cdots) \overset
	{\widehat r_{12}}{\longrightarrow}\cr
&(\cdots,-\overline{m},-2,-1,\cdots)\overset
	{\widehat r_{\overline{m}1}}{\longrightarrow} (\cdots,-\overline{m},-2, -1,\cdots)\overset
	{\widehat r_{12}}{\longrightarrow}\cr
&(\cdots, -\overline{m},-1, -2,\cdots)\overset
	{\widehat r_{\overline{m}1}}{\longrightarrow}(\cdots, -1, -\overline{m},-2,\cdots)\overset
	{\widehat r_{12}}{\longrightarrow}\cr
&(\cdots,-2, -\overline{m},-1,\cdots),
 \end{align*}
which shows that  the order of the action of $\widehat r_{12}\cdot \widehat r_{\overline{m}1}$ on such sequences is $6$.

(Case 1.3) Consider the case when the sequence looks like $(\ldots,  \mp\overline m, \pm1)$,  corresponding to (Case 1.1) of Proposition \ref{app: gen seq}, say
$(\ldots, -\overline m, 1)$. We have
\begin{align}
&	 (\cdots,2, -\overline{m},1)\overset
	{\widehat r_{\overline{m}1}}{\longrightarrow} (\cdots,2, -1,-\overline{m})\overset
	{\widehat r_{12}}{\longrightarrow}\cr
&(\cdots, 1, -2,-\overline{m})\overset
	{\widehat r_{\overline{m}1}}{\longrightarrow}(\cdots,1, -2,-\overline{m})\overset
	{\widehat r_{12}}{\longrightarrow}\cr
&(\cdots,2,-1, -\overline{m})\overset
	{\widehat r_{\overline{m}1}}{\longrightarrow}(\cdots,2, -\overline{m}, 1)\overset
	{\widehat r_{12}}{\longrightarrow}\cr
&(\cdots, 1,-\overline{m},2)\overset
	{\widehat r_{\overline{m}1}}{\longrightarrow}(\cdots, 1,-\overline m,2) \overset
	{\widehat r_{12}}{\longrightarrow}\cr
&(\cdots,2,-\overline{m},1),
 \end{align}
  which shows that  the order of the action of $\widehat r_{12}\cdot \widehat r_{\overline{m}1}$ on such sequences is $4$, and
   \begin{align*}
&	 (\cdots,-2, -\overline{m},1)\overset
	{\widehat r_{\overline{m}1}}{\longrightarrow} (\cdots,-2, -1,-\overline{m})\overset
	{\widehat r_{12}}{\longrightarrow}\cr
&(\cdots, -1, -2,-\overline{m})\overset
	{\widehat r_{\overline{m}1}}{\longrightarrow}(\cdots,-1, -2,-\overline{m})\overset
	{\widehat r_{12}}{\longrightarrow}\cr
&(\cdots,-2,-1, -\overline{m})\overset
	{\widehat r_{\overline{m}1}}{\longrightarrow}(\cdots, -2, -\overline{m}, 1)\overset
	{\widehat r_{12}}{\longrightarrow}\cr
&(\cdots, -1,-\overline{m},2)\overset
	{\widehat r_{\overline{m}1}}{\longrightarrow}(\cdots, -\overline m,-1,2) \overset
	{\widehat r_{12}}{\longrightarrow}\cr
&(\cdots,-\overline{m},-2,1)\overset
	{\widehat r_{\overline{m}1}}{\longrightarrow}(\cdots,-\overline{m},-2,1)\overset
	{\widehat r_{12}}{\longrightarrow}\cr
&(\cdots,-\overline{m},-1,2)\overset
	{\widehat r_{\overline{m}1}}{\longrightarrow}(\cdots,-1, -\overline m,2)\overset
	{\widehat r_{12}}{\longrightarrow}\cr
&(\cdots,-2,-\overline{m},1)
 \end{align*}
 which shows that  the order of the action of $\widehat r_{12}\cdot \widehat r_{\overline{m}1}$ on such sequences is $6$.

	(Case 1.4) Consider the case when $\overline{m}$ is not adjacent to $1$ and $2$ in the defining sequence. We have
%
\begin{align*}
&(\cdots,\overline{m}, 3, \pm 2,\pm1,\cdots)\overset
	{\widehat r_{\overline{m}1}}{\longrightarrow}(\cdots,\overline{m}, 3, \pm2,\pm1,\cdots)\overset
	{\widehat r_{12}}{\longrightarrow} \cr
&(\cdots,\overline{m},  3, \pm1, \pm2,\cdots)\overset
	{\widehat r_{\overline{m}1}}{\longrightarrow}(\cdots,\overline{m},  3,  \pm1,\pm 2,\cdots)\overset
	{\widehat r_{12}}{\longrightarrow}\cr
&(\cdots, \overline{m}, 3, \pm2, \pm1,\cdots),
\end{align*}
	which shows that  the order of the action of $\widehat r_{12}\cdot \widehat r_{\overline{m}1}$ on such sequences  is $2$, and
\begin{align*}
&(\cdots,\overline{m}, 3, \pm 2,\mp1,\cdots)\overset
	{\widehat r_{\overline{m}1}}{\longrightarrow}(\cdots,\overline{m}, 3, \pm2,\mp1,\cdots)\overset
	{\widehat r_{12}}{\longrightarrow} \cr
&(\cdots,\overline{m},  3, \pm1, \mp2,\cdots)\overset
	{\widehat r_{\overline{m}1}}{\longrightarrow}(\cdots,\overline{m},  3,  \pm1,\mp 2,\cdots)\overset
	{\widehat r_{12}}{\longrightarrow}\cr
&(\cdots, \overline{m}, 3, \pm2, \mp1,\cdots),
\end{align*}
 which  shows that  the order of the action of $\widehat r_{12}\cdot \widehat r_{\overline{m}1}$ on such sequences   is $2$.

Recall that there are $\frac{n}{m+n}2^{m+n-1}\cdot(m+n)!+\frac{m}{m+n}2^{m+n}\cdot(m+n)!= (n+2m)2^{m+n-1}\cdot(m+n-1)!$ possible fundamental systems (cf. \S\ref{S5.2}).
 When $m+n\geq 3$, all (Cases 1.1-1.4) necessarily  occur on some sequences. So the least common multiple of the orders $6$, $4$ and $2$ respectively occurring for different cases must be exactly the order of $\widehat r_{12}\cdot \widehat r_{\overline{m}1}$, which is equal to $12$.
		
(2) 	Next calculate the order of $\widehat r_{\overline{m-1,m}}\cdot \widehat r_{\overline{m}1}$. By the same arguments as the ones for $\widehat r_{12}\cdot \widehat r_{\overline{m}1}$,  we can compute the order of $\widehat r_{\overline{m-1,m}}\cdot \widehat r_{\overline{m}1}$ which is $12$.	 Summing up, the proof is completed.
\end{proof}

\subsection{Coxeter graph for $\mathfrak{spo}(2m|2)$  (type $C(m)$)}\label{sec: type C for CoGr}
 As mentioned in \S\ref{sec: fun sys}(iii), the Lie superalgebra $\mathfrak{spo}(2m|2)$ has the fundamental system
 \begin{align}\label{eq: FR first type C}
 \Pi=\{\vep_1-\delta_1,\delta_1-\delta_2,\ldots, \delta_{m-1}-\delta_m, 2\delta_m\}
 \end{align}
    which coincides with the extended standard fundamental system  $\widetilde\Pi$.  We have to deal with this case separately. We first look at  the case  $\mathfrak{spo}(2|2)$ which is different from others.

\subsubsection{$\mathfrak{spo}(2|2)$}\label{ex: spo22}
	For $\mathfrak{spo}(2|2)$, the root system is $\Phi=\{\pm\delta_1\pm\varepsilon_1, \pm\delta_1\}$. As $\widehat r_{\delta_1+\varepsilon_{1}} =\widehat r_{2\delta_1}\cdot \widehat r_{\delta_1-\varepsilon_{1}}\cdot \widehat r_{2\delta_1}$, it is enough to   consider  $\widehat r_{\delta_1-\varepsilon_{1}}$ and $\widehat r_{2\delta_1}$.
Set $\widehat r_{\bar{1}1}:=\widehat r_{\delta_1-\varepsilon_{1}}$ and $\widehat r_{\bar{1}}=\widehat r_{2\delta_1}$.	
There are  $6$ different fundamental systems as follows:
\begin{align*}	
&\{ \varepsilon_1-\delta_1, 2\delta_1\}{\overset{\widehat r_{\bar11}}{\longleftrightarrow}}
\{\delta_1-\varepsilon_1, \varepsilon_1+\delta_1\}{\overset{\widehat r_{\bar1}}{\longleftrightarrow}}
\{-\delta_1 -\varepsilon_1, \varepsilon_1-\delta_1\}{\overset{\widehat r_{\bar11}}{\longleftrightarrow}}\cr
{\longleftrightarrow}  &\{ -2\delta_1, \delta_1-\varepsilon_1\}{\overset{\widehat r_{\bar1}}{\longleftrightarrow}}
	\{2\delta_1, -\delta_1-\varepsilon_1\}{\overset{\widehat r_{\bar1}(\widehat r_{\bar11}\widehat r_{\bar1})^2}{\longleftrightarrow}}
	\{\varepsilon_1+\delta_1, -2\delta_1\}
		\end{align*}
which are indexed by $\{1,2,\ldots, 6\}$ in the above order. Here ${\overset{\widehat r_{\bar1}(\widehat r_{\bar11}\widehat r_{\bar1})^2}{\longleftrightarrow}}
	\{\varepsilon_1+\delta_1, -2\delta_1\}$ is equivalent to
$$ \{ \varepsilon_1-\delta_1, 2\delta_1\}{\overset{\widehat r_{\bar1}}{\longrightarrow}}\{\varepsilon_1+\delta_1, -2\delta_1\}.$$

Similar to the arguments in the case of $\mathfrak{gl}(1|2)$, we have that $\widehat r_{\bar{1}1} = (12)(34)$, $\widehat r_{\bar{1}} = (16)(23)(45)$, and	$\widehat r_{\bar{1}1}\widehat r_{\bar{1}} = (162453)$. So the order of $\widehat r_{\bar{1}1}\widehat r_{\bar{1}}$ is $6$.
Then it is easily seen that $\widehat{W}(\mathfrak{spo}(2|2))\cong G_2$, which is exactly the dihedral group of order $12$. The Coxeter graph of $\wsc(\mathfrak{spo}(2|2))$ is	
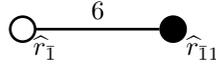
\begin{center}
	\begin{tikzpicture}
		
		\node[circle,draw,minimum size=8pt,inner sep=0pt,fill=white] (r1) at (-2,1) {};
		
		\node[circle,draw,minimum size=8pt,inner sep=0pt,fill=white] (r2) at (0,1) {};
		\draw (r2.135) -- (r2.315);
		\draw(r2.225) -- (r2.45);

		\draw (r1) -- node[above=2pt] {\footnotesize 6} (r2);

		\node[below=3pt] at (r1) {$\widehat r_{\bar{1}}$};
		\node[below=3pt] at (r2) { $\widehat r_{\bar{1}1}$};
		
	\end{tikzpicture}
\end{center}
%

\subsubsection{$\mathfrak{spo}(2m|2)$ ($m\geq2$)}
Keep in mind that the fundamental root system $\Pi$ (coinciding with the extend one $\widetilde \Pi$) is presented as in (\ref{eq: FR first type C}). Now consider the general $\mathfrak{spo}(2m|2)$ ($m\geq2$).
Still set
$\widehat r_{\overline m}:=\widehat r_{-2\delta_m}=\widehat r_{2\delta_m},
\widehat r_{\overline{ i,i+1}}:=\widehat r_{\delta_i-\delta_{i+1}}$
 and
$\widehat r_{\overline{1}1} := \widehat r_{ \delta_{1}-\varepsilon_1}$.

 Summarizing Lemmas \ref{lem: lemma 2.8 for type C} and
 \ref{lem: lemma 2.9 for type C} and reformulating them,  we have the following result which is parallel to  Proposition \ref{app: gen seq} with minor modifications.


\begin{proposition}\label{app: type C seq}
Suppose $\ggg=\mathfrak{spo}(2m|2)$ with $m>1$, and make the standard fundamental system (\ref{eq: FR first type C}) correspond to the sequence
    $(1,\bar 1,\bar 2, \ldots, \overline m)$.
Suppose a fundamental system $\Pi_1$ comes from another fundamental system $\Pi_0$ by a simple reflection $\widehat r_\theta$ for $\theta\in \Pi$,
    \begin{itemize}
    \item[(1)] if $\theta=\vep_1-\delta_1$
    \begin{itemize}
        \item[(Case 1.1)] when $\pm 1$ is in the end of the defining sequence of $\Pi_0$, in the meanwhile the sequence is of the form
        $$(\cdots\cdots,\mp \overline{1},\pm 1),$$
        which means that $\Pi_0$ looks like $(\cdots, \mp(\delta_1+\varepsilon_1), \pm(\vep_1-\delta_1))$,
        then  $\Pi_1$ corresponds to the sequence $$(\cdots\cdots,\mp 1,\mp\overline{1}).$$
         This means that the defining sequence of $\Pi_1$ is obtained by exchanging the positions of $\mp\overline 1$ and $\pm1$ of $\Pi_0$, and the leftward $\pm 1$ changes the sign.

   \item[(Case 1.2)] Otherwise, equivalently the defining sequence of $\Pi_0$ can be written as
        $$(\cdots ,\pm \overline{1},\pm 1, \cdots) \text{ or }(\cdots, \pm 1,\pm \overline 1, \cdots),$$
        which means that $\Pi_0$ looks like $(\cdots, \pm(\delta_1-\varepsilon_1), \pm(\varepsilon_1+\heartsuit), \cdots)$ or  $(\cdots, \pm(\varepsilon_1-\delta_1), \pm(\delta_1+\diamondsuit), \cdots)$ with $\heartsuit,\diamondsuit\in \{\pm\vep_i\mid i\in I(m|n)\}$ such that $\varepsilon_1+\heartsuit$ and $\delta_1+\diamondsuit$ are roots,
        then  $\Pi_1$ corresponds to the sequence $$(\cdots ,\pm 1,\pm \overline{1},\cdots) \text{ or }(\cdots, \pm \overline 1, \pm1,\cdots),$$
         respectively.  This means that we only exchange the positions of $\pm\overline m$ and $\pm1$, all others do not change.
        \end{itemize}

    \item[(2)] Suppose  $\theta=\delta_{{i}}-\delta_{{i+1}}$
    and consequently the defining sequence of $\Pi_0$ can be written as
\begin{itemize}
\item[(Case 2.1)]
       $(\cdots ,\pm\bar{i},\cdots,\pm \overline{(i+1)}, \cdots) \text{ or }(\cdots, \pm\overline{(i+1)},\cdots,\pm\bar{i}, \cdots)$; or
  \item[(Case 2.2)] $(\cdots ,\pm\bar{i},\cdots,\mp \overline{(i+1)}, \cdots) \text{ or }(\cdots, \mp\overline{(i+1)},\cdots,\pm\bar{i}, \cdots)$.
\end{itemize}

        Then in (Case 2.1),  $\Pi_1$ corresponds to the sequence $$(\cdots ,\pm \overline{(i+1)},\cdots,\pm \bar{i},\cdots) \text{ or }(\cdots, \pm\bar{i},\cdots,\pm\overline{(i+1)}, \cdots),$$
         respectively,  which means that we only exchange the positions of $\pm\bar{i}$ and $\pm\overline{(i+1)}$, all others do not change;
          In (Case 2.2),  $\Pi_1$ corresponds to the sequence $$(\cdots ,\pm \overline{(i+1)},\cdots,\mp \bar{i},\cdots) \text{ or }(\cdots, \mp\bar{i},\cdots,\pm\overline{(i+1)}, \cdots),$$
         respectively,  which means that we exchange the positions of $\pm\bar{i}$ and $\mp\overline{(i+1)}$, and simultaneously change both signs, all others do not change.

    \item[(3)] If $\widehat r_\theta= \widehat r_{-2\delta_m}= \widehat r_{2\delta_m}$ and consequently the defining sequence of\, $\Pi_0$ can be written as $(\cdots,\pm \overline m,\cdots)$, then $\Pi_1$ corresponds to the sequence $(\cdots,\mp \overline m,\cdots)$.  This means that we only change the sign of $\pm\overline m$,   all others do not change.
        \end{itemize}
\end{proposition}

Then as mentioned in Remark \ref{rem: super W-equivariant}, we have a correspondence between fundamental systems and their defining sequences for type $C$.

\begin{theorem}\label{prop: presenting seq type C} (Defining-sequence Theorem) The set of fundamental systems of $\ggg=\mathfrak{spo}(2m|2)$ are in a one-to-one correspondence with the set of defining sequences.  {Such a one-to-one correspondence is $\widehat W$-equivariant in the sense that  it is compatible with  $\widehat W$-action.}
\end{theorem}

{
\begin{proof} It follows from the similar arguments to Theorem \ref{prop: presenting seq type D}.
\end{proof}
}

\subsubsection{} Based on Proposition \ref{app: type C seq} and Theorem \ref{prop: presenting seq type C}, the Coxeter graph for $C(m)$ can be presented below.
\begin{theorem} \label{thm: type C}	
The Coxeter graph of  $\wsc$ for $\ggg=\mathfrak{spo}(2m|2)$ ($m\geq2$) is

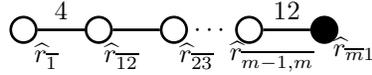
\begin{center}
	\begin{tikzpicture}
		
		\node[circle,draw,minimum size=8pt,inner sep=0pt,fill=white] (r1) at (-4,1) {};
		
		\node[circle,draw,minimum size=8pt,inner sep=0pt,fill=white] (r2) at (-3,1) {};
		
		\node[circle,draw,minimum size=8pt,inner sep=0pt,fill=white] (r3) at (-2,1) {};
		
		\node[circle,draw,minimum size=8pt,inner sep=0pt,fill=white] (r4) at (-1,1) {};
		
		\node[circle,draw,minimum size=8pt,inner sep=0pt,fill=white] (r5) at (0,1) {};
		\draw (r5.135) -- (r5.315);
		\draw(r5.225) -- (r5.45);

		\draw (r1) -- node[above=2pt] {\footnotesize 4} (r2);
		\draw (r2) -- (r3);
		\node at(-1.5,1){ $\cdots$};
		\draw (r4) -- node[above=2pt] {\footnotesize 12} (r5);
		
		\node[below=3pt] at (r1) {$\widehat r_{\overline 1}$};
		\node[below=3pt] at (r2) {$\widehat r_{\overline{12}}$};
		\node[below=3pt] at (r3) {$\widehat r_{ \overline {23}}$};
		\node[below=3pt] at (r4) {$\widehat r_{\overline{m-1,m}}$};
		\node[below=3pt] at (r5) {$\widehat r_{\overline{m}1}$};
		
	\end{tikzpicture}
\end{center}

\end{theorem}

\begin{proof}
For the Coxeter graph for $\mathfrak{spo}(2m|2)$, the situation is similar to  the one for $\mathfrak{spo}(2m|2n)$ with $n\geq 2$. By the same arguments,  the order of $\widehat r_{\overline{m-1,m}}\cdot \widehat r_{\overline{m}1}$ is equal to $12$. Thus we  obtain the Coxeter graph as stated.
\end{proof}

\section{Coxeter graph for $B(m|n)$}\label{se:4}

Let $\ggg=\mathfrak{spo}(2m|2n+1)$. Due to (0.4.2) in \S\ref{sec: main results forward}, we will assume $n>0$ in this section.

\subsection{} Recall that we have an order for the standard  fundamental system of $\ggg$ which is (cf. (\ref{F1.3}))
\begin{align}\label{eq: st fund sys B}
\Pi=\{\varepsilon_{\bar i}-\varepsilon_{\overline {i+1}},\varepsilon_{\overline m}-\varepsilon_{1}, \varepsilon_{k}-\varepsilon_{k+1}, \varepsilon_{n}\mid 1\leq i\leq m-1, 1\leq k\leq n-1\}.
\end{align}
Corresponding to the standard  fundamental system,
 the positive root system of $\ggg$ is  $$\{\delta_{i}\pm\delta_{j},  2\delta_{p}, \varepsilon_{k}\pm \varepsilon_{l}, \varepsilon_{q}\}\cup\{\delta_{p}\pm\varepsilon_{q},\delta_{p}\}$$
where $1\leq i<j\leq m, 1\leq k<l\leq n, 1\leq p\leq m, 1\leq q\leq n$.
 We have already made an appointment by setting $\varepsilon_{\bar i}=\delta_{i}, 1\leq i\leq m.$

Recall that the extended standard fundamental system  $\widetilde \Pi$ for $\mathfrak{spo}(2m|2n+1)$ with $n>0$ is $ \widetilde\Pi= \{-2\delta_1\}\cup \Pi$ (cf. (\ref{eq: ext fund sys})), and that the Coxeter system is $(\wsc, \ssp)$.

\subsection{Defining-sequence theorem}
Keep in mind that the fundamental root system $\Pi$ is presented as in (\ref{eq: st fund sys B}).
Set
$\widehat r_{\bar 1}:=\widehat r_{-2\delta_1}=\widehat r_{2\delta_1},
\widehat r_{\overline{ i,i+1}}:=\widehat r_{\delta_i-\delta_{i+1}}$,
$\widehat r_{i,i+1}:=\widehat r_{\vep_i-\vep_{i+1}}$, $\widehat r_n:=\widehat r_{\vep_n}$
 and
$\widehat r_{\overline{m}1} := \widehat r_{ \delta_m-\varepsilon_1}$.

 Summarizing Lemmas \ref{lem: lemma 2.10 for type B} and
 \ref{lem: lemma 2.11 for type B} and reformulating them,  we have the following results which are parallel to  Proposition \ref{app: gen seq} with minor modifications.


\subsubsection{$\ggg=\mathfrak{spo}(2m|2n+1)$ with $n>1$}
\begin{proposition}\label{app: type B seq}
Let $\ggg=\mathfrak{spo}(2m|2n+1)$ with $n>1$, and the ordered  standard fundamental system (\ref{eq: st fund sys B}) correspond to the sequence
    $(\bar 1,\bar 2, \ldots, \overline m, 1,\ldots,n-1,n)$.
Let $\Pi_1$ be a fundamental system coming from another fundamental system $\Pi_0$ by a simple reflection $\widehat r_\theta\in \textsf{S}_{\widetilde\Pi}$. The following statements hold.
    \begin{itemize}
    \item[(1)] Suppose $\theta=\delta_m-\vep_1$,  consequently the defining sequence of $\Pi_0$ can be written as
        $$(\cdots ,\pm \overline{m},\pm 1, \cdots) \text{ or }(\cdots, \pm 1,\pm \overline m, \cdots),$$
        which means that $\Pi_0$ looks like $(\cdots, \pm(\delta_m-\varepsilon_1), \pm(\vep_1+\heartsuit), \cdots)$ or  $(\cdots, \pm(\vep_1-\delta_m), \pm(\delta_m+\diamondsuit), \cdots)$ with $\heartsuit,\diamondsuit\in \{\pm\vep_i\mid i\in I(m|n)\}$ such that $\vep_1+\heartsuit$ and $\delta_m+\diamondsuit$ are roots,
        then  $\Pi_1$ corresponds to the sequence $$(\cdots ,\pm 1,\pm \overline{m},\cdots) \text{ or }(\cdots, \pm \overline m, \pm1,\cdots),$$
         respectively.  This means that we only exchange the positions of $\pm\overline m$ and $\pm1$, all others do not change.

    \item[(2)] Suppose  $\theta=\varepsilon_{\widetilde{i}}-\varepsilon_{\widetilde{i+1}}$ for $\widetilde{i}\in\{\bar 1,\ldots,\overline{m-1}\}$ or $ \{1,\ldots,n-1\}$
    and consequently the defining sequence of $\Pi_0$ can be written as
\begin{itemize}
\item[(Case 2.1)]
       $(\cdots ,\pm\widetilde{i},\cdots,\pm \widetilde{(i+1)}, \cdots) \text{ or }(\cdots, \pm\widetilde{(i+1)},\cdots,\pm\widetilde{i}, \cdots)$; or
  \item[(Case 2.2)] $(\cdots ,\pm\widetilde{i},\cdots,\mp \widetilde{(i+1)}, \cdots) \text{ or }(\cdots, \mp\widetilde{(i+1)},\cdots,\pm\widetilde{i}, \cdots)$.
\end{itemize}

        Then in (Case 2.1),  $\Pi_1$ corresponds to the sequence $$(\cdots ,\pm \widetilde{(i+1)},\cdots,\pm \widetilde{i},\cdots) \text{ or }(\cdots, \pm\bar{i},\cdots,\pm\widetilde{(i+1)}, \cdots),$$
         respectively,  which  means that we only exchange the positions of $\pm\widetilde{i}$ and $\pm\widetilde{(i+1)}$, all others do not change;
          In (Case 2.2),  $\Pi_1$ corresponds to the sequence $$(\cdots ,\pm \widetilde{(i+1)},\cdots,\mp \widetilde{i},\cdots) \text{ or }(\cdots, \mp\widetilde{i},\cdots,\pm\widetilde{(i+1)}, \cdots),$$
         respectively,  which  means that we exchange the positions of $\pm\widetilde{i}$ and $\mp\widetilde{(i+1)}$, and simultaneously change both signs, all others do not change.

    \item[(3)] Suppose $\widehat r_\theta= \widehat r_{-2\delta_1}= \widehat r_{2\delta_1}$, consequently the defining sequence of $\Pi_0$ can be written as $(\cdots,\pm \overline 1,\cdots)$, then $\Pi_1$ corresponds to the sequence $(\cdots,\mp \overline 1,\cdots)$.  This means that we only change the sign of $\pm\overline 1$,   all others do not change.
        \item[(4)] Suppose $\widehat r_\theta= \widehat r_{\vep_n}$,  consequently the defining sequence of $\Pi_0$ can be written as $(\cdots, \pm n,\cdots)$, then $\Pi_1$ corresponds to the sequence $(\cdots,\mp n,\cdots)$. This  means that we only change the sign of $\pm n$,   all others do not change.

        \end{itemize}
\end{proposition}

\subsubsection{ $\ggg=\mathfrak{spo}(2m|3)$}  In this case, the ordered standard fundamental system (\ref{eq: st fund sys B}) becomes $$\Pi=\{\delta_1-\delta_2,\ldots,\delta_{m-1}-\delta_m, \delta_m-\vep_1,\vep_1\}.$$
\begin{proposition}\label{app: type B seq n=1}
Suppose $\ggg=\mathfrak{spo}(2m|3)$, and make the ordered standard fundamental system (\ref{eq: st fund sys B}) correspond to the sequence
    $(\bar 1,\bar 2, \ldots, \overline m, 1)$.
Suppose a fundamental system $\Pi_1$ comes from another fundamental system $\Pi_0$ by a simple reflection $\widehat r_\theta\in \textsf{S}_{\widetilde\Pi}$.
    \begin{itemize}
    \item[(1)] If $\theta=\delta_m-\vep_1$ and consequently the defining sequence of $\Pi_0$ can be written as
        $$(\cdots ,\pm \overline{1},\pm 1, \cdots) \text{ or }(\cdots, \pm 1,\pm \overline 1, \cdots),$$
        which means that $\Pi_0$ looks like $(\cdots, \pm(\delta_m-\vep_1), \pm(\vep_1+\heartsuit), \cdots)$ or  $(\cdots, \pm(\vep_1-\delta_m), \pm(\delta_m+\diamondsuit), \cdots)$ with $\heartsuit,\diamondsuit\in \{\pm\vep_i\mid i\in I(m|n)\}$ such that $\vep_1+\heartsuit$ and $\delta_1+\diamondsuit$ are roots,
        then  $\Pi_1$ corresponds to the sequence $$(\cdots ,\pm 1,\pm \overline{m},\cdots) \text{ or }(\cdots, \pm \overline m, \pm1,\cdots),$$
         respectively.  This means that we only exchange the positions of $\pm\overline m$ and $\pm1$, all others do not change.

    \item[(2)] If $\theta=\delta_{{i}}-\delta_{{i+1}}$
    and consequently the defining sequence of $\Pi_0$ can be written as
\begin{itemize}
\item[(Case 2.1)]
       $(\cdots ,\pm\bar{i},\cdots,\pm \overline{(i+1)}, \cdots) \text{ or }(\cdots, \pm\overline{(i+1)},\cdots,\pm\bar{i}, \cdots)$; or
  \item[(Case 2.2)] $(\cdots ,\pm\bar{i},\cdots,\mp \overline{(i+1)}, \cdots) \text{ or }(\cdots, \mp\overline{(i+1)},\cdots,\pm\bar{i}, \cdots)$,
\end{itemize}
        then in (Case 2.1),  $\Pi_1$ corresponds to the sequence $$(\cdots ,\pm \overline{(i+1)},\cdots,\pm \bar{i},\cdots) \text{ or }(\cdots, \pm\bar{i},\cdots,\pm\overline{(i+1)}, \cdots),$$
         respectively,  which  means that we only exchange the positions of $\pm\bar{i}$ and $\pm\overline{(i+1)}$, all others do not change;
          In (Case 2.2),  $\Pi_1$ corresponds to the sequence $$(\cdots ,\pm \overline{(i+1)},\cdots,\mp \bar{i},\cdots) \text{ or }(\cdots, \mp\bar{i},\cdots,\pm\overline{(i+1)}, \cdots),$$
         respectively,  which  means that we exchange the positions of $\pm\bar{i}$ and $\mp\overline{(i+1)}$, and simultaneously change both signs, all others do not change.

    \item[(3)] If $\widehat r_\theta= \widehat r_{-2\delta_m}= \widehat r_{2\delta_m}$ and consequently the defining sequence of\, $\Pi_0$ can be written as $(\cdots,\pm \overline m,\cdots)$, then $\Pi_1$ corresponds to the sequence $(\cdots,\mp \overline m,\cdots)$.  This  means that we only change the sign of $\pm\overline m$,   all others do not change.
     \item[(4)] If $\widehat r_\theta= \widehat r_{\vep_1}$, and consequently the defining sequence of $\Pi_0$ can be written as $(\cdots, \pm 1,\cdots)$, then $\Pi_1$ corresponds to the sequence $(\cdots,\mp 1,\cdots)$. This  means that we only change the sign of $\pm 1$,   all others do not change.
        \end{itemize}
\end{proposition}

\subsubsection{} By summing up Propositions \ref{app: type B seq} and \ref{app: type B seq n=1} along with Lemma \ref{lem: lemma 2.11 for type B},  we have established   a correspondence between fundamental systems and their defining sequences for type $B(m|n)$.

\begin{theorem}\label{prop: presenting seq type B} (Defining-sequence Theorem) The set of fundamental systems of $\ggg=\mathfrak{spo}(2m|2n+1)$ with $n>0$ are in a one-to-one correspondence with the set of defining sequences.  {Such a one-to-one correspondence is $\widehat W$-equivariant in the sense that  it is compatible with  $\widehat W$-action.}
\end{theorem}

{
\begin{proof} It follows from the similar arguments to Theorem \ref{prop: presenting seq type D}.
\end{proof}
}

\subsection{Weyl group for $\mathfrak{spo}(2m|2n+1)$ and its Coxeter graph} By definition, {{the (ordinary) Weyl group  $\scrw(\mathfrak{spo}(2m|2n+1))$}} of $\mathfrak{spo}(2m|2n+1)$ is exactly the Weyl group of $\mathfrak{g}_{\bar{0}}=\mathfrak{sp}(2m)\oplus\mathfrak{so}(2n+1)$, which is isomorphic to $(\mathbb{Z}_{2}^m\rtimes\mathfrak{S}_{m})\times(\mathbb{Z}_{2}^n\rtimes\mathfrak{S}_{n})$.
Precisely,
\begin{align*}
\scrw(\mathfrak{spo}(2m|2n+1))=\langle& \widehat r_{\theta}\mid \theta=\delta_{i}-\delta_{i+1}, \text{~or~} \theta=2\delta_{1}, \text{~or~} \theta=\varepsilon_{k}-\varepsilon_{k+1}, \text{~or~} \theta=\varepsilon_{n},\cr
 &1\leq i\leq m-1, 1\leq k\leq n-1 , \text{~with relations}\rangle
 \end{align*}
where the relations are just the defining relations of corresponding generators in $(\mathbb{Z}_{2}^m\rtimes\mathfrak{S}_{m})
 \times(\mathbb{Z}_{2}^n\rtimes\mathfrak{S}_{n})$.

The Coxeter graph of the Weyl group for $\mathfrak{spo}(2m|2n+1)$ consists of the following two separated sub-graphs:

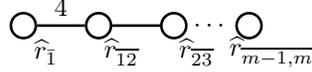
\begin{center}
	\begin{tikzpicture}
		
		\node[circle,draw,minimum size=8pt,inner sep=0pt,fill=white] (r1) at (-4,1) {};
		
		\node[circle,draw,minimum size=8pt,inner sep=0pt,fill=white] (r2) at (-3,1) {};
		
		\node[circle,draw,minimum size=8pt,inner sep=0pt,fill=white] (r3) at (-2,1) {};
		
		\node[circle,draw,minimum size=8pt,inner sep=0pt,fill=white] (r4) at (-1,1) {};

		\draw (r1) --node[above=2pt] {\footnotesize 4} (r2);
		\draw (r2) --  (r3);
		\node at(-1.5,1){ $\cdots$};
		
		\node[below=3pt] at (r1) {$\widehat r_{\bar 1}$};
		\node[below=3pt] at (r2) { $\widehat r_{\overline{12}}$};
		\node[below=3pt] at (r3) {$\widehat r_{\overline{23}}$};
		\node[below=3pt] at (r4) {$\widehat r_{\overline{m-1,m}}$};
		
	\end{tikzpicture}
\end{center}

and

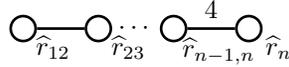
\begin{center}
	\begin{tikzpicture}
		
		\node[circle,draw,minimum size=8pt,inner sep=0pt,fill=white] (r1) at (0,1) {};
	
		\node[circle,draw,minimum size=8pt,inner sep=0pt,fill=white] (r2) at (1,1) {};
		
		\node[circle,draw,minimum size=8pt,inner sep=0pt,fill=white] (r3) at (2,1) {};
		
		\node[circle,draw,minimum size=8pt,inner sep=0pt,fill=white] (r4) at (3,1) {};

		\draw (r1) --(r2);
		\node at(1.5,1){ $\cdots$};
		\draw (r3) --  node[above=2pt] {\footnotesize 4} (r4);
		
		\node[below=3pt] at (r1) {$\widehat r_{1 2}$};
		\node[below=3pt] at (r2) {$\widehat r_{23}$};
		\node[below=3pt] at (r3) {$\widehat r_{n-1, n}$};
		\node[below=3pt] at (r4) {$\widehat r_{n}$};
		
	\end{tikzpicture}
\end{center}

\subsection{The super Weyl group for $\mathfrak{spo}(2m|2n+1)$ and its Coxeter graph: exceptional cases} At first, we have to separately deal with some exceptional cases $\mathfrak{spo}(2m|2n+1)$ with $n\geq1$ and $m+n<4$.

\subsubsection{$\mathfrak{spo}(2|3)$}\label{e5.2} For $\mathfrak{spo}(2|3)$, there are 8 fundamental systems, they are listed in order
$$\{\varepsilon_{\bar{1}}-\varepsilon_1,\varepsilon_1\},
\{\varepsilon_{\bar{1}}+\varepsilon_1,-\varepsilon_1\},
\{-\varepsilon_{\bar{1}}-\varepsilon_1,\varepsilon_1\},
\{-\varepsilon_{\bar{1}}+\varepsilon_1,-\varepsilon_1\},$$$$
\{\varepsilon_{\bar{1}}-\varepsilon_1,-\varepsilon_{\bar{1}}\},
\{\varepsilon_{\bar{1}}+\varepsilon_1,-\varepsilon_{\bar{1}}\},
\{-\varepsilon_{\bar{1}}-\varepsilon_1,\varepsilon_{\bar{1}}\},
\{-\varepsilon_{\bar{1}}+\varepsilon_1,\varepsilon_{\bar{1}}\}.$$
The corresponding sequences of these  fundamental systems are defined to be
 $$\bar 1 1,\bar 1 -1, -\bar 1 1, -\bar 1 -1,  -1-\bar 1, 1-\bar 1,-1\bar 1, 1\bar 1.$$
Parametrize  the above fundamental systems in order by  $1,2,3,4,5,6,7,8$ respectively.
Similar to Example \ref{gl12},  we get $$\widehat r_{\bar{1}1} = (18)(45), \widehat r'_{\bar{1}1} = (27)(36), \widehat r_{\bar{1}} = (13)(24)(57)(68),  \widehat r_{1} = (12)(34)(56)(78).$$

 As $\widehat r_{\bar{1}}\cdot \widehat r_{\bar{1}1} =\widehat r'_{\bar{1}1}\cdot \widehat r_{\bar{1}}$, we have that the set of generators of $\widehat{W}(\mathfrak{spo}(2|3))$ is $$\{a=\widehat r_{\bar{1}}, b=\widehat r_{1}, c=\widehat r_{\bar{1}1}\}.$$

 Using the software SageMath, one can get the Coxeter group
 $$\wsc(\mathfrak{spo}(2|3))=\langle a,b,c\mid a^2=1, b^2=1, c^2=1, (a  b)^2=1, (a  c)^4=1, (b  c)^{4}=1 \rangle.$$
 Then the Coxeter graph of $\wsc(\mathfrak{spo}(2|3)$ is

 
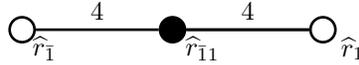
\begin{center}
 	\begin{tikzpicture}
 		
 		\node[circle,draw,minimum size=8pt,inner sep=0pt,fill=white] (r1) at (-2,1) {};

 		\node[circle,draw,minimum size=8pt,inner sep=0pt,fill=white] (r2) at (0,1) {};
 		\draw (r2.135) -- (r2.315);
 		\draw(r2.225) -- (r2.45);
 		
 		\node[circle,draw,minimum size=8pt,inner sep=0pt,fill=white] (r3) at (2,1) {};

 		\draw (r1) -- node[above=2pt] {\footnotesize 4} (r2);
 		\draw (r2) --  node[above=2pt] {\footnotesize 4} (r3);
 		
 		\node[below=3pt] at (r1) {$\widehat r_{\bar{1}}$};
 		\node[below=3pt] at (r2) {$\widehat r_{\bar{1}1}$};
 		\node[below=3pt] at (r3) { $\widehat r_{1 }$};
 		
 	\end{tikzpicture}
 \end{center}

\subsubsection{$\mathfrak{spo}(2|5)$}\label{e5.3} 	
	For  $\mathfrak{spo}(2|5)$, there are $48$ fundamental systems.
	As
	$\widehat r_{\varepsilon_{\bar 1} }\cdot \widehat r_{\varepsilon_{\bar 1} - \varepsilon_1} = \widehat r_{\varepsilon_{\bar 1} +\varepsilon_1}\cdot \widehat r_{\varepsilon_{\bar 1}}$,
	$\widehat r_{\varepsilon_1}\cdot \widehat r_{\varepsilon_1+\varepsilon_2} = \widehat r_{\varepsilon_1-\varepsilon_2}\cdot \widehat r_{\varepsilon_1}$,
	$\widehat r_{\varepsilon_2} = \widehat r_{\varepsilon_1-\varepsilon_2}\cdot \widehat r_{\varepsilon_1}\cdot \widehat r_{\varepsilon_1-\varepsilon_2}$,
	$\widehat r_{\varepsilon_1+\varepsilon_2} =\widehat r_{\varepsilon_1}\cdot  \widehat r_{\varepsilon_1-\varepsilon_2}\cdot \widehat r_{\varepsilon_1}$,
	$\widehat r_{\varepsilon_{\bar 1}-\varepsilon_1}\cdot \widehat r_{\varepsilon_1 - \varepsilon_2} = \widehat r_{\varepsilon_1 - \varepsilon_2}\cdot \widehat r_{\varepsilon_{\bar 1}-\varepsilon_2}$, 	
we can get	the set of generators of $\widehat{W}(\mathfrak{spo}(2|5))$ to be $$\{a=\widehat r_{\varepsilon_{\bar 1} - \varepsilon_1}, b=\widehat r_{\varepsilon_1 - \varepsilon_2}, c=\widehat r_{\varepsilon_{\bar 1}}, d=\widehat r_{\varepsilon_2}\}.$$
		
	Using SageMath, one can get the Coxeter group
	$$\wsc(\mathfrak{spo}(2|5))=\langle a,b,c,d\mid a^2=b^2=c^2=d^2=1, (a  b)^{12}=(c  a)^4=(b  d)^{4}=1,$$$$(a  d)^2 = (b  c)^2 = (c  d)^2 =1 \rangle.$$
	The Coxeter graph of $\wsc(\mathfrak{spo}(2|5))$ is

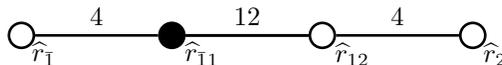
\begin{center}
	\begin{tikzpicture}
		
		\node[circle,draw,minimum size=8pt,inner sep=0pt,fill=white] (r1) at (-2,1) {};
		
		\node[circle,draw,minimum size=8pt,inner sep=0pt,fill=white] (r2) at (0,1) {};		
		\draw (r2.135) -- (r2.315);
		\draw(r2.225) -- (r2.45);
		
		\node[circle,draw,minimum size=8pt,inner sep=0pt,fill=white] (r3) at (2,1) {};
		
		\node[circle,draw,minimum size=8pt,inner sep=0pt,fill=white] (r4) at (4,1) {};

		\draw (r1) -- node[above=2pt] {\footnotesize 4} (r2);
		\draw (r2) -- node[above=2pt] {\footnotesize 12} (r3);
		\draw (r3) -- node[above=2pt] {\footnotesize 4} (r4);
		
		\node[below=3pt] at (r1) {$\widehat r_{\bar{1}}$};
		\node[below=3pt] at (r2) {$\widehat r_{\bar{1}1}$};
		\node[below=3pt] at (r3) {$\widehat r_{1 2}$};
		\node[below=3pt] at (r4) { $\widehat r_{2}$};
		
	\end{tikzpicture}
\end{center}

\subsubsection{$\mathfrak{spo}(4|3)$} \label{e5.4}
	
	For $\mathfrak{spo}(4|3)$, there are $48$ fundamental systems.
	The set of generators of $\widehat{W}(\mathfrak{spo}(4|3))$ is $$\{ a=\widehat r_{\varepsilon_{\bar 1} - \varepsilon_{\bar 2}}, b=\widehat r_{\varepsilon_{\bar 2} - \varepsilon_1}, c=\widehat r_{\varepsilon_{\bar 1}}, d=\widehat r_{\varepsilon_1}\}.$$
	
	Using the computer software SageMath, we can get the Coxeter group
$$\wsc(\mathfrak{spo}(4|3))=\langle a,b,c,d\mid a^2=b^2=c^2=d^2=1, (a  b)^{12}=(c  a)^4=(b  d)^{4}=1,$$$$(a  d)^2 = (b  c)^2 = (c  d)^2 =1 \rangle.$$
The Coxeter graph of $\wsc(\mathfrak{spo}(4|3))$ is

	
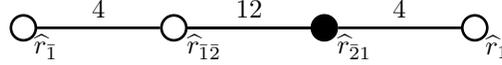
\begin{center}
		\begin{tikzpicture}
			
			\node[circle,draw,minimum size=8pt,inner sep=0pt,fill=white] (r1) at (-2,1) {};
			
			\node[circle,draw,minimum size=8pt,inner sep=0pt,fill=white] (r2) at (0,1) {};
			
			\node[circle,draw,minimum size=8pt,inner sep=0pt,fill=white] (r3) at (2,1) {};
			\draw (r3.135) -- (r3.315);
			\draw(r3.225) -- (r3.45);
			
			\node[circle,draw,minimum size=8pt,inner sep=0pt,fill=white] (r4) at (4,1) {};

			\draw (r1) -- node[above=2pt] {\footnotesize 4} (r2);
			\draw (r2) -- node[above=2pt] {\footnotesize 12}(r3);
			\draw (r3) -- node[above=2pt] {\footnotesize 4}(r4);
			
			\node[below=3pt] at (r1) {$\widehat r_{\bar{1}}$};
			\node[below=3pt] at (r2) {$\widehat r_{\bar{1}\bar{2}}$};
			\node[below=3pt] at (r3) { $\widehat r_{\bar 2 1}$};
			\node[below=3pt] at (r4) { $\widehat r_{1}$};
			
		\end{tikzpicture}
	\end{center}

\subsection{General $\mathfrak{spo}(2m|2n+1)$ ($m,n>0$)}
Based on Proposition \ref{app: type B seq}  and Theorem \ref{prop: presenting seq type B} along with the above examples, the Coxeter graph for  $\mathfrak{spo}(2m|2n+1)$ is described below.

\begin{theorem} \label{thm: 6.5}
The following statements hold for  $\mathfrak{spo}(2m|2n+1)$ with $m, n>0$.
\begin{itemize}
\item[(1)] Suppose   $m + n\geq 4$.
For the standard fundamental system $\Pi$, the Coxeter graph of $\wsc$ for $\mathfrak{spo}(2m|2n+1))$ is

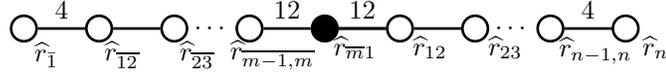
\begin{center}
	\begin{tikzpicture}
		
		\node[circle,draw,minimum size=8pt,inner sep=0pt,fill=white] (r1) at (-4,1) {};
		
		\node[circle,draw,minimum size=8pt,inner sep=0pt,fill=white] (r2) at (-3,1) {};
		
		\node[circle,draw,minimum size=8pt,inner sep=0pt,fill=white] (r3) at (-2,1) {};
		
		\node[circle,draw,minimum size=8pt,inner sep=0pt,fill=white] (r4) at (-1,1) {};
		
		\node[circle,draw,minimum size=8pt,inner sep=0pt,fill=white] (r5) at (0,1) {};
		\draw (r5.135) -- (r5.315);
		\draw(r5.225) -- (r5.45);
		
		\node[circle,draw,minimum size=8pt,inner sep=0pt,fill=white] (r6) at (1,1) {};
		
		\node[circle,draw,minimum size=8pt,inner sep=0pt,fill=white] (r7) at (2,1) {};
		
		\node[circle,draw,minimum size=8pt,inner sep=0pt,fill=white] (r8) at (3,1) {};
		
		\node[circle,draw,minimum size=8pt,inner sep=0pt,fill=white] (r9) at (4,1) {};

		\draw (r1) -- node[above=2pt] {\footnotesize 4}(r2);
		\draw (r2) -- (r3);
		\node at(-1.5,1){ $\cdots$};
		\draw (r4) -- node[above=2pt] {\footnotesize 12}(r5);
		\draw (r5) -- node[above=2pt] {\footnotesize 12}(r6);
		\draw (r6) -- (r7);
		\node at(2.5,1){ $\cdots$};
		\draw (r8) -- node[above=2pt] {\footnotesize 4}(r9);

		\node[below=3pt] at (r1) {$\widehat r_{\bar 1}$};
		\node[below=3pt] at (r2) {$\widehat r_{\overline{12}}$};
		\node[below=3pt] at (r3) {$\widehat r_{\overline{23}}$};
		\node[below=3pt] at (r4) {$\widehat r_{\overline{m-1,m}}$};
		\node[below=3pt] at (r5) { $\widehat r_{\overline{m}1}$};
		\node[below=3pt] at (r6) {$\widehat r_{1 2}$};
		\node[below=3pt] at (r7) { $\widehat r_{23}$};
		\node[below=3pt] at (r8) { $\widehat r_{n-1, n}$};
		\node[below=3pt] at (r9) { $\widehat r_{n}$};
		
	\end{tikzpicture}
\end{center}

\item[(2)] Suppose $m+n<4$.
\begin{itemize}
\item[(2.1)] In the case $\mathfrak{spo}(2|3)$, the Coxeter  graph is

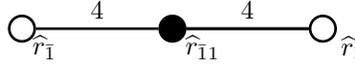
\begin{center}
	\begin{tikzpicture}
		
		\node[circle,draw,minimum size=8pt,inner sep=0pt,fill=white] (r1) at (-2,1) {};

		\node[circle,draw,minimum size=8pt,inner sep=0pt,fill=white] (r2) at (0,1) {};
		\draw (r2.135) -- (r2.315);
		\draw(r2.225) -- (r2.45);
		
		\node[circle,draw,minimum size=8pt,inner sep=0pt,fill=white] (r3) at (2,1) {};

		\draw (r1) -- node[above=2pt] {\footnotesize 4} (r2);
		\draw (r2) -- node[above=2pt] {\footnotesize 4}(r3);
		
		\node[below=3pt] at (r1) {$\widehat r_{\bar{1}}$};
		\node[below=3pt] at (r2) {$\widehat r_{\bar{1}1}$};
		\node[below=3pt] at (r3) {$\widehat r_{1 }$};
		
	\end{tikzpicture}
\end{center}

\item[(2.2)] In the case $\mathfrak{spo}(2|5)$, the Coxeter  graph is

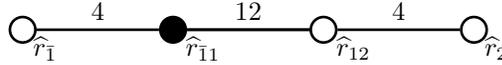
\begin{center}
	\begin{tikzpicture}
		
		\node[circle,draw,minimum size=8pt,inner sep=0pt,fill=white] (r1) at (-2,1) {};

		\node[circle,draw,minimum size=8pt,inner sep=0pt,fill=white] (r2) at (0,1) {};
		\draw (r2.135) -- (r2.315);
		\draw(r2.225) -- (r2.45);
		
		\node[circle,draw,minimum size=8pt,inner sep=0pt,fill=white] (r3) at (2,1) {};

		\node[circle,draw,minimum size=8pt,inner sep=0pt,fill=white] (r4) at (4,1) {};		
		
		\draw (r1) -- node[above=2pt] {\footnotesize 4} (r2);
		\draw (r2) -- node[above=2pt] {\footnotesize 12} (r3);
		\draw (r3) --  node[above=2pt] {\footnotesize 4}(r4);
		
		\node[below=3pt] at (r1) {$\widehat r_{\bar{1}}$};
		\node[below=3pt] at (r2) {$\widehat r_{\overline{1}1}$};
		\node[below=3pt] at (r3) {$\widehat r_{12}$};
		\node[below=3pt] at (r4) { $\widehat r_{2}$};
		
	\end{tikzpicture}
\end{center}

\item[(2.3)] In the case $\mathfrak{spo}(4|3)$, the Coxeter  graph is

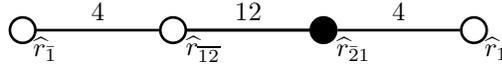
\begin{center}
	\begin{tikzpicture}
		
		\node[circle,draw,minimum size=8pt,inner sep=0pt,fill=white] (r1) at (-2,1) {};

		\node[circle,draw,minimum size=8pt,inner sep=0pt,fill=white] (r2) at (0,1) {};
		
		\node[circle,draw,minimum size=8pt,inner sep=0pt,fill=white] (r3) at (2,1) {};
		\draw (r3.135) -- (r3.315);
		\draw(r3.225) -- (r3.45);

		\node[circle,draw,minimum size=8pt,inner sep=0pt,fill=white] (r4) at (4,1) {};		
		
		\draw (r1) -- node[above=2pt] {\footnotesize 4} (r2);
		\draw (r2) -- node[above=2pt] {\footnotesize 12} (r3);
		\draw (r3) --  node[above=2pt] {\footnotesize 4}(r4);
		
		\node[below=3pt] at (r1) {$\widehat r_{\bar{1}}$};
		\node[below=3pt] at (r2) {$\widehat r_{\overline{12}}$};
		\node[below=3pt] at (r3) {$\widehat r_{\bar 2 1}$};
		\node[below=3pt] at (r4) { $\widehat r_{1}$};
		
	\end{tikzpicture}
\end{center}

\end{itemize}
\end{itemize}
\end{theorem}

\begin{proof} For Part (2) with $n>0$ and $m+n< 4$, we have $m=1, n = 1$ or  $m=1, n = 2$ and $m=2, n = 1$. These are what are dealt with in Examples \ref{e5.2}, \ref{e5.3} and \ref{e5.4}. So we only need to prove Part (1).

Suppose $m+n\geq 4$ and $m,n>0$ in the following.
Note that the Coxeter graph of  the subsystem $(\scrw, \spo)$ is already known (see \cite[Theorem 1.3.3]{GP}). It suffices for us to compute $m_{xy}$ involving $\widehat r_{\overline{m}1}$ as appearing in \S\ref{sec: fun sys} for Proposition \ref{prop: gen system}. Keep in mind that for a given defining sequence $\pi$, the action of $\widehat r_{\overline m 1}$ on $\pi$ is nontrivial only when $\sharp\overline m$ is adjacent to $\sharp 1$ ($\sharp\in\{\pm1\}$). The nontrivial action is just to exchange the position of $\sharp \overline m$ and $\sharp 1$. Now we proceed with arguments in several steps.


(I)	For $\widehat r_{\overline{i,i+1}}$ $(1\leq i \leq m-2)$, in a sequence, it exchanges the positions of the numbers $\bar{i}$ and $\overline{i+1}$,  but the positions of $m$ and $1$ do not change. Then we can get $\widehat r_{\overline{i,i+1}}\cdot \widehat r_{\overline{m}1}= \widehat r_{\overline{m}1}\cdot \widehat r_{\overline{i,i+1}}$.  This implies $(\widehat r_{\overline{i,i+1}}\cdot \widehat r_{\overline{m}1})^2=1, 1\leq i \leq m-2.$

(II) For $\widehat r_{\bar{1}}$,  if $\overline{m}\neq \bar{1}$, we get   $(\widehat r_{\bar{1}}\cdot \widehat r_{\overline{m}1})^2=1.$
	
(III)	For $\widehat r_{j,j+1}$ $(2\leq j \leq n-1)$, similar to the above cases, we get $(\widehat r_{j,j+1}\cdot \widehat r_{\overline{m}1})^2=1.$

(VI)
For $\widehat r_{n}$, if $1\neq n$, we have $( \widehat r_{\overline{m}1} \cdot \widehat r_{n})^2=1.$
	
(V) For $\widehat r_{12}$, similarly to the proof of Theorem \ref{Thm5.3}, we  only need to consider the following two cases.

(Case V.1) $\pm2$ is adjacent to the pair $\overline m,1$ in the sequence.  In this case, the sequence looks like $(\cdots,  \overline{m},1, \pm2\cdots)$ or  $(\cdots,  \pm2, \overline{m},1,\cdots)$.  (we don't care the specific position of $\pm2$. Computations for both are almost the same. The results can be shown to be the same).  We have
 \begin{align*}
&(\cdots,\overline{m},1,2,\cdots)\overset
	{\widehat r_{\overline{m}1}}{\longrightarrow}(\cdots, 1,\overline{m},2,\cdots)\overset
	{\widehat r_{12}}{\longrightarrow}\cr
&(\cdots,2,\overline{m},1,\cdots)\overset
	{\widehat r_{\overline{m}1}}{\longrightarrow}(\cdots,2, 1,\overline{m},\cdots)\overset
	{\widehat r_{12}}{\longrightarrow}\cr
&(\cdots,1, 2,\overline{m},\cdots)\overset
	{\widehat r_{\overline{m}1}}{\longrightarrow}(\cdots,1, 2,\overline{m},\cdots)\overset
	{\widehat r_{12}}{\longrightarrow}\cr
&(\cdots,2, 1,\overline{m},\cdots)\overset
	{\widehat r_{\overline{m}1}}{\longrightarrow}(\cdots,2,\overline{m},1,\cdots)\overset
	{\widehat r_{12}}{\longrightarrow}\cr
&(\cdots,1,\overline{m},2,\cdots)\overset
	{\widehat r_{\overline{m}1}}{\longrightarrow}(\cdots,\overline{m},1,2, \cdots)\overset
	{\widehat r_{12}}{\longrightarrow}\cr
&(\cdots,\overline{m},2,1,\cdots)\overset
	{\widehat r_{\overline{m}1}}{\longrightarrow}(\cdots,\overline{m},2,1,\cdots)\overset
	{\widehat r_{12}}{\longrightarrow}\cr
&(\cdots,\overline{m},1,2,\cdots),
\end{align*}
{ which shows that the order of the action of $\widehat r_{12}\cdot \widehat r_{\overline{m}1}$ on such sequences is $6$,  and
 \begin{align*}
&(\cdots,\overline{m},1,-2,\cdots)\overset
	{\widehat r_{\overline{m}1}}{\longrightarrow}(\cdots, 1,\overline{m},-2,\cdots)\overset
	{\widehat r_{12}}{\longrightarrow}\cr
&(\cdots, 2,\overline{m}, -1,\cdots)\overset
	{\widehat r_{\overline{m}1}}{\longrightarrow}(\cdots,2, \overline{m},-1\cdots)\overset
	{\widehat r_{12}}{\longrightarrow}\cr
&(\cdots,1,\overline{m},-2\cdots)\overset
	{\widehat r_{\overline{m}1}}{\longrightarrow}(\cdots, \overline{m},1,-2\cdots)\overset
	{\widehat r_{12}}{\longrightarrow}\cr
&(\cdots,\overline{m}, 2,-1\cdots)\overset
	{\widehat r_{\overline{m}1}}{\longrightarrow}(\cdots,\overline{m},2,-1,\cdots)\overset
	{\widehat r_{12}}{\longrightarrow}\cr
&(\cdots,\overline{m},1,-2,\cdots),
\end{align*}
 which shows that  the order of the action of $\widehat r_{12}\cdot \widehat r_{\overline{m}1}$ on such sequences  is $4$.

(Case V.2) $\pm2$ is not adjacent to the pair  $\overline{m}, 1$.  In this case, the sequence looks like
{
$$(\cdots, \pm2, \overline{m-1},-\overline{m}, -1,\cdots)$$
}
(we don't care what the roots  between $\pm2$ and the pair are). We have
\begin{align*}
&(\cdots,2, \overline{m-1}, \overline{m}, 1,\cdots)\overset
	{\widehat r_{\overline{m}1}}{\longrightarrow}(\cdots,2,\overline{m-1}, 1,\overline{m},\cdots)\overset
	{\widehat r_{12}}{\longrightarrow}\cr
&(\cdots, 1,\overline{m-1},2,\overline{m},\cdots)\overset
	{\widehat r_{\overline{m}1}}{\longrightarrow}
(\cdots, 1,\overline{m-1},2,\overline{m},\cdots)\overset
	{\widehat r_{12}}{\longrightarrow}\cr
&(\cdots, 2,\overline{m-1}, 1, \overline{m},\cdots)\overset
	{\widehat r_{\overline{m}1}}{\longrightarrow}
(\cdots, 2,\overline{m-1}, \overline{m},1, \cdots)\overset
	{\widehat r_{12}}{\longrightarrow}\cr
&(\cdots, 1,  \overline{m-1}, \overline{m},2,\cdots)\overset
	{\widehat r_{\overline{m}1}}{\longrightarrow}
(\cdots, 1,\overline{m-1}, \overline{m},2,\cdots)\overset
	{\widehat r_{12}}{\longrightarrow}\cr
 &(\cdots,2, \overline{m-1}, \overline{m},1,\cdots),
 \end{align*}
 which shows that the  order of the action of $\widehat r_{12}\cdot \widehat r_{\overline{m}1}$ on such sequences is $4$, and
 \begin{align*}
&(\cdots,-2, \overline{m-1}, \overline{m}, 1,\cdots)\overset
	{\widehat r_{\overline{m}1}}{\longrightarrow}(\cdots,-2,\overline{m-1}, 1,\overline{m},\cdots)\overset
	{\widehat r_{12}}{\longrightarrow}\cr
&(\cdots, -1,\overline{m-1},2,\overline{m},\cdots)\overset
	{\widehat r_{\overline{m}1}}{\longrightarrow}
(\cdots, -1,\overline{m-1},2,\overline{m},\cdots)\overset
	{\widehat r_{12}}{\longrightarrow}\cr
&(\cdots, -2,\overline{m-1}, 1, \overline{m},\cdots)\overset
	{\widehat r_{\overline{m}1}}{\longrightarrow}
(\cdots, -2,\overline{m-1}, \overline{m},1, \cdots)\overset
	{\widehat r_{12}}{\longrightarrow}\cr
&(\cdots, -1,  \overline{m-1}, \overline{m},2,\cdots)\overset
	{\widehat r_{\overline{m}1}}{\longrightarrow}
(\cdots, -1,\overline{m-1}, \overline{m},2,\cdots)\overset
	{\widehat r_{12}}{\longrightarrow}\cr
 &(\cdots,-2, \overline{m-1}, \overline{m},1,\cdots),
 \end{align*}
  which shows that the order of the action of $\widehat r_{12}\cdot \widehat r_{\overline{m}1}$ on such sequences is $4$.
 }

As there are $2^{m+n}\cdot(m+n)!$ positive systems, then the number of the  corresponding sequences is $2^{m+n}\cdot(m+n)!$ (cf. \S\ref{s5.1}). So when $m+n\geq 4$, $ \overline{m-1}, \overline{m},1, 2$ and $ 2, \overline{m-1}, -\overline{m},-1$ must appear in some sequences.  According to the above arguments,  the order of $\widehat r_{\overline{m}1}\cdot\widehat r_{12}$ is the least common multiple of $6$, $4$, $2$, which is equal to $12$.

(VII) Consider $\widehat r_{\overline {m-1,m}}\cdot \widehat r_{\overline{m}1}$.  By the same computations as the previous ones concerning  $\widehat r_{\overline{m}1}\cdot\widehat r_{12}$,  one can obtain that  the order of  $\widehat r_{\overline {m-1,m}}\cdot \widehat r_{\overline{m}1}$ is equal to $12$.

This completes the proof of Part (1).

According to the analysis in the beginning, we accomplish the proof.
\end{proof}


\section{Appendix : Proposals for Coxeter graphs for exceptional classical Lie superalgebras}\label{se:6}
\subsection{The root system and Weyl group for $D(2,1,\alpha)$} \subsubsection{The root system}
Recall that the Lie superalgebra $\ggg = D(2|1; \alpha)$ is a family of simple contragredient Lie superalgebras, which depends on a parameter $\alpha$.
There are isomorphisms of Lie superalgebras
with different parameters
$D(2|1; \alpha)\cong D(2,1;-1 -\alpha^{-1})\cong D(2|1; \alpha^{-1})$ when $\alpha$ is nonzero.

Let $\ggg=D(2,1;\alpha)$. Recall that $\ggg= \ggg_\bz\oplus \ggg_\bo$ with $\ggg_\bz\cong \mathfrak{sl}(2)^{\oplus 3}$ and as a $\ggg_\bz$-module, $\ggg_\bo\cong (\bbc^2)^{\oplus 2}$.
For a Cartan subalgebra $\hhh$ of $\ggg_\bz$, its linear dual $\hhh^*$ has a basis $\{\delta, \varepsilon_{1}, \varepsilon_{2} \}$ with $(\delta,\delta)=-(1+\alpha), (\varepsilon_{1},\varepsilon_{1})=1, (\varepsilon_{2},\varepsilon_{2})=\alpha, (\varepsilon_{1},\varepsilon_{2})=0$ and $(\delta,\varepsilon_{i})=0, i=1,2$ (see for example, \cite{K} or \cite{CSW} for the details).
Its root system
$\Phi=\Phi_{\bar{0}}\bigcup \Phi_{\bar{1}}$ is presented as
$$\{\pm2\delta,\pm2\varepsilon_{i}\}\cup
\{\pm\delta\pm\varepsilon_{1}\pm\varepsilon_{2}\}.$$
    We can take
     $$\Pi=\{\delta+\varepsilon_1+ \varepsilon_2, -2\varepsilon_1, -2\varepsilon_2\}$$
as the standard fundamental system of $D(2,1,\alpha)$.  Then the corresponding extended standard fundamental system is
   $$\widetilde \Pi=\{-2\delta\}\cup \{ \delta+\varepsilon_1+ \varepsilon_2, -2\varepsilon_1, -2\varepsilon_2\}.$$

\subsubsection{The Weyl group}
Recall that $\ggg_\bz$ admits the Weyl group $\langle \widehat r_{2\delta}\rangle\times  \langle \widehat r_{2\varepsilon_1}\rangle\times \langle \widehat r_{2\varepsilon_2}\rangle$ which is isomorphic to $\bbz_2\times \bbz_2\times \bbz_2$.

\subsection{Coxeter group for $D(2,1,\alpha)$ and its Coxeter graph}

\begin{proposal}\label{thm: excep d}
	The Coxeter graph of $\wsc$ for $D(2,1,\alpha)$  is	

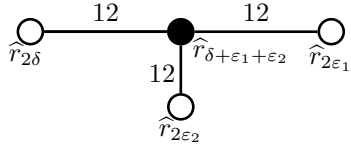
\begin{center}
	\begin{tikzpicture}
		
		\node[circle,draw,minimum size=8pt,inner sep=0pt,fill=white] (r1) at (2,0) {};

		\node[circle,draw,minimum size=8pt,inner sep=0pt,fill=white] (r2) at (0,1) {};

		\node[circle,draw,minimum size=8pt,inner sep=0pt,fill=white] (r3) at (2,1) {};
		\draw (r3.135) -- (r3.315);
		\draw(r3.225) -- (r3.45);
		
		\node[circle,draw,minimum size=8pt,inner sep=0pt,fill=white] (r4) at (4,1) {};

		\draw (r2) -- node[above,sloped] {\footnotesize 12} (r3);
		\draw (r1) -- node[left] {\footnotesize 12} (r3);
		\draw (r3) -- node[above,sloped] {\footnotesize 12} (r4);
		
		\node[below=3pt] at (r1) {$\widehat r_{2\varepsilon_2} $};
		\node[below=3pt] at (r2) {$ \widehat r_{2\delta~~~~}$~~~~~~~};
		\node at(2.8,0.75){ $\widehat r_{\delta+\varepsilon_{1}+\varepsilon_2}$};
		\node[below=3pt] at (r4) {$\widehat r_{2\varepsilon_{1}}$};
		
	\end{tikzpicture}
\end{center}

\end{proposal}

\begin{expl} Note that the Coxeter graph of  the subsystem $(\scrw, \spo)$ is already known, which is just the one of three vertices mutually isolated. It suffices for us to compute $m_{xy}$ involving $\widehat r_{\delta+\varepsilon_{1}+\varepsilon_{2}}$ in \S\ref{sec: cox sys} associated
with Proposition \ref{prop: gen system}. 	
	We only need to calculate  the order of  $\widehat r_{\delta+\varepsilon_{1}+\varepsilon_{2}}\cdot \widehat r_{2\delta}$. The others can be  similarly fulfilled.
	
	We need to  calculate the actions of  $\widehat r_{\delta+\varepsilon_{1}+\varepsilon_{2}}\cdot \widehat r_{2\delta}$ on all the  fundamental systems, and get the order of $\widehat r_{\delta+\varepsilon_{1}+\varepsilon_{2}}\cdot \widehat r_{2\delta}$ on each fundamental system.  The least common multiple of these  orders is $12$. So the order of $\widehat r_{\delta+\varepsilon_{1}+\varepsilon_{2}}\cdot \widehat r_{2\delta}$ is $12$. So far, we can only  give the calculations of the following two fundamental systems. Denote $\widehat r_{\delta12}$ by $\widehat r_{\delta+\varepsilon_{1}+\varepsilon_{2}}$.
{We have
\begin{align*}
&\{ \delta-\varepsilon_{1}-\varepsilon_{2}, 2\varepsilon_{1},2\varepsilon_{2}\}\overset
	{\widehat r_{2\delta}}{\longrightarrow}\{ -\delta-\varepsilon_{1}-\varepsilon_{2}, 2\varepsilon_{1},2\varepsilon_{2}\}\overset
	{\widehat r_{\delta12}}{\longrightarrow}\cr
&\{\delta+\varepsilon_{1}+\varepsilon_{2},-\delta+\varepsilon_{1}-\varepsilon_{2} ,-\delta-\varepsilon_{1}+\varepsilon_{2}\} \overset
	{\widehat r_{2\delta}}{\longrightarrow}\{-\delta+\varepsilon_{1}+\varepsilon_{2},\delta+\varepsilon_{1}-\varepsilon_{2} ,\delta-\varepsilon_{1}+\varepsilon_{2}\}\overset
	{\widehat r_{\delta12}}{\longrightarrow} \cr &\{-\delta+\varepsilon_{1}+\varepsilon_{2},\delta+\varepsilon_{1}-\varepsilon_{2} ,\delta-\varepsilon_{1}+\varepsilon_{2}\}\overset
	{\widehat r_{2\delta}}{\longrightarrow} \{\delta+\varepsilon_{1}+\varepsilon_{2},-\delta+\varepsilon_{1}-\varepsilon_{2} ,-\delta-\varepsilon_{1}+\varepsilon_{2}\}\overset
	{\widehat r_{\delta12}}{\longrightarrow}\cr
 &\{ -\delta-\varepsilon_{1}-\varepsilon_{2}, 2\varepsilon_{1},2\varepsilon_{2}\}
\overset
	{\widehat r_{2\delta}}{\longrightarrow} \{ \delta-\varepsilon_{1}-\varepsilon_{2}, 2\varepsilon_{1},2\varepsilon_{2}\}\overset
	{\widehat r_{\delta12}}{\longrightarrow}\cr
 &\{ \delta-\varepsilon_{1}-\varepsilon_{2}, 2\varepsilon_{1},2\varepsilon_{2}\}.
\end{align*}
So the order of the action $\widehat r_{\delta+\varepsilon_{1}+\varepsilon_{2}}\cdot \widehat r_{2\delta}$ on this fundamental system is $4$.
}	

Furthermore, we have
{
\begin{align*}
&\{ 2\varepsilon_{1}, \delta-\varepsilon_{1}-\varepsilon_{2}, -2\delta\}\overset
	{\widehat r_{2\delta}}{\longrightarrow}\{2\varepsilon_{1}, -\delta-\varepsilon_{1}-\varepsilon_{2}, 2\delta\}\overset
	{\widehat r_{\delta12}}{\longrightarrow}\cr
&\{-\delta+\varepsilon_{1}-\varepsilon_{2},\delta+\varepsilon_{1}+\varepsilon_{2} ,\delta-\varepsilon_{1}-\varepsilon_{2}\} \overset
	{\widehat r_{2\delta}}{\longrightarrow}\{\delta+\varepsilon_{1}-\varepsilon_{2},-\delta+\varepsilon_{1}+\varepsilon_{2} ,-\delta-\varepsilon_{1}-\varepsilon_{2}\}\overset
	{\widehat r_{\delta12}}{\longrightarrow} \cr
&\{-2\varepsilon_{2},-2\delta,\delta+\varepsilon_{1}+\varepsilon_{2}\}
\overset
	{\widehat r_{2\delta}}{\longrightarrow} \{-2\varepsilon_{2},2\delta,-\delta+\varepsilon_{1}+\varepsilon_{2}\}\overset
	{\widehat r_{\delta12}}{\longrightarrow}\cr &\{-2\varepsilon_{2},2\delta,-\delta+\varepsilon_{1}+\varepsilon_{2}\}\overset
	{\widehat r_{2\delta}}{\longrightarrow} \{-2\varepsilon_{2},-2\delta,\delta+\varepsilon_{1} +\varepsilon_{2}\}\overset
	{\widehat r_{\delta12}}{\longrightarrow} \cr &\{\delta+\varepsilon_{1}-\varepsilon_{2},-\delta+\varepsilon_{1}+\varepsilon_{2} ,-\delta-\varepsilon_{1}-\varepsilon_{2}\}\overset
	{\widehat r_{2\delta}}{\longrightarrow} \{-\delta+\varepsilon_{1}-\varepsilon_{2},\delta+\varepsilon_{1}+\varepsilon_{2} ,\delta-\varepsilon_{1}-\varepsilon_{2}\}\overset
	{\widehat r_{\delta12}}{\longrightarrow} \cr
&\{2\varepsilon_{1}, -\delta-\varepsilon_{1}-\varepsilon_{2}, 2\delta\}\overset
	{\widehat r_{2\delta}}{\longrightarrow} \{2\varepsilon_{1}, \delta-\varepsilon_{1}-\varepsilon_{2}, -2\delta\}\overset
	{\widehat r_{\delta12}}{\longrightarrow} \cr
&\{2\varepsilon_{1}, \delta-\varepsilon_{1}-\varepsilon_{2}, -2\delta\}.
\end{align*}
}
 So the order of the action of  $\widehat r_{\delta+\varepsilon_{1}+\varepsilon_{2}}\cdot \widehat r_{2\delta}$ on this fundamental system is $6$.

{Summing up, we have that the order of $\widehat r_{\delta+\varepsilon_{1}+\varepsilon_{2}}\cdot \widehat r_{2\delta}$ is $12$.
According to the analysis in the beginning of the proof, the proposal does make sense. }
\end{expl}


\subsection{Root system for $F(4)$}{
Let $\ggg=F(4)$. Recall that
$\ggg_\bz\cong \mathfrak{sl}(2)\oplus \mathfrak{so}(7)$ with Cartan subalgebra $\hhh$ and $\ggg_\bo=\bbc^2\oplus \bbc^8$, as a $\ggg_\bz$-module. Here $\bbc^8$ is  the $8$-dimensional spin representation of $\mathfrak{so}(7)$. The root system of $\ggg$ can be
described via the basis  $\{\varepsilon_{1},\varepsilon_{2}, \varepsilon_{3}, \delta\}$ in $\hhh^*$ with $(\delta,\delta)=-3, (\varepsilon_{i},\varepsilon_{i})=1, (\varepsilon_{i},\varepsilon_{j})=0, (\delta,\varepsilon_{i})=0, i,j =1,2,3, i\neq j$.
Moreover,
$\Phi=\Phi_{\bar{0}}\bigcup \Phi_{\bar{1}}$ is precisely presented as
$$\{\pm\varepsilon_{i}\pm\varepsilon_{j},\pm\varepsilon_{i},\pm\delta\}\cup
\{\frac{1}{2}(\pm\varepsilon_{1}\pm\varepsilon_{2}
\pm\varepsilon_{3}\pm\delta)\}.$$


      \subsection{Coxeter graph} We can take $$\Pi=\{\frac{1}{2}(\delta+\varepsilon_1+ \varepsilon_2+\varepsilon_3), -\varepsilon_3, \varepsilon_3-\varepsilon_1,  \varepsilon_1-\varepsilon_2\}$$
 as  the standard fundamental system of $F(4)$. Then we can take
  $$\widetilde \Pi=\{-\delta\}\cup \{\frac{1}{2}(\delta+\varepsilon_1+ \varepsilon_2+\varepsilon_3), -\varepsilon_3, \varepsilon_3-\varepsilon_1,  \varepsilon_1-\varepsilon_2\}$$}
as the extended standard fundamental system.


\begin{proposal}\label{thm: excep f}
	The Coxeter graph of $\wsc$ for $F(4)$ is	

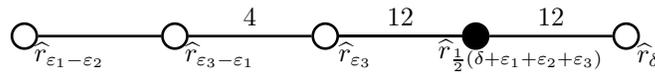
\begin{center}
	\begin{tikzpicture}
		
		\node[circle,draw,minimum size=8pt,inner sep=0pt,fill=white] (r1) at (-5,1) {};

		\node[circle,draw,minimum size=8pt,inner sep=0pt,fill=white] (r2) at (-3,1) {};
		
		\node[circle,draw,minimum size=8pt,inner sep=0pt,fill=white] (r3) at (-1,1) {};
		
		\node[circle,draw,minimum size=8pt,inner sep=0pt,fill=white] (r4) at (1,1) {};
			\draw (r4.135) -- (r4.315);
		\draw(r4.225) -- (r4.45);
		
		\node[circle,draw,minimum size=8pt,inner sep=0pt,fill=white] (r5) at (3,1) {};
		
		\draw (r1) -- (r2);
		\draw (r2) -- node[above=2pt] {\footnotesize 4} (r3);
		\draw (r3) -- node[above=2pt] {\footnotesize 12} (r4);
		\draw (r4) -- node[above=2pt] {\footnotesize 12} (r5);
		
		\node[below=3pt] at (r1) {$\widehat r_{\varepsilon_1-\varepsilon_2} $};
		\node[below=3pt] at (r2) {$ \widehat r_{\varepsilon_3-\varepsilon_1}$};
		\node[below=3pt] at (r3) {$\widehat r_{\varepsilon_3}$};
		\node[below=3pt] at (r4) {$\widehat r_{\frac{1}{2}(\delta+\varepsilon_{1}+\varepsilon_2+\varepsilon_3)}$};		
		\node[below=3pt] at (r5) {$\widehat r_{\delta}$};
		
	\end{tikzpicture}
\end{center}

\end{proposal}

\begin{expl}
		
 {
Note that the Coxeter graph of  the subsystem $(\scrw, \spo)$ is already known (see for example, \cite[Theorem 1.3.3]{GP}). It suffices for us to compute $m_{xy}$ involving $\widehat r_{\frac{1}{2}(\delta+\varepsilon_{1}+\varepsilon_2+\varepsilon_3)}$ in \S\ref{sec: cox sys} associated with Proposition \ref{prop: gen system}.} 	
	Here we give the calculation of $\widehat r_{\frac{1}{2}(\delta+\varepsilon_{1}+\varepsilon_2+\varepsilon_3)}\cdot \widehat r_{\varepsilon_3}$.  { The others cases are  similar.}

We need to calculate the actions of $\widehat r_{\frac{1}{2}(\delta+\varepsilon_{1}+\varepsilon_2+\varepsilon_3)}\cdot \widehat r_{\varepsilon_3}$ on all the  fundamental systems, and get the corresponding orders of them.  The least common multiple of these  orders is $12$. So the order of $\widehat r_{\frac{1}{2}(\delta+\varepsilon_{1}+\varepsilon_2+\varepsilon_3)}\cdot \widehat r_{\varepsilon_3}$ is $12$. So far, we can only give the calculations of the following two fundamental systems.

For $\prod:=\{\alpha_1=\frac{1}{2}(\delta+\varepsilon_{1}+\varepsilon_2+\varepsilon_3),\alpha_2=-\varepsilon_{1},
\alpha_3=\varepsilon_{1}-\varepsilon_{2},\alpha_4=\varepsilon_{2}-\varepsilon_{3}\}$, we have
{\tiny{
\begin{align*}
& \prod\overset
	{\widehat r_{\varepsilon_{3}}}{\longrightarrow}
\{\gamma:=\frac{1}{2}(\delta+\varepsilon_{1}+\varepsilon_2-\varepsilon_3),\alpha_2,\alpha_3, \varepsilon_2+\varepsilon_3 \}
\overset
	{\widehat r_{\alpha_1}}{\longrightarrow}\cr
&\{\gamma, \alpha_2, \alpha_3, \varepsilon_{2}+\varepsilon_{3}\} \overset
	{\widehat r_{\varepsilon_{3}}}{\longrightarrow}\prod
\overset
	{\widehat r_{\alpha_1}}{\longrightarrow}\cr
 &\{ -\alpha_1,\frac{1}{2}(\delta-\varepsilon_{1}+\varepsilon_2+\varepsilon_3),
\alpha_3,\alpha_4\}\overset
	{\widehat r_{\varepsilon_{3}}}{\longrightarrow} \{-\gamma,\frac{1}{2}(\delta-\varepsilon_{1}+\varepsilon_2-\varepsilon_3),
\alpha_3,\varepsilon_2+\varepsilon_3\}\overset
	{\widehat r_{\alpha_1}}{\longrightarrow} \cr
&\{-\gamma,\frac{1}{2}(\delta-\varepsilon_{1}+\varepsilon_2-\varepsilon_3),
\alpha_3,\varepsilon_2+\varepsilon_3\}
\overset
	{\widehat r_{\varepsilon_{3}}}{\longrightarrow} \{  -\alpha,\frac{1}{2}(\delta-\varepsilon_{1}+\varepsilon_2+\varepsilon_3),
\alpha_3,\varepsilon_{2}-\varepsilon_{3}\}\overset
	{\widehat r_{-\alpha_1}=\widehat r_{\alpha_1}}{\longrightarrow} \cr
&\prod.
\end{align*}
}}
So the order of the action $\widehat r_{\frac{1}{2}(\delta+\varepsilon_{1}+\varepsilon_2+\varepsilon_3)}\cdot \widehat r_{\varepsilon_3}$ on this fundamental system is $4$.

For $\coprod:=\{ \beta_1=\frac{1}{2}(\delta+\varepsilon_{1}+\varepsilon_2+\varepsilon_3),\beta_2=\frac{1}{2}(\delta-\varepsilon_{1}-\varepsilon_2-\varepsilon_3),
\beta_3=\frac{1}{2}(-\delta-\varepsilon_{1}-\varepsilon_2+\varepsilon_3),\beta_4=\varepsilon_{2}-\varepsilon_{3}\}$,  we have
{
{\tiny{
\begin{align*}
&\coprod\overset
	{\widehat r_{2\varepsilon_{3}}}{\longrightarrow}\{ \frac{1}{2}(\delta+\varepsilon_{1}+\varepsilon_2-\varepsilon_3),
\frac{1}{2}(\delta-\varepsilon_{1}-\varepsilon_2+\varepsilon_3),
\frac{1}{2}(-\delta-\varepsilon_{1}-\varepsilon_2-\varepsilon_3),
\varepsilon_{2}+\varepsilon_{3}\}\overset
	{\widehat r_{\beta_1}}{\longrightarrow}\cr
&\{-\varepsilon_3,  \varepsilon_{1}+\varepsilon_2, \beta_1,
\frac{1}{2}(-\delta-\varepsilon_{1}+\varepsilon_2+\varepsilon_3)\} \overset{\widehat r_{2\varepsilon_{3}}}{\longrightarrow}\{\varepsilon_3, \varepsilon_{1}+\varepsilon_2, \frac{1}{2}(\delta+\varepsilon_{1}+\varepsilon_2-\varepsilon_3), \frac{1}{2}(-\delta-\varepsilon_{1}+\varepsilon_2-\varepsilon_3)\}
\overset
	{\widehat r_{\beta_1}}{\longrightarrow}\cr
 &\{ \varepsilon_3, \varepsilon_{1}+\varepsilon_2, \frac{1}{2}(\delta+\varepsilon_{1}+\varepsilon_2-\varepsilon_3), \frac{1}{2}(-\delta-\varepsilon_{1}+\varepsilon_2-\varepsilon_3)\}
 \overset
	{\widehat r_{2\varepsilon_{3}}}{\longrightarrow} \{-\varepsilon_3,  \varepsilon_{1}+\varepsilon_2, \frac{1}{2}(\delta+\varepsilon_{1}+\varepsilon_2+\varepsilon_3),
\frac{1}{2}(-\delta-\varepsilon_{1}+\varepsilon_2+\varepsilon_3)\}
\overset
	{\widehat r_{\beta_1}}{\longrightarrow}\cr
 &\{ \frac{1}{2}(\delta+\varepsilon_{1}+\varepsilon_2-\varepsilon_3),
 \frac{1}{2}(\delta-\varepsilon_{1}-\varepsilon_2+\varepsilon_3),
\frac{1}{2}(-\delta-\varepsilon_{1}-\varepsilon_2-\varepsilon_3),
\varepsilon_{2}+\varepsilon_{3}\}\overset
	{\widehat r_{2\varepsilon_{3}}}{\longrightarrow}\coprod
\overset
	{\widehat r_{\beta_1}}{\longrightarrow}\cr
& \{ \frac{1}{2}(-\delta-\varepsilon_{1}-\varepsilon_2-\varepsilon_3),\delta, \varepsilon_{3},
\varepsilon_{2}-\varepsilon_{3}\}\overset
	{\widehat r_{2\varepsilon_{3}}}{\longrightarrow}\{ \frac{1}{2}(-\delta-\varepsilon_{1}-\varepsilon_2+\varepsilon_3),\delta, -\varepsilon_{3},
\varepsilon_{2}+\varepsilon_{3}\}\overset
	{\widehat r_{\beta_1}}{\longrightarrow}\cr
&\{ \frac{1}{2}(-\delta-\varepsilon_{1}-\varepsilon_2+\varepsilon_3),\delta, -\varepsilon_{3},
\varepsilon_{2}+\varepsilon_{3}\}\overset
	{\widehat r_{2\varepsilon_{3}}}{\longrightarrow}\{ \frac{1}{2}(-\delta-\varepsilon_{1}-\varepsilon_2-\varepsilon_3),\delta, \varepsilon_{3},
\varepsilon_{2}-\varepsilon_{3}\}\overset
	{\widehat r_{\beta_1}}{\longrightarrow}\cr
&\coprod.
\end{align*}
}}
So  the order of the action of $\widehat r_{\beta_1}\widehat r_{\varepsilon_3}$ on this fundamental system is $6$.
Note that $$\beta_1=\alpha_1=\frac{1}{2}(\delta+\varepsilon_{1}+\varepsilon_2+\varepsilon_3).$$
We finally have  that  the order of $\widehat r_{\frac{1}{2}(\delta+\varepsilon_{1}+\varepsilon_2+\varepsilon_3)}\cdot \widehat r_{\varepsilon_3}$ is $12$.

According the analysis in the beginning, the proposal  does make sense.  }	
\end{expl}


\subsection{Root system of $G(3)$} {
Let $\ggg=\ggg_\bz\oplus \ggg_\bo$ be the exceptional simple Lie superalgebra $G(3)$. We have $\ggg_\bz\cong G_2\oplus\mathfrak{sl}(2)$, and $\ggg_\bo\cong \bbc^7\oplus \bbc^2$ as an adjoint $\ggg_\bz$-module. Here $\bbc^7$ denotes the $7$-dimensional irreducible $G_2$-module, and $\bbc^2$ the natural $\mathfrak{sl}(2)$-module (see \cite{K} or \cite{CSW} for more details on $G(3)$).

The root system of $\ggg$ can be described as follows. Fix a Cartan subalgebra  $\hhh\in\ggg_\bz$. Assume $\varepsilon_i\in\hhh^*$ ($i=1,2,3$)
satisfy the linear relation $\varepsilon+\varepsilon_2+\varepsilon_3=0$.
Choose the standard simple system  $\Pi=\{\delta-\varepsilon_1,   \varepsilon_2-\varepsilon_3, -\varepsilon_2, \}$. Then the standard positive roots are $\Phi^+=\Phi^+_0\cup\Phi^+_1$
with
\begin{align}
&\Phi_0^+=\{2\delta, \varepsilon_1, -\varepsilon_2,-\varepsilon_3, \varepsilon_1-\varepsilon_2,\varepsilon_1-\varepsilon_3, \varepsilon_2-\varepsilon_3\},\cr
&\Phi_1^+=\{\delta, \delta\pm \varepsilon_i\mid i=1,2,3\}.
\end{align}

 There is a bilinear form on $\hhh^*$ given by
 $$ (\delta,\delta)=2, (\varepsilon_{i},\varepsilon_{i})=-2, (\varepsilon_{i},\varepsilon_{j})=1, (\delta,\varepsilon_{i})=0, i,j =1,2,3, i\neq j.$$

 The extended standard fundamental system for $G(3)$ is $$\widetilde \Pi=\{-2\delta\}\cup \{\delta-\varepsilon_1,   \varepsilon_2-\varepsilon_3,  -\varepsilon_2\}.$$
 }

\subsection{Coxeter graph}

\begin{proposal}\label{thm: excep g}
	The Coxeter graph of $\wsc(G(3))$  is	
	
	
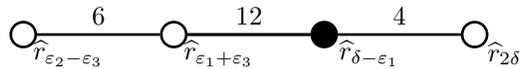
\begin{center}
		\begin{tikzpicture}
			
			\node[circle,draw,minimum size=8pt,inner sep=0pt,fill=white] (r1) at (-4,1) {};

			\node[circle,draw,minimum size=8pt,inner sep=0pt,fill=white] (r2) at (-3,1) {};
			
			\node[circle,draw,minimum size=8pt,inner sep=0pt,fill=white] (r3) at (-2,1) {};
			\draw (r3.135) -- (r3.315);
			\draw(r3.225) -- (r3.45);
			
			\node[circle,draw,minimum size=8pt,inner sep=0pt,fill=white] (r4) at (-1,1) {};

			\draw (r1) -- node[above=2pt] {\footnotesize 6} (r2);
			\draw (r2) -- node[above=2pt] {\footnotesize 12}(r3);
			\draw (r3) --  node[above=2pt] {\footnotesize 4}(r4);

			\node[below=3pt] at (r1) {$ \widehat r_{\varepsilon_2-\varepsilon_3}$};
			\node[below=3pt] at (r2) {$\widehat r_{\varepsilon_1+\varepsilon_3}$};
			\node[below=3pt] at (r3) {$\widehat r_{\delta-\varepsilon_{1}}$};
			\node[below=3pt] at (r4) {$\widehat r_{2\delta}$};

		\end{tikzpicture}
	\end{center}

\end{proposal}

\begin{expl}
	
{
Note that the Coxeter graph of  the subsystem $(\scrw, \spo)$ is already known. It suffices for us to compute $m_{xy}$ involving $\widehat r_{\delta - \varepsilon_1}$ in \S\ref{sec: cox sys} associated with Proposition \ref{prop: gen system}.} So the orders of $ \widehat r_{\varepsilon_2-\varepsilon_3}\cdot\widehat r_{\delta - \varepsilon_1}$, $\widehat r_{\varepsilon_1+\varepsilon_3}\cdot\widehat r_{\delta - \varepsilon_1}$ and $ \widehat r_{2\delta}\cdot\widehat r_{\delta - \varepsilon_1}$ should be given. Here we give the calculation of  $\widehat r_{\varepsilon_1+\varepsilon_3}\cdot\widehat r_{\delta - \varepsilon_1}$. { The other cases are  similar.}
	
	For	$\widehat r_{\varepsilon_1+\varepsilon_3}\cdot\widehat r_{\delta - \varepsilon_1}$, we need to calculate the actions of $\widehat r_{\varepsilon_1+\varepsilon_3}\cdot\widehat r_{\delta - \varepsilon_1}$ on all the  fundamental systems, and get the corresponding orders of them.  The least common multiple of these  orders is $12$. So the order of $\widehat r_{\varepsilon_1+\varepsilon_3}\cdot\widehat r_{\delta - \varepsilon_1}$ is $12$. Unfortunately, we can only give the calculations of the following  fundamental system.

For $\prod:=\{\delta-\varepsilon_{1}, -\varepsilon_{2},\varepsilon_{2}-\varepsilon_{3} \}$, we have
{
\begin{align*}
	&\prod
	\overset{\widehat r_{\delta-\varepsilon_{1}}}{\longrightarrow}
	\{-\delta+\varepsilon_{1}, \delta+\varepsilon_{3},\varepsilon_{2}-\varepsilon_{3} \}
	\overset{\widehat r_{\varepsilon_{1}+\varepsilon_{3}}}{\longrightarrow}\cr
&	\{-\delta-\varepsilon_{3}, \delta-\varepsilon_{1},\varepsilon_{1}-\varepsilon_{2} \}
	\overset{\widehat r_{\delta-\varepsilon_{1}}}{\longrightarrow}
	\{\varepsilon_{2}, -\delta+\varepsilon_{1},\delta-\varepsilon_{2} \}\overset{\widehat r_{\varepsilon_{1}+\varepsilon_{3}}}{\longrightarrow}\cr
&\{\varepsilon_{1}+\varepsilon_{3}, -\delta-\varepsilon_{3},\delta+\varepsilon_{2}\}
	\overset{\widehat r_{\delta-\varepsilon_{1}}}{\longrightarrow}
	\{\varepsilon_{1}+\varepsilon_{3}, -\delta-\varepsilon_{3},\delta+\varepsilon_{2}\}
	\overset{\widehat r_{\varepsilon_{1}+\varepsilon_{3}}}{\longrightarrow}\cr
&\{\varepsilon_{2}, -\delta+\varepsilon_{1},\delta-\varepsilon_{2} \}
	\overset{\widehat r_{\delta-\varepsilon_{1}}}{\longrightarrow}
	\{-\delta-\varepsilon_{3}, \delta-\varepsilon_{1},\varepsilon_{1}-\varepsilon_{2} \}
	\overset{\widehat r_{\varepsilon_{1}+\varepsilon_{3}}}{\longrightarrow}\cr
&	\{-\delta+\varepsilon_{1}, \delta+\varepsilon_{3},\varepsilon_{2}-\varepsilon_{3} \}\overset{\widehat r_{\delta-\varepsilon_{1}}}{\longrightarrow}\{\delta-\varepsilon_{1}, -\varepsilon_{2},\varepsilon_{2}-\varepsilon_{3} \}\overset{\widehat r_{\varepsilon_{1}+\varepsilon_{3}}}{\longrightarrow}\cr
&	\{\delta+\varepsilon_{3}, \varepsilon_{2},\varepsilon_{1}-\varepsilon_{2} \}\overset{\widehat r_{\delta-\varepsilon_{1}}}{\longrightarrow}\{\delta+\varepsilon_{3}, \varepsilon_{2},\varepsilon_{1}-\varepsilon_{2} \}\overset{\widehat r_{\varepsilon_{1}+\varepsilon_{3}}}{\longrightarrow}\cr
&\prod.
\end{align*}
}
 So the order of the action of  $\widehat r_{\varepsilon_1+\varepsilon_3}\cdot\widehat r_{\delta - \varepsilon_1}$ on $\prod$  is $6$.

For $\coprod:=\{\delta+\varepsilon_{3}, \varepsilon_{1},\varepsilon_{2}-\varepsilon_{1} \}$, we have
{
\begin{align*}
&\coprod	\overset{\widehat r_{\delta-\varepsilon_{1}}}{\longrightarrow}
	\{\delta+\varepsilon_{3}, \varepsilon_{1},\varepsilon_{2}-\varepsilon_{1} \}
	\overset{\widehat r_{\varepsilon_{1}+\varepsilon_{3}}}{\longrightarrow}\cr
&\{\delta-\varepsilon_{1}, -\varepsilon_{3},\varepsilon_{3}-\varepsilon_{2} \}
	\overset{\widehat r_{\delta-\varepsilon_{1}}}{\longrightarrow}
	\{\varepsilon_{1}-\delta, \delta+\varepsilon_{2},\varepsilon_{3}-\varepsilon_{2} \}\overset{\widehat r_{\varepsilon_{1}+\varepsilon_{3}}}{\longrightarrow}\cr
&\{-\delta+\varepsilon_{3}, \varepsilon_{1}+\varepsilon_{3},\varepsilon_{2}-\varepsilon_{1}\}
	\overset{\widehat r_{\delta-\varepsilon_{1}}}{\longrightarrow}
	\{-\delta+\varepsilon_{3}, \varepsilon_{1}+\varepsilon_{3},\varepsilon_{2}-\varepsilon_{1}\}
	\overset{\widehat r_{\varepsilon_{1}+\varepsilon_{3}}}{\longrightarrow}\cr
&\{\varepsilon_{1}-\delta, \delta+\varepsilon_{2},\varepsilon_{3}-\varepsilon_{2} \}
	\overset{\widehat r_{\delta-\varepsilon_{1}}}{\longrightarrow}
	\{\delta-\varepsilon_{1}, -\varepsilon_{3},\varepsilon_{3}-\varepsilon_{2} \}
	\overset{\widehat r_{\varepsilon_{1}+\varepsilon_{3}}}{\longrightarrow}\cr
&\coprod.
\end{align*}
}
So the order of the action of  $\widehat r_{\varepsilon_1+\varepsilon_3}\cdot\widehat r_{\delta - \varepsilon_1}$
 on $\coprod$ is $4$.

 Thus, the order of $\widehat r_{\varepsilon_1+\varepsilon_3}\cdot\widehat r_{\delta - \varepsilon_1}$ is the least common multiple of $4$ and $6$, which is  equal to $12$.  So the proposal does make sense.
\end{expl}

\section*{Statement}
The authors have no conflict of interest to declare that are relevant to this article.

\bibliographystyle{amsplain}

\end{document}